\documentclass[11pt,a4paper,reqno]{amsart}

\usepackage[a4paper,left=2.0cm,right=2.0cm,top=2.5cm,bottom=2.5cm]{geometry}

\usepackage[english]{babel}
\usepackage[T1]{fontenc}
\usepackage{times}
\usepackage{amsmath}
\usepackage{amssymb}
\usepackage{amsthm}
\usepackage{amsfonts}
\usepackage{amsbsy}
\usepackage{graphicx}
\usepackage{xcolor}
\usepackage{mathtools}
\usepackage{enumitem}
\usepackage{cite}
\usepackage[foot]{amsaddr}
\usepackage{url}
\usepackage[hidelinks]{hyperref}

	\usepackage{tikz,tikz-3dplot,pgfplots,subcaption}
	\pgfplotsset{compat=1.17}
	\usetikzlibrary{arrows.meta,backgrounds}
	\usepgfplotslibrary{patchplots,fillbetween}
	\pgfplotsset{%
		layers/standard/.define layer set={%
			background,axis background,axis grid,axis ticks,axis lines,axis tick labels,pre main,main,axis descriptions,axis foreground%
		}{
			grid style={/pgfplots/on layer=axis grid},%
			tick style={/pgfplots/on layer=axis ticks},%
			axis line style={/pgfplots/on layer=axis lines},%
			label style={/pgfplots/on layer=axis descriptions},%
			legend style={/pgfplots/on layer=axis descriptions},%
			title style={/pgfplots/on layer=axis descriptions},%
			colorbar style={/pgfplots/on layer=axis descriptions},%
			ticklabel style={/pgfplots/on layer=axis tick labels},%
			axis background@ style={/pgfplots/on layer=axis background},%
			3d box foreground style={/pgfplots/on layer=axis foreground},%
		},
	}

	\allowdisplaybreaks
	\usepackage{placeins}

	\newtheorem{theorem}{Theorem}[section]
	\newtheorem{assumption}[theorem]{Assumption}
	\newtheorem{remark}[theorem]{Remark}
	\newtheorem{proposition}[theorem]{Proposition}
	\newtheorem{corollary}[theorem]{Corollary}
	
	\newtheorem{example}[theorem]{Example}
	\newtheorem{definition}[theorem]{Definition}
	
	\newtheorem{alg}[theorem]{Algorithm}
	\newenvironment{algorithm}[1]{\goodbreak \noindent\rule{\linewidth}{0.5pt}\vspace{-9pt}~\begin{alg}[#1]~\vspace{-7pt}~\\
			\rule{\linewidth}{0.5pt}~\\}{\vspace{-5pt}~\\
			\rule{\linewidth}{0.5pt}~\end{alg}}

	\newcommand{\N}{\mathbb{N}}
	\newcommand{\R}{\mathbb{R}}

	\newcommand{\U}{\mathbb{U}}

	\newcommand{\bu}{{\boldsymbol{u}}}
	\newcommand{\bv}{{\boldsymbol{v}}}
	\newcommand{\bz}{{\boldsymbol{z}}}
	\newcommand{\bA}{{\boldsymbol{A}}}
	\newcommand{\bU}{{\boldsymbol{U}}}
	\newcommand{\bk}{{\boldsymbol{k}}}
	\newcommand{\bc}{{\boldsymbol{c}}}
	
	\newcommand{\bx}{{\boldsymbol{x}}}
	\newcommand{\by}{{\boldsymbol{y}}}
	\newcommand{\bI}{{\boldsymbol{I}}}
	\newcommand{\I}{{\mathcal{I}}}
	\newcommand{\bN}{{\boldsymbol{N}}}
	\newcommand{\bs}{{\boldsymbol{s}}}
	\newcommand{\supp}{\operatorname{supp}}

	\definecolor{tuc}{RGB}{0,95,80}
	\definecolor{tucmath}{RGB}{161,11,112}
	\definecolor{tucalert}{RGB}{255,153,0}

\title{ANOVA approximation for high-dimensional control systems}

\author{Melanie Kircheis$^1$}\address{$^1$Faculty of Applied Computer Sciences and Biosciences, University of Applied Sciences Mittweida, Germany\\ Mail: \textsc{kircheis@hs-mittweida.de}}
\author{Daniel Potts$^2$}\address{$^2$Faculty of Mathematics, Chemnitz University of Technology, Germany\\ Mail: \textsc{\{potts,manuel.schaller\}@math.tu-chemnitz.de}}
\author{Manuel Schaller$^2$}
\author{Karl Worthmann$^3$}\address{$^3$Optimization-based Control Group, Institute of Mathematics, Technische Universit\"at Ilmenau, Germany\\ Mail: \textsc{karl.worthmann@tu-ilmenau.de}}

\thanks{MK, MS and KW gratefully acknowledge the support of the Deutsche Forschungsgemeinschaft (DFG, German Research Foundation) -- Project-ID 519323897.}

\begin{document}

\begin{abstract}
We study the data-driven approximation of control-affine systems, where the drift and each control vector field exhibit a low-dimensional structure.
While the ANOVA (ANalysis Of VAriance) decomposition is tailored to low-dimensional structures, a straightforward implementation does not exploit control affinity, rendering the approximation computationally expensive. 
We propose a coupled least squares approach for given state-control samples, whose ansatz is affine linear in the control by construction and can be realized by a preprocessing of the ANOVA index sets.
Based on a Marcinkiewicz--Zygmund inequality for the resulting ansatz space, we prove a probabilistic $L_2$-error bound with explicitly given constants.
Separate error control on drift and input vector fields renders the approach flexible to, e.\,g., parameter changes. 
Finally, we numerically illustrate the identification of the coupling structure via global sensitivity indices on discretizations of a nonlinear advection-diffusion equation, the shallow water equations and the viscous Burgers equation.
\end{abstract}

\maketitle
\smallskip

\smallskip

		\section{Introduction}

		\noindent Whenever a first-principles model of a dynamical system is unavailable, too complex, or too costly to evaluate, a surrogate model can be learned from measured data. Such a model can then be used for prediction and controller design, e.\,g., within model predictive control~\cite{schimperna2026data}.
		In this work, we consider discrete-time control-affine systems
		\begin{align*}
			\bx^+ = F(\bx) + G(\bx)\bu
		\end{align*}
		with state~$\bx$, successor state~$\bx^+ \in \R^d$, and control input~$\bu\in\R^m$. Such control-affine structures naturally arise, e.\,g., from (time-discretized) mechanical, chemical, or fluid-dynamical models as well as from spatial semi-discretizations of controlled partial differential equations~\cite{BuLe05,HiPiUlUl09}.
		Our approach relies on two structural properties. First, the dynamics are affine linear in the control. Second, as is the case in many applications, each component of~$F$ and~$G$ depends on a few state variables only and couples them through low-order interactions. For instance, in finite-difference or finite-volume discretizations of PDEs, each component couples only with its spatial neighbours.
		Identifying this coupling structure from data reduces the computational effort, and the learned model becomes interpretable and amenable to structure-exploiting controller design.

		The ANOVA (ANalysis Of VAriance) decomposition is tailored to such low-dimensional structures, since it splits high-dimensional functions into terms depending on subsets of the variables only and quantifies the importance of each term by global sensitivity indices.
		For control systems, however, the ANOVA approximation is not directly applicable. On the one hand, $F$ and~$G$ can, in general, not be sampled separately. On the other hand, a naive ANOVA approximation of the joint map~$(\bx,\bu)\mapsto H(\bx,\bu) = F(\bx) + G(\bx)\bu$ would spend considerable
		computational effort on interactions of the control variables which are excluded a priori by the control-affine structure.
		In this paper, we propose an approach which takes the control-affine structure into account and works with given samples~$(\bx^i,\bu^i,H(\bx^i,\bu^i))$, i.\,e., it does not require particular control values to be chosen by the user.	
		Rather, it encodes the control-affine structure into the approximation ansatz itself. To this end, the ANOVA index sets and the basis in the control variables are chosen a priori. This preprocessing step ensures that the ansatz is affine linear in~$\bu$, so that spurious control interactions cannot arise by construction.

		We briefly comment on related work.
		Data-driven identification of dynamical systems under structural assumptions has been studied, for example, using sparse regression approaches such as SINDy~\cite{BrPrKu16} and its extension to control systems~\cite{BrPrKu16c}. In these methods, sparsity is enforced on the coefficient vector over a fixed dictionary. In contrast, our approach identifies sparsity in the \emph{variable couplings}, i.\,e., in the ANOVA terms.
		Koopman-based methods, in particular extended dynamic mode decomposition~\cite{WiKeRo15}, likewise exploit the control-affine structure by means of bilinear surrogate models. This is possible because the Koopman \emph{generator} inherits the control affinity of the dynamics~\cite{PeOtRo20}. For an error analysis of kernel-based EDMD we refer the reader to~\cite{KohnPhil24} and, for an extension to control systems, to~\cite{bold2025kernel}.
		Finally, this work builds upon the ANOVA and high-dimensional model representation (HDMR) literature on sensitivity analysis~\cite{RaAl99,So01} and on the ANOVA approximation approach~\cite{PottSchm21,SchmischkeDiss}.
		To the best of our knowledge, a combination of the ANOVA approximation with the control-affine structure as proposed here has not been considered so far.

		\medskip

		\noindent\textbf{Contribution and outline.}
		In the preparatory Section~\ref{sec:ANOVA}, we briefly recall the ANOVA decomposition, its data-driven implementation and the corresponding error analysis for the Chebyshev basis.
		The contribution of this work is Section~\ref{sec:control}, where we make this machinery applicable to control-affine systems by means of a coupled least squares approach for given state-control samples.
		In Section~\ref{sec:linearity}, we characterize its limit for infinitely many samples as an orthogonal projection in Proposition~\ref{Prop:limit_projection}.
		In Section~\ref{sec:error_coupled}, we establish a Marcinkiewicz--Zygmund inequality for the resulting ansatz space in Proposition~\ref{Prop:MZ}, i.\,e., a norm equivalence between the discrete and the continuous norm which holds uniformly on the ansatz space with high probability. Based on it, we prove a probabilistic $L_2$-error bound with explicit constants in Theorem~\ref{Thm:err_coupled_MZ}, which also yields error bounds for the drift and the input map separately, see Remark~\ref{Remark:no_separation}. Section~\ref{sec:2d} extends these results to multivariate controls, see Corollary~\ref{Cor:err_est_fullanova_m}.
		In Section~\ref{sec:numerics}, we illustrate the approach numerically, before we conclude in Section~\ref{sec:conclusions}.
		Appendix~\ref{sec:complete_func} describes how the coupled approach may be implemented within existing ANOVA software by a suitable preprocessing of the index sets, see Algorithm~\ref{alg:preprocessed}.

		\section{The ANOVA decomposition of scalar-valued functions\label{sec:ANOVA}}

		In this part, we recall the ANOVA decomposition and refer the reader to~\cite{CaMoOw97, RaAl99, LiOw06, KuoSloan10} for a detailed exposition. 
		Let~$\bv \subseteq [d] \coloneqq \{1,\ldots,d\}$ denote a subset of indices.
		For high-dimensional functions~$f\colon \R^d\to\R$, we
		utilize additive decompositions of the form
		\begin{align}
			\label{eq:anova}
			f(\bx) = f_{\varnothing} + \sum_{j=1}^d f_{\{j\}}(x_j) + \sum_{1\leq i < j\leq d} f_{\{i,j\}}(x_i,x_j) + \dots + f_{[d]}(x_1,\dots,x_d) \eqqcolon \sum_{\bv\subseteq [d]} f_\bv(\bx_\bv),
		\end{align}
		where the term~$f_\bv$ corresponding to~$\bv$ only depends on~$x_j$ with~$j\in \bv$, or, equivalently, on~$\bx_\bv \coloneqq (x_j)_{j \in \bv}$ using the inherent ordering of~$[d]$.
		Clearly, every function~$f\colon \R^d\to\R$ may be decomposed in such a way by summing over $2^d$~terms with~$\bv$ in the power set of~$[d]$.
		In this work, however, we are interested in the case where the sum above contains significantly fewer summands and the cardinality~$|\bv|$ of the involved index sets~$\bv$ is small in comparison to the dimension~$d$.
		Such functions are said to have a \emph{low-dimensional structure}.

		\begin{example}\label{ex:einfach}
			For~$d=4$ consider the function
			\begin{align}
				\label{eq:func_example}
				f(\bx) = 1 + x_1^2 + x_1 x_2^4 + x_3 x_4 - \sin(x_4).
			\end{align}
			Clearly, this function can be represented using the five summands
			\begin{align}
				\label{eq:terms_example}
				f_{\varnothing}= 1, \quad\! f_{\{1\}}(x_1) = x_1^2, \quad f_{\{1,2\}}(x_1,x_2) = x_1 x_2^4, \quad f_{\{3,4\}}(x_3,x_4) = x_3 x_4, \quad f_{\{4\}}(x_4) = -\sin(x_4).
			\end{align}
			This yields significantly less summands than the number~$2^4 = 16$ of possible terms, which corresponds to the cardinality of the power set of~$[4] = \{1,2,3,4\}$, since each subset~$\bv$ may contribute a term~$f_\bv$ in~\eqref{eq:anova}.
			Moreover, in \eqref{eq:func_example}, each non-trivial term~$f_\bv$, $\bv \subseteq [d]$, depends on two arguments at most.
		\end{example}

		The representation~\eqref{eq:anova} is not unique in general.
		It becomes unique if one additionally requires the terms~$f_\bv$ to have mean zero and to be pairwise $L_2$-orthogonal, see, e.\,g.,~\cite[Theorem~2.1 and Example~2.2]{KuoSloan10}. This is the ANOVA decomposition defined below.
		These conditions characterize the~$f_\bv$ through $L_2$ inner products over the domain~$\Omega \subseteq \mathbb{R}^d$, i.\,e., through integrals that are however inaccessible when~$f$ is only known through pointwise samples.
		In Section~\ref{sec:ANOVA_implementation}, we therefore replace them by a data-driven least squares approximation over a subspace~$\operatorname{span}\{\varphi_1,\ldots,\varphi_N\}$ spanned by linearly independent functions.

		\subsection{Definition}

		In this part, we gather the notational conventions needed for the definition of the ANOVA decomposition.
		For a bounded set~$\Omega \subset \R^d$ and a probability density~$\omega \colon \Omega\to(0,\infty)$ we write~$L_2(\Omega,\omega)$ for the weighted Lebesgue space of real-valued functions with the inner product~$\langle f,g \rangle_{L_2(\Omega,\omega)} \coloneqq \int_\Omega f g \,\omega \,\mathrm d\bx$ and the induced norm.
		We consider tensor product domains and densities
		\begin{align}
			\label{eq:tensor_density}
			\Omega \coloneqq \bigtimes_{j=1}^d \Omega_j, \quad  \omega(\bx) \coloneqq \prod_{j=1}^d \omega_j(x_j), \quad \int_{\Omega_j} \omega_j(x_j)\,\mathrm dx_j = 1, \quad j=1,\dots,d,
		\end{align}
		and write~$\Omega_\bv \coloneqq \bigtimes_{j\in\bv} \Omega_j$ and~$\omega_\bv(\bx_\bv) \coloneqq \prod_{j\in\bv} \omega_j(x_j)$ for~$\bv\subseteq[d]$.
		By separability, $L_2(\Omega,\omega)$ has an orthonormal tensor product basis~$\varphi_{\bk} \coloneqq \prod_{j=1}^d \varphi_{k_j}(x_j)$, $\bk\in\N_0^d$, built from orthonormal bases of the~$L_2(\Omega_j,\omega_j)$. Hence, every~$f\in L_2(\Omega,\omega)$ admits the expansion
		\begin{align}
			\label{eq:basis_expansion_f}
			f(\bx) = \sum_{\bk\in\N_0^d} c_\bk(f) \,\varphi_\bk(\bx), \qquad
			c_\bk(f) \coloneqq \langle f,\varphi_\bk \rangle_{L_2(\Omega,\omega)},
		\end{align}
		where we may assume~$\varphi_{\boldsymbol 0} \equiv 1$ without loss of generality.
		In the error analysis and in the numerical experiments, we work with the \emph{Chebyshev basis}\footnote{A general interval~$[a,b]$ is mapped to~$[-1,1]$ via~$x\mapsto\frac{2(x-a)}{b-a}-1$, so restricting to~$\Omega_j=[-1,1]$ in~\eqref{eq:basis_cheb} is no loss of generality.}
		\begin{align}
			\label{eq:basis_cheb}
			\begin{split}
				\varphi^{\mathrm{cheb}}_{k_j}(x_j) \coloneqq \sqrt{2}^{\min\{1,k_j\}} \cos(k_j\arccos(x_j)), \ k_j\in\N_0, \\\quad \text{with } \ \Omega_j = [-1,1], \ \omega^{\mathrm{cheb}}_j(x_j) = \frac{1}{\pi\textstyle\sqrt{\smash[b]{1-x_j^2}}}, \quad j\in[d]. 
			\end{split}
		\end{align}
		With these conventions in place, we can now state the ANOVA decomposition. We follow~\cite{KuoSloan10}, where it is derived together with the properties collected below.

		\begin{definition}\label{Def:anova}
			For a function~$f\in L_2(\Omega,\omega)$ and a subset of indices~$\bv \subseteq [d]$, the \emph{ANOVA (ANalysis Of VAriance) terms} are recursively defined by
			\begin{align}
				\label{eq:anova_terms}
				f_{\varnothing} \coloneqq \int_{\Omega} f(\bx) \,\omega(\bx) \,\mathrm{d}\bx, \qquad
				f_{\bv} (\bx_\bv) \coloneqq \int_{\Omega_{\bv^c}} f(\bx) \,\omega_{\bv^c}(\bx_{\bv^c}) \,\mathrm{d}\bx_{\bv^c} - \sum_{\bz\subsetneq \bv} f_{\bz}(\bx_\bz),
			\end{align}
			where the complement of~$\bv$ is defined as~$\bv^c \coloneqq [d]\setminus\bv$.
			The corresponding expansion of the form~\eqref{eq:anova} is called the \emph{ANOVA decomposition} and it fulfills the zero-mean property
			\begin{align}
				\label{eq:zero_mean} 
				\int_{\Omega_\bv} f_{\bv} (\bx_\bv) \,\omega_{\bv} (\bx_\bv) \,\mathrm{d}\bx_\bv = 0, \quad \bv\neq\varnothing,
			\end{align}
			the annihilation property
			\begin{align*}
				\int_{\Omega_j} f_{\bv} (\bx_\bv) \,\omega_j(x_j) \,\mathrm{d}x_j = 0, \quad j\in\bv\neq\varnothing,
			\end{align*}
			and thus $L_2(\Omega,\omega)$-orthogonality
			\begin{align}
				\label{eq:orthogonality}
				\langle f_{\bv}, f_{\bz}\rangle_{L_2(\Omega,\omega)} = 0, \quad \bv\neq\bz\subseteq[d].
			\end{align}
		\end{definition}

		The following example illustrates that the ANOVA terms need not coincide with the naive representation of a function.
		\begin{example}\label{ex:analytic}
			We illustrate Definition~\ref{Def:anova} by computing the ANOVA terms of the function~\eqref{eq:func_example} from Example~\ref{ex:einfach} in the weighted Chebyshev space~$L_2([-1,1]^4,\omega^{\mathrm{cheb}})$.
			Evaluating~\eqref{eq:anova_terms} yields the constant term~$f_{\varnothing} = \frac 32$, the one-dimensional terms
			\begin{align*}
				f_{{\{1\}}}(x_1) = x_1^2+\tfrac{3}{8}x_1-\tfrac{1}{2},
				\qquad f_{{\{4\}}}(x_4) = -\sin(x_4),
				\qquad f_{{\{2\}}} = f_{{\{3\}}} = 0,
			\end{align*}
			and the two-dimensional terms
			\begin{align*}
				f_{{\{1,2\}}}(x_1,x_2) = x_1x_2^4-\tfrac{3}{8}x_1,
				\qquad f_{{\{3,4\}}}(x_3,x_4) = x_3x_4.
			\end{align*}
			All remaining pair terms as well as all terms of order three and four vanish.
			Hence, the decomposition is of the expected form
			\begin{align}
				\label{eq:structure_example}
				f(\bx) = f_{\varnothing} + f_{\{1\}}(x_1) + f_{\{1,2\}}(x_1,x_2) + f_{\{3,4\}}(x_3,x_4) + f_{\{4\}}(x_4).
			\end{align}
			Note, however, that the ANOVA terms do not coincide with the representation~\eqref{eq:terms_example}. The zero-mean and orthogonality constraints enforce additional lower-order corrections such as the summand~$\frac38 x_1$ appearing in~$f_{\{1\}}$ and~$f_{\{1,2\}}$.
			Thus, even if a function has an obvious low-dimensional structure, the number of ANOVA terms required to represent it is not known a priori.
		\end{example}

		Using the basis expansion~\eqref{eq:basis_expansion_f}, the ANOVA terms~\eqref{eq:anova_terms} may also be written as
		\begin{align}
			\label{eq:basis_expansion_fv}
			f_\bv(\bx_\bv) = \sum_{\substack{\bk\in\N_0^d\\\supp\bk=\bv}} c_\bk(f) \,\varphi_\bk(\bx_\bv),
			\quad \bv\subseteq[d],
		\end{align}
		where~$\supp\bk \coloneqq \{j\in[d] \colon k_j \neq 0\}$, cf.~\cite{PottSchm19a}.
		Note that for~$\supp\bk=\bv$ the tensor-product basis function~\mbox{$\varphi_\bk(\bx)=\prod_{j=1}^d\varphi_{k_j}(x_j)$} satisfies~$\varphi_{k_j}\equiv 1$ for all~$j\notin\bv$, since~$\varphi_0\equiv 1$. Therefore, $\varphi_\bk$ depends only on~$\bx_\bv$, which justifies writing~$\varphi_\bk(\bx_\bv)$. The same convention is used in Appendix~\ref{sec:complete_func} for indices with a non-zero control component.
		In other words, the representation~\eqref{eq:basis_expansion_fv} exhibits the ANOVA decomposition as a partition of the space of basis coefficients~$c_\bk(f)$.
		Therein, the term for~$\bv=\varnothing$ collects the single coefficient~$c_{\boldsymbol 0}(f)$, the term for a singleton~$\bv=\{j\}$ collects all coefficients on the respective axis except the $0$-th one, and so on. Figure~\ref{fig:anova_decomposition} illustrates this partition for the index set~$\N_0^d$ underlying the bases considered here.
		Every coefficient of~\eqref{eq:basis_expansion_f} is thus attributed to exactly one ANOVA term.
		\begin{figure}[htb]
			\centering
\tdplotsetmaincoords{70}{120}
\begin{minipage}[b]{0.16\textwidth}
    \centering
    \begin{tikzpicture}[scale=1,tdplot_main_coords]
        \draw[-,opacity = 0.4] (0,0,0) -- (1,0,0);
        \draw[-,opacity = 0.4] (0,0,0) -- (0,1,0);
        \draw[-,opacity = 0.4] (0,0,0) -- (0,0,1);
        \pgfmathsetmacro\a{3}
        \foreach \x in {0,...,\a}{
            \foreach \y in {0,...,\a}{
                \foreach \z in {0,...,\a}{
                \pgfmathsetmacro\yi{\y/(\a+1)}
                \pgfmathsetmacro\xi{\x/(\a+1)}
                \pgfmathsetmacro\zi{\z/(\a+1)}
                \shade[ball color = tucmath!70, opacity = 0.8] (\xi,\yi,\zi) circle (2pt);
        }}}
    \end{tikzpicture}
    $$\bk \in \N_0^3$$
\end{minipage}%
\begin{minipage}[t]{0.032\textwidth}
    \vspace*{-44pt}
    $=$
\end{minipage}%
\begin{minipage}[b]{0.16\textwidth}
    \centering
    \begin{tikzpicture}[scale=1,tdplot_main_coords]
        \draw[-,opacity = 0.4] (0,0,0) -- (1,0,0);
        \draw[-,opacity = 0.4] (0,0,0) -- (0,1,0);
        \draw[-,opacity = 0.4] (0,0,0) -- (0,0,1);
        \shade[ball color = tuc!40, opacity = 0.8] (0,0,0) circle (2pt);
    \end{tikzpicture}\\
    \vspace*{-10pt}%
    $$\bk = \boldsymbol 0 $$
\end{minipage}%
\begin{minipage}[t]{0.032\textwidth}
    \vspace*{-44pt}
    $+$
\end{minipage}%
\begin{minipage}[b]{0.16\textwidth}
    \centering
    \begin{tikzpicture}[scale=1,tdplot_main_coords]
        \draw[-,opacity = 0.4] (0,0,0) -- (1,0,0);
        \draw[-,opacity = 0.4] (0,0,0) -- (0,1,0);
        \draw[-,opacity = 0.4] (0,0,0) -- (0,0,1);
        \pgfmathsetmacro\a{3}
        \foreach \y in {1,...,\a}{
            \pgfmathsetmacro\yi{\y/(\a+1)}
            \shade[ball color = tuc!70, opacity = 0.8] (0,\yi,0) circle (2pt);
            \shade[ball color = tucmath!70, opacity = 0.8] (\yi,0,0) circle (2pt);
            \shade[ball color = tucalert!70, opacity = 0.8] (0,0,\yi) circle (2pt);
        }
    \end{tikzpicture}\\
    \vspace*{-10pt}%
    $$|\supp \bk|  = 1$$
\end{minipage}%
\begin{minipage}[t]{0.032\textwidth}
    \vspace*{-44pt}
    $+$
\end{minipage}%
\begin{minipage}[b]{0.16\textwidth}
    \centering
    \begin{tikzpicture}[scale=1,tdplot_main_coords]
        \draw[-,opacity = 0.4] (0,0,0) -- (1,0,0);
        \draw[-,opacity = 0.4] (0,0,0) -- (0,1,0);
        \draw[-,opacity = 0.4] (0,0,0) -- (0,0,1);
        \pgfmathsetmacro\a{3}
        \foreach \x in {1,...,\a}{
            \foreach \y in {1,...,\a}{
                \pgfmathsetmacro\yi{\y/(\a+1)}
                \pgfmathsetmacro\xi{\x/(\a+1)}
                \shade[ball color = tuc!70, opacity = 0.8] (\xi,\yi,0) circle (2pt);
                \shade[ball color = tucmath!70, opacity = 0.8] (\xi,0,\yi) circle (2pt);
                \shade[ball color = tucalert!70, opacity = 0.8] (0,\xi,\yi) circle (2pt);
        }}
        \draw[fill=tucalert, semitransparent] (0,0,0) -- (0,1,0) -- (0,1,1) -- (0,0,1)--cycle;
        \draw[fill=tucmath, semitransparent] (0,0,0) -- (1,0,0) -- (1,0,1) -- (0,0,1)--cycle;
        \draw[fill=tuc, semitransparent] (0,0,0) -- (1,0,0) -- (1,1,0) -- (0,1,0)--cycle;
    \end{tikzpicture}
    $$|\supp \bk|  = 2$$
\end{minipage}%
\begin{minipage}[t]{0.032\textwidth}
    \vspace*{-44pt}
    $+$
\end{minipage}%
\begin{minipage}[b]{0.16\textwidth}
    \centering
    \begin{tikzpicture}[scale=1,tdplot_main_coords]
    \draw[-,opacity = 0.4] (0,0,0) -- (1,0,0);
    \draw[-,opacity = 0.4] (0,0,0) -- (0,1,0);
    \draw[-,opacity = 0.4] (0,0,0) -- (0,0,1);
    \pgfmathsetmacro\a{3}
    \foreach \x in {1,...,\a}{
        \foreach \y in {1,...,\a}{
            \foreach \z in {1,...,\a}{
            \pgfmathsetmacro\yi{\y/(\a+1)}
            \pgfmathsetmacro\xi{\x/(\a+1)}
            \pgfmathsetmacro\zi{\z/(\a+1)}
            \shade[ball color = tucmath!70, opacity = 0.8] (\xi,\yi,\zi) circle (2pt);
    }}}
    \end{tikzpicture}
    $$|\supp \bk|  = 3$$
\end{minipage}
			\caption{Visualization of the partition in the space of basis coefficients, as provided by the ANOVA decomposition, for a 3-dimensional function in the setting~$\bk\in\N_0^d$ underlying the Chebyshev basis~\eqref{eq:basis_cheb}.}\label{fig:anova_decomposition}
		\end{figure}
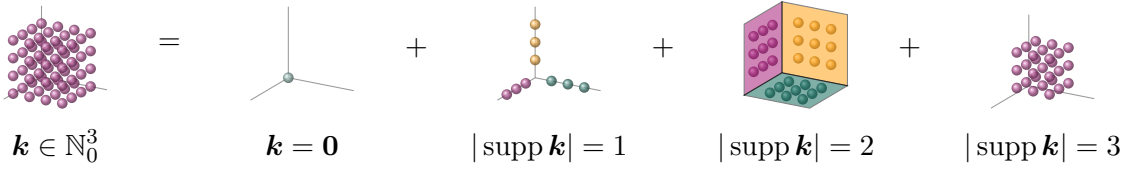

		The ANOVA decomposition also quantifies how much an individual term contributes to~$f$, which will later allow us to discard the irrelevant ones.
		With the variance of~$f$ and of an ANOVA term,
		\begin{align}
			\label{eq:variance_f}
			\sigma^2(f) \coloneqq \|f\|_{L_2(\Omega,\omega)}^2 - f_{\varnothing}^2,
			\qquad
			\sigma^2(f_\bv) \coloneqq \|f_{\bv}\|_{L_2(\Omega_\bv,\omega_\bv)}^2, \quad \bv \subseteq [d],
		\end{align}
		the zero-mean property~\eqref{eq:zero_mean} and orthogonality~\eqref{eq:orthogonality} yield the variance decomposition~$\sigma^2(f) = \sum_{\varnothing\neq\bv \subseteq [d]} \sigma^2(f_\bv)$, which explains the name ``analysis of variance''.
		The share of the variance explained by the interaction of the variables indexed by~$\bv$ is measured by the \emph{global sensitivity indices}, or \emph{Sobol indices},
		\begin{align}
			\label{eq:gsis}
			S_\bv(f) \coloneqq \frac{\sigma^2(f_\bv)}{\sigma^2(f)}
			= \frac{\sum_{\bk\in\N_0^d,\,\supp\bk=\bv} |c_\bk(f)|^2}{\sum_{\bk\in\N_0^d,\,\bk\neq\boldsymbol 0} |c_\bk(f)|^2} \in [0,1], \quad \varnothing\neq\bv \subseteq [d],
		\end{align}
		see~\cite{So01}, where the second expression follows from the basis representations~\eqref{eq:basis_expansion_f} and~\eqref{eq:basis_expansion_fv}.

		\subsection{Efficient and data-driven implementation\label{sec:ANOVA_implementation}}

		Although~\eqref{eq:anova} carries valuable information about a function, its defining relations~\eqref{eq:anova_terms} are not amenable to computation. On the one hand, the number of terms grows exponentially in the dimension~$d$. On the other hand, each term is defined through integrals, which are inaccessible when~$f$ is only known through samples.
		To enable an efficient and purely data-driven implementation, we therefore replace these relations by a chain of approximations, cf.~\cite{PottSchm21}.
		We first \emph{truncate} the decomposition to its relevant low-dimensional terms, then \emph{project} each term onto a finite basis, \emph{estimate} the resulting coefficients from samples, and finally \emph{post-process} the approximation to quantify the importance of the individual ANOVA terms.
		These four steps are detailed in the following.

		\subsubsection{Truncation}
		Since it is unknown a priori which of the~$2^d$ terms in~\eqref{eq:anova} are relevant, we first truncate the decomposition, that is, we remove terms~$f_\bv$.
		To this end, we invoke the sparsity-of-effects principle~\cite{Montgomery17}, according to which many phenomena are well described by a few low-order interactions.
		More precisely, we introduce the \emph{superposition dimension}~$q\in\N$, $q\ll d$, and truncate~\eqref{eq:anova} as
		\begin{align}
			\label{eq:anova_superposition}
			f(\bx) \approx \sum_{|\bv|\leq q} f_\bv(\bx_\bv).
		\end{align}
		In many real-world problems the superposition dimension is indeed low, i.\,e., small choices~$q\in\{2,3\}$ already suffice for a good approximation, see~\cite{CaMoOw97,KuoSloan10,DeVore11,PottSchm19a}.

		Now that we have reduced the number of ANOVA terms to be determined, we turn to their actual computation.
		Both representations~\eqref{eq:anova_terms} and~\eqref{eq:basis_expansion_fv} involve integrals and are thus not directly amenable to computation. We build on the latter and render it computable in two steps, the first of which is the following projection.

		\subsubsection{Projection}
		The representation~\eqref{eq:basis_expansion_fv} of an ANOVA term still involves infinitely many basis coefficients.
		Therefore, we introduce another truncation step, that is, for every subset~$\bv\subseteq[d]$ we define the index set
		\begin{align}
			\label{eq:indexset_Nv}
			\I_{{\bN}^\bv} \coloneqq \left\{ \bk\in\N_0^d \colon \supp\bk=\bv,\, \bk_\bv \in \bigtimes_{j=1}^{|\bv|} \left\{1,\dots,N_j^\bv-1\right\} \right\},
			\qquad
			{\bN}^\bv \coloneqq \left(N_j^\bv\right)_{j=1}^{|\bv|} \in (2\N)^{|\bv|}.
		\end{align}
		In other words, for~$\bv=\varnothing$ we keep the single coefficient~$c_{\boldsymbol 0}(f)$ and for all~$\bv\neq\varnothing$ we consider only the first~\mbox{$N_j^\bv-1$} coefficients of~\eqref{eq:basis_expansion_fv} in every direction~$j$ that is present in~$\bv$.
		For instance, when considering~$\bv=\{1,2\}$ with~\mbox{${\bN}^\bv=(6,4)^\top\in(2\N)^2$} for our 4-dimensional example function~\eqref{eq:func_example}, then the index set~$\I_{{\bN}^\bv}$ contains indices~$\bk=(k_1,k_2,0,0)^\top$ with~$k_1\in\{1,2,\dots,5\}$ and~$k_2\in\{1,2,3\}$.
		Hence, the approximation of the ANOVA terms~\eqref{eq:anova_terms} is given by
		\begin{align*}
			f_\bv(\bx_\bv) \approx \sum_{\bk\in\I_{{\bN}^\bv}} c_\bk(f) \,\varphi_\bk(\bx_\bv),
			\quad \bv\subseteq[d].
		\end{align*}
		Using the (disjoint) union
		\begin{align}
			\label{eq:indexset_N}
			\I_{\bN} \coloneqq \bigcup_{|\bv|\leq q} \I_{{\bN}^\bv}
		\end{align}
		of the index sets~\eqref{eq:indexset_Nv} we approximate the function~$f$ in~\eqref{eq:basis_expansion_f} as
		\begin{align}
			\label{eq:basis_expansion_f_N}
			f(\bx) \approx P_\bN f(\bx) \coloneqq \sum_{\bk\in\I_{\bN}} c_\bk(f) \,\varphi_\bk(\bx),
		\end{align} 
		i.\,e., by the orthogonal projection of~$f$ onto the subspace~$\operatorname{span}\{\varphi_{\bk}\}_{\bk\in\I_{\bN}} \subset L_2(\Omega,\omega)$ employing the finite dictionary~$\mathcal D \coloneqq \{\varphi_{\bk}\}_{\bk\in\I_{\bN}}\subset \{\varphi_{\bk}\}_{\bk\in\N_0^d}$.

		\subsubsection{Estimation}
		The above approximation \eqref{eq:basis_expansion_f_N} of the function~$f$ still requires the integrals~$c_\bk(f) = \langle f,\varphi_\bk \rangle_{L_2(\Omega,\omega)}$, $\bk\in\I_{\bN}$, from~\eqref{eq:basis_expansion_f}, which are inaccessible in a purely data-driven setting.
		We therefore keep the ansatz~\eqref{eq:basis_expansion_f_N} and determine its coefficients from data.
		To this end, assume we are given i.i.d.\ sample points~$\{\bx^1,\ldots,\bx^M\}\subset\Omega$, $M\in\N$, drawn according to the density~$\omega$, as well as the corresponding function evaluations~$\{f(\bx^1),\ldots,f(\bx^M)\}\subset\R$.
		Then we may approximate the basis coefficients~$c_\bk(f)$, $\bk\in\I_{\bN}$, in~\eqref{eq:basis_expansion_f_N} by solving the least squares problem
		\begin{align}
			\label{eq:lsqr_problem}
			\min_{\bc\in\R^{|\I_{\bN}|}} \| \bA \bc - \boldsymbol f \|_2^2,
		\end{align}
		where
		\begin{align*}
			\bA \coloneqq \left(\varphi_\bk(\bx^i)\right)_{i=1,\,\bk\in\I_{\bN}}^{M} \in \R^{M\times |\I_{\bN}|},
			\qquad\text{and}\qquad
			\boldsymbol f \coloneqq \left(f(\bx^i)\right)_{i=1}^M \in \R^M.
		\end{align*}
		We compute the solution of~\eqref{eq:lsqr_problem} via the normal equations of the first kind
		\begin{align}\label{eq:normal}
			\bA^\top \bA \bc = \bA^\top \boldsymbol f.
		\end{align}
		If~$\bA$ has full column rank~$|\I_{\bN}|$, the matrix~$\bA^\top \bA$ is invertible and the unique solution of the least squares problem~\eqref{eq:lsqr_problem} is given by
		\begin{align}
			\label{eq:coeffs_pinv}
			(\tilde c_\bk)_{\bk\in\I_\bN} = \tilde\bc \coloneqq (\bA^\top \bA)^{-1} \bA^\top \boldsymbol f.
		\end{align}
		Full column rank of~$\bA$ requires the overdetermined setting~$M\geq|\I_{\bN}|$. Moreover, if the~$\bx^i$ are drawn i.i.d.\ according to~$\omega$ and the oversampling factor~$M/|\I_{\bN}|$ is large enough, $\bA$ is well-conditioned with high probability, see~\cite[Theorem~4.2]{SchmischkeDiss} and~\cite{KaUlVo19}.
		In the underdetermined case, regularization techniques are available~\cite{SchmischkeDiss}, which we do not pursue here.
		In high dimensions,~\eqref{eq:coeffs_pinv} is computed iteratively by the LSQR algorithm~\cite{PaSa82}, which never forms~$\bA$ explicitly but only requires multiplication with~$\bA$ and~$\bA^\top$. We use the implementation~\cite{ANOVAapprox}, which provides fast algorithms for these products.

		\subsubsection{Post-processing and sensitivities}
		It remains to assess the importance of the individual ANOVA terms.
		We approximate the global sensitivity indices~\eqref{eq:gsis} using the estimated coefficients~\eqref{eq:coeffs_pinv} and the index sets~\eqref{eq:indexset_Nv},~\eqref{eq:indexset_N} as
		\begin{align}
			\label{eq:gsis_approx}
			S_\bv(f)
			\approx
			S_\bv(\tilde f)
			= \frac{\sum_{\bk\in\I_{{\bN}^\bv}} |\tilde c_\bk|^2}{\sum_{\bk\in\I_{\bN}\setminus\{\boldsymbol{0}\}} |\tilde c_\bk|^2}, \quad \varnothing\neq\bv \subseteq [d],
		\end{align}
		with the intermediate approximation~$\tilde f$ built from the~$\tilde c_\bk$ of~\eqref{eq:coeffs_pinv}.
		This is accurate whenever~$f_\bv \approx \tilde f_\bv$, since~$\big| \sigma^2(f_\bv)-\sigma^2(\tilde f_\bv) \big| \leq \|f_\bv-\tilde f_\bv\| \,\big(2\|f_\bv\|+\|f_\bv-\tilde f_\bv\|\big)$ in~$L_2(\Omega_\bv,\omega_\bv)$, cf.~\cite[(5.3)]{PottWeid25}.
		Given a threshold~$\varepsilon\in(0,1)$, the relevant ANOVA terms are collected in
		\begin{align}
			\label{eq:indexset_V}
			V \coloneqq \{\varnothing\} \cup \left\{ \bv \subseteq [d] \colon 0<|\bv|\leq q,\ S_\bv(\tilde f) > \varepsilon \right\},
		\end{align}
		where~$\bv=\varnothing$ is added manually, since~\eqref{eq:gsis} carries no information about the constant term.
		Restricting~\eqref{eq:indexset_N} to these terms gives
		\begin{align}
			\label{eq:indexset_N_truncated}
			\I(V) \coloneqq \bigcup_{\bv\in V} \I_{{\bN}^\bv},
		\end{align}
		and we solve the correspondingly truncated least squares problem
		\begin{align}
			\label{eq:lsqr_problem_truncated}
			\min_{\bc\in\R^{|\I(V)|}} \| \bA_V \bc - \boldsymbol f \|_2^2
		\end{align}
		with the truncated matrix~$\bA_V \coloneqq \left(\varphi_\bk(\bx^i)\right)_{i=1,\,\bk\in \I(V)}^{M} \in \R^{M\times|\I(V)|}$
		to improve the fit of~\eqref{eq:coeffs_pinv}, and set
		\begin{align}\label{eq:discproj}
			f^\star(\bx) \coloneqq \sum_{\bk\in\I(V)} c^\star_\bk \varphi_\bk(\bx)
		\end{align}
		as our approximation of~$f$.
		The resulting workflow is summarized as follows.

		\begin{algorithm}{ANOVA approximation, cf.~\cite{PottSchm19a, PottSchm21, ANOVAapprox, pyANOVAapprox}}\label{alg:anova}
			\vspace{-1.5em}
			\begin{tabbing}
				\textbf{Input:} \quad \= $d\in\N$ \hspace{10em} \= dimension \\
				\> $q\in\N$, $q\ll d$ \> superposition dimension \\
				\> $\{\bx^1,\ldots,\bx^M\}\subset\Omega$, $M\in\N$ \> i.i.d.\ sample points drawn according to the density~$\omega$ \\
				\> $\boldsymbol f = \left(f(\bx^i)\right)_{i=1}^M$ \> function evaluations at the sample points \\
				\> $\varepsilon\in(0,1)$ \> threshold parameter \\
				\> $r>0$ \> confidence parameter, cf.\ Theorem~\ref{Thm:ANOVA_cheb}
			\end{tabbing}
			\begin{enumerate}
				\item For all subsets with~$|\bv|\leq q$ choose~${\bN}^\bv = \left(N_j^\bv\right)_{j=1}^{|\bv|} \in (2\N)^{|\bv|}$ in~\eqref{eq:indexset_Nv} such that the index set~$\I_\bN$ in~\eqref{eq:indexset_N} satisfies~$|\I_{\bN}| \leq \frac{1}{10\cdot2^{\alpha}(r+1)} \cdot \frac{M}{\log M}$ with the given~$r>0$, where~$\alpha\coloneqq\max_{\bk\in\I_\bN} |\supp{\bk}|\leq q$.
				\item Compute the estimated coefficients~$\tilde\bc = (\tilde c_\bk)_{\bk\in\I_\bN}$ as solution to the least squares problem~\eqref{eq:lsqr_problem} using the LSQR algorithm.
				\item Compute the approximate global sensitivity indices
				\begin{align*}
					S_\bv(\tilde f)
					= \frac{\sum_{\bk\in\I_{{\bN}^\bv}} |\tilde c_\bk|^2}{\sum_{\bk\in\I_{\bN}\setminus\{\boldsymbol{0}\}} |\tilde c_\bk|^2}, \quad \varnothing\neq\bv \subseteq [d],
				\end{align*}
				see~\eqref{eq:gsis_approx}.
				\item Determine the set
				\begin{align*}
					V = \{\varnothing\} \cup \left\{ \bv \subseteq [d] \colon 0<|\bv|\leq q,\ S_\bv(\tilde f) > \varepsilon \right\}
				\end{align*}
				of relevant ANOVA terms, see~\eqref{eq:indexset_V}, and set~$\I(V)=\bigcup_{\bv\in V} \I_{{\bN}^\bv}$, see~\eqref{eq:indexset_N_truncated}.
				\item Compute the coefficients~$\bc^\star = (c^\star_\bk)_{\bk\in\I(V)}$ as solution to the truncated least squares problem~\eqref{eq:lsqr_problem_truncated} using the LSQR algorithm.
			\end{enumerate}
			\textbf{Output:} coefficients~$\left(c^\star_\bk\right)_{\bk\in\I(V)}\in\R^{|\I(V)|}$ of the approximation~$f^\star = \sum_{\bk\in\I(V)} c^\star_\bk \varphi_\bk \approx f$, see~\eqref{eq:discproj}
		\end{algorithm}

		Algorithm~\ref{alg:anova} combines the truncation, projection, estimation, and post-processing steps introduced above into a single data-driven ANOVA approximation.
		Before analyzing its approximation error, we comment on two ways in which the procedure can be refined in practice.

		\begin{remark}[Possible improvements of the algorithm]
			Note that in steps 4 and 5 of Algorithm~\ref{alg:anova}, the index sets~\eqref{eq:indexset_Nv} may be further adjusted.
			Since the truncation by~$V$ in~\eqref{eq:indexset_V} typically renders~$\I(V)$ in~\eqref{eq:indexset_N_truncated} considerably smaller than~$\I_\bN$ in~\eqref{eq:indexset_N}, the condition~$M\geq|\I(V)|$ now suffices.
			This allows for larger choices of the corresponding bandwidth parameters~$\bN^\bv$ for the least squares problem~\eqref{eq:lsqr_problem_truncated}.
			In doing so, the additional basis coefficients are spent on the relevant terms~$\bv\in V$ only.
			Note, however, that~$M\geq|\I(V)|$ only guarantees that the truncated least squares problem is not underdetermined. The error bound of Theorem~\ref{Thm:ANOVA_cheb} below requires the strictly stronger sampling condition~$|\I(V)| \leq \frac{1}{10\cdot2^{\alpha}(r+1)}\cdot\frac{M}{\log M}$. If the bandwidths are enlarged beyond this condition, the resulting approximation~$f^\star$ is still computable, but it is no longer covered by the theory.

			Moreover, it is also possible to reiterate steps 3--5 of Algorithm~\ref{alg:anova} until no further truncation of the ANOVA decomposition can be achieved, i.\,e., until the set~$V$ in~\eqref{eq:indexset_V} no longer changes.
		\end{remark}

		\subsection{Error Analysis}

		In this part, we recall the approximation error of the ANOVA approximation, restricting our discussion to the Chebyshev basis~\eqref{eq:basis_cheb} for the sake of clarity of presentation.
		In this case, we define the Sobolev-type space of dominating mixed smoothness
		\begin{align*}
			H^s_{\mathrm{mix}}([-1,1]^d) \coloneqq \left\{ f\colon[-1,1]^d\to\R \;\Big|\; \|f\|_{H^s_{\mathrm{mix}}([-1,1]^d)}^2<\infty \right\},
			\quad s>0
		\end{align*}
		with the norm
		\begin{align}
			\label{eq:Hsmix_norm}
			\|f\|_{H^s_{\mathrm{mix}}([-1,1]^d)}^2 \coloneqq \sum_{\bk\in\N_0^d} |c_\bk(f)|^2 \prod_{j=1}^d \left(1+|k_j|^2\right)^s,
		\end{align}
		where the coefficients~$c_\bk(f)$ are given as in~\eqref{eq:basis_expansion_f}, and we refer the reader to~\cite[Corollary~2.2]{BaLuNaUl23} and~\cite[p.~9]{EgMiUl25} for details.

		The following result, which we recall from~\cite{SchmischkeDiss}, yields a worst-case bound for the error of the ANOVA approximation in Algorithm~\ref{alg:anova}.

		\begin{theorem}[\!\!{\cite[Theorem~4.10]{SchmischkeDiss}}]
			\label{Thm:ANOVA_cheb}
			Let~$M>3$ sample points drawn i.i.d.\ w.r.t.\ the Chebyshev density~$\omega^{\mathrm{cheb}}$ in~\eqref{eq:basis_cheb} be given such that for the index set~$\I(V)$ in~\eqref{eq:indexset_N_truncated} we have
			\begin{align*}
				|\I(V)| \leq \frac{M}{10\cdot 2^\alpha(r+1)\log M}
			\end{align*}
			for some~$r>0$ and~$\alpha\coloneqq\max_{\bk\in\I(V)} |\supp{\bk}|$.
			Then for all~$f\in H^s_{\mathrm{mix}}([-1,1]^d)$, $s>\frac12$, the solution~$f^\star$ in~\eqref{eq:discproj} of the least squares problem~\eqref{eq:lsqr_problem_truncated} satisfies the estimate
			\begin{align*}
				\|f - f^\star\|_{L_2([-1,1]^d,\omega^{\mathrm{cheb}})}^2
				\leq
				5 \|f\|_{H^s_{\mathrm{mix}}([-1,1]^d)}^2 \cdot \max \left\{
				\frac{2^d(1+\sqrt{5})^2}{2^\alpha\cdot 5|\I(V)|} \sum_{\bk\in\N_0^d\setminus\I(V)} \frac{1}{\beta(\bk)},
				\frac{1}{\inf_{\bk\in\N_0^d\setminus\I(V)} \beta(\bk)}
				\right\}
			\end{align*}
			with probability at least~$1-(2^{3/4}+1)M^{-r}$, where
			\begin{align}
				\label{eq:weight_Hmix}
				\beta(\bk) \coloneqq \prod_{j=1}^d \left(1+|k_j|^2\right)^s
			\end{align}
			denotes the weight function of~$H^s_{\mathrm{mix}}([-1,1]^d)$, cf.~\eqref{eq:Hsmix_norm}.
		\end{theorem}

		The error bound of Theorem~\ref{Thm:ANOVA_cheb} is governed by the chosen index set~$\I(V)$ and holds with a probability depending on the coupled parameters~$r$ and~$M$.
		Before relating it to the individual approximation steps, we point out a limitation in the high-dimensional regime. The first argument of the maximum carries the explicit factor~$2^d$, and the sum~$\sum_{\bk\notin\I(V)}1/\beta(\bk)$ is exponential in~$d$ as well, since~$\sum_{\bk\in\N_0^d}1/\beta(\bk) = C_s^d$ with~\mbox{$C_s\coloneqq\sum_{n\in\N_0}(1+n^2)^{-s}>1$}. The second argument, in contrast, is free of any explicit dimension dependence. Theorem~\ref{Thm:ANOVA_cheb} is therefore informative for moderate~$d$ and large smoothness~$s$ (note~$C_s\to1$ as~$s\to\infty$), while for large~$d$ it is a qualitative rather than a quantitative statement.

		\begin{remark}[{Error terms and choice of the parameters}]
			\label{Remark:lsqr_error}
			Writing~$P_{\I(V)}$ for the orthogonal projection
			\begin{align}\label{eq:projection}
				P_{\I(V)} f(\bx) \coloneqq \sum_{\bk\in\I(V)} \langle f,\varphi_\bk\rangle_{L_2(\Omega,\omega)} \,\varphi_\bk(\bx)
			\end{align}
			onto~$\operatorname{span}\{\varphi_{\bk}\}_{\bk\in\I(V)}$, cf.~\eqref{eq:basis_expansion_f_N}, the error of the approximation~$f^\star$ from~\eqref{eq:discproj} splits as
			\begin{align*}
				\|f^\star - f\|_{L_2(\Omega,\omega)}^2 = \underbrace{\|f^\star - P_{\I(V)} f\|_{L_2(\Omega,\omega)}^2}_{\text{estimation error}} + \underbrace{\|P_{\I(V)} f - f\|_{L_2(\Omega,\omega)}^2}_{\text{truncation and projection error}}.
			\end{align*}
			Therein, equality holds by the theorem of Pythagoras. Both~$f^\star$ and~$P_{\I(V)}f$ lie in~$\operatorname{span}\{\varphi_{\bk}\}_{\bk\in\I(V)}$, whereas \mbox{$f-P_{\I(V)}f$} is orthogonal to that subspace, such that the two summands of~$f^\star-f=(f^\star-P_{\I(V)}f)+(P_{\I(V)}f-f)$ are orthogonal.
					The two approximations differ only in how their coefficients are obtained, using the empirical problem~\eqref{eq:lsqr_problem_truncated} instead of the exact inner products. Hence, the first term measures the error introduced by the regression step, while the second term is the error of the best-approximation in the subspace.

			Both sources appear in Theorem~\ref{Thm:ANOVA_cheb}, though not as a trade-off between the two arguments of the maximum. By~\eqref{eq:Hsmix_norm}, the second argument is the best-approximation bound
			\begin{align*}
				\|f-P_{\I(V)}f\|^2_{L_2(\Omega,\omega)} \leq \|f\|^2_{H^s_{\mathrm{mix}}([-1,1]^d)} \big/ \inf_{\bk\notin\I(V)}\beta(\bk),
			\end{align*}
			whereas the first carries the constants of the sampling step. Both decrease monotonically as~$\I(V)$ grows and tend to zero for~$\I(V)\uparrow\N_0^d$, and neither contains~$M$. The sample size enters only through the admissibility condition~$|\I(V)|\leq M/(10\cdot 2^\alpha(r+1)\log M)$.
			Consequently, enlarging~$\I(V)$ improves the bound, and for a given~$M$ it is optimal to choose~$\I(V)$ as large as that condition permits.
			The projection error itself vanishes for~$\I(V)=\N_0^d$ and can be reduced either by admitting more basis coefficients per ANOVA term, cf.~\eqref{eq:indexset_Nv}, or by admitting more ANOVA terms. For specific~$\I(V)$, the two quantities of Theorem~\ref{Thm:ANOVA_cheb} are evaluated explicitly in~\cite[pp.~126--127]{SchmischkeDiss}.
			The estimation error vanishes as~$M\to\infty$ as due to the normal equations~\eqref{eq:normal} and for i.i.d.\ samples, Monte-Carlo integration gives~$\frac1M\bA_V^\top \bA_V \to \bI$ and~$\frac1M\bA_V^\top\boldsymbol f \to (\langle f,\varphi_\bk\rangle)_{\bk\in\I(V)}$, so that~\eqref{eq:normal} converges to the coefficients of~\eqref{eq:projection}.

			For finite~$M$, the parameters are coupled through the two constraints
			\begin{align}
				c_1 \leq \frac{M}{(r+1)\log M}, \quad & c_1 \coloneqq 10\cdot 2^\alpha |\I(V)|, \label{eq:1st_constraint} \\[1ex]
				1-c_2M^{-r} \geq p, \quad & c_2 \coloneqq 2^{3/4}+1, \label{eq:2nd_constraint}
			\end{align}
			for a prescribed target certainty~$p\in(0,1)$. Here~$r>0$ is a purely theoretical parameter, while~$M$ is the actual sampling effort.
			Since the right-hand side of~\eqref{eq:1st_constraint} increases monotonically for~$M>3$, the smallest admissible~$M$ is optimal, so~\eqref{eq:1st_constraint} may be assumed active. This gives~$r = M/(c_1\log M)-1$ and, inserted into~\eqref{eq:2nd_constraint}, the condition~$\log((1-p)/c_2) \geq \log M - M/c_1$ determining~$M$.
			These constraints couple target accuracy (through~$c_1$, since a smaller error bound requires a larger~$\I(V)$), certainty~$p$ and sample size~$M$. Fixing any one of them yields a monotone trade-off between the remaining two.
			For fixed~$c_1$, for instance, the achievable certainty~$p(M) = 1-c_2\,M^{\,1-M/(c_1\log M)}$ increases monotonically in~$M$ with~$p(M)\to1$. Conversely, for fixed~$M$ a larger~$\I(V)$ lowers the error bound but requires a smaller~$r$, which in turn leads to a smaller certainty.
		\end{remark}

		\section{Control systems\label{sec:control}}

		In this section, we present our
		main contribution, i.\,e.,
		a tailored and efficient ANOVA approximation for control-affine
		systems whose right-hand side exhibits
		a low-dimensional coupling structure.
		Such systems arise from spatial discretizations of controlled partial differential equations, where the locality of the differential operators induces sparse interactions between the state variables. Consequently, each component couples only with its spatial neighbours.
		We will make the ANOVA machinery of Section~\ref{sec:ANOVA} applicable to this setting while leveraging the control-affine structure.

		Control-affine dynamical systems are of the form
		\begin{align*}
			\bx^+ = F(\bx) + G(\bx)\bu.
		\end{align*}
		Therein, $\bx\in \R^d$ is the state, the control~$\bu$ evolves in a compact control constraint set~$\U \subset \R^m$, and the maps~$F\colon\R^d \to \R^d$ and~$G\colon\R^d \to \R^{d\times m}$ with~$d,m \in \N$ describe the drift and the input, respectively. Note that the number of actuators is usually much smaller than the number of states, i.\,e., $m \ll d$. 
		We now define~$H\colon\R^d\times \R^m \to \R^d$ via 
		\begin{align}
			\label{eq:system_H}
			H(\bx,\bu) \coloneqq 
			\begin{bmatrix}
				F(\bx) & G(\bx)
			\end{bmatrix}\begin{bmatrix}
				1 \\ \bu
			\end{bmatrix} =  F(\bx) + G(\bx)\bu
		\end{align}
		and assume that~$F$ and~$G$ are of low-dimensional structure. More precisely, we require that, for suitable index sets, we have
		\begin{align}
			\label{eq:ANOVA_struct_F,G}
			F_j(\bx) = \sum_{\bv \in V_{j,0}} f^j_\bv(\bx_\bv), \qquad G_{j\ell}(\bx) = \sum_{\bv \in V_{j,\ell}} g^{j\ell}_\bv(\bx_\bv), \quad j=1,\dots,d,\, \ell=1,\ldots,m,
		\end{align}
		where we denote the functions~$F$ and~$G$ entry-wise as~$F(\bx) = \left( F_j(\bx) \right)_{j=1}^d$ and~$G(\bx) = \left( G_{j\ell}(\bx) \right)_{j=1, \ell=1}^{d,\,m}$, with~$V_{j,\ell}\subset 2^{[d]}$ being index sets with small cardinality for all~$j=1,\dots,d$, and~$\ell=0,\dots,m$. Moreover, we denote by~$G_\ell\coloneqq(G_{j\ell})_{j=1}^d$ the $\ell$-th column of~$G$, i.\,e., the $\ell$-th control vector field.
		Our task is then to reconstruct this structure from given samples.
		We illustrate the class of systems~\eqref{eq:system_H} with two examples, namely a small benchmark oscillator and a higher-dimensional system stemming from a discretized partial differential equation.

		\begin{example}
			\label{ex:Duffing}
			Consider the Euler approximation of a controlled Duffing oscillator
			\begin{align}
				\label{eq:Duffing}
				\bx^+ = \bx + \Delta t \begin{pmatrix} x_2 \\ x_1 - 3 x_1^3 u \end{pmatrix}
				= \begin{pmatrix} x_1 + \Delta t x_2 \\ x_2 + \Delta t(x_1 - 3x_1^3 u) \end{pmatrix}
			\end{align}
			which corresponds to the system~\eqref{eq:system_H} with~$d=2$, $m=1$, and
			\begin{align*}
				F(\bx) = \begin{pmatrix} x_1 + \Delta t x_2 \\ x_2 + \Delta t x_1 \end{pmatrix}, \qquad
				G(\bx) = \begin{pmatrix} 0 \\ -3\Delta t x_1^3 \end{pmatrix}.
			\end{align*}
			Clearly, these functions~$F$ and~$G$ are of low-dimensional structure compared to the small state space dimension of this problem, since only one-dimensional terms occur, but no constant or two-dimensional terms. In the notation of~\eqref{eq:ANOVA_struct_F,G} we have the index sets~$V_{1,0}=V_{2,0}=\{\{1\}, \{2\}\}$, $V_{1,1}=\varnothing$, as~$G_{11}\equiv 0$ contributes no terms at all, and~$V_{2,1}=\{\{1\}\}$.
		\end{example}

		The next example is higher-dimensional and closer to the applications targeted in Section~\ref{sec:numerics}, where we will illustrate the proposed method on discretizations of partial differential equations, which naturally lead to high-dimensional problems with a low-dimensional coupling structure.

		\begin{example}[Controlled viscous Burgers equation]
			\label{ex:burgers}
			Consider the controlled viscous Burgers equation
			\begin{align*}
				\frac{\partial}{\partial t}x = D\,\frac{\partial^2}{\partial \omega^2}x - x\,\frac{\partial}{\partial \omega}x + \gamma(\omega)\, x\, U_1 + b(\omega)\, U_2,
				\qquad \omega\in(0,L), \; t>0,
			\end{align*}
			with viscosity~$D>0$ and two scalar controls~$\bu=(U_1,U_2)^\top$ acting through nonnegative actuator profiles~$\gamma,b:(0,L) \to \R$.
			We discretize the spatial domain by an equidistant grid~$\omega_i \coloneqq i\,\Delta \omega$, $i=1,\dots,d$, with grid size~$\Delta \omega \coloneqq L/d>0$. All quantities indexed by~$i$ denote evaluations at the grid point~$\omega_i$, i.\,e., $X_i \approx x(\omega_i,\cdot)$, $\gamma_i \coloneqq \gamma(\omega_i)$, and~$b_i \coloneqq b(\omega_i)$.
			Using a finite-difference approximation of the spatial derivatives on this grid and an explicit Euler method with time step~$\Delta t>0$, we obtain a system of the form~\eqref{eq:system_H} with~$m=2$ and
			\begin{align*}
				F_i(X) &= X_i + \Delta t\left(D\,\frac{X_{i+1} - 2X_i + X_{i-1}}{\Delta \omega^2} - X_i\,\frac{X_{i+1}-X_{i-1}}{2\Delta \omega} \right), \\
				G_{i,1}(X) &= \Delta t\,\gamma_i\, X_i, \qquad
				G_{i,2}(X) = \Delta t\, b_i, \qquad i=1,\dots,d,
			\end{align*}
			at each spatial point~$i$. Here, we understand the indices cyclically, i.\,e.,~$X_0\coloneqq X_d$ and~$X_{d+1}\coloneqq X_1$, corresponding to periodic boundary conditions.
			Again, $F$ and~$G$ are of low-dimensional structure. The diffusion term contributes the one-dimensional terms of node~$i$ and of its two neighbours, while the advection term contributes the two adjacent pair interactions. The input map is state-affine in the first control and constant in the second.
			In the notation of~\eqref{eq:ANOVA_struct_F,G} this yields, for every~$i=1,\dots,d$, the index sets
			\begin{align*}
				V_{i,0} = \bigl\{\{i-1\},\{i\},\{i+1\},\{i-1,i\},\{i,i+1\}\bigr\}, \quad
				V_{i,1} = \bigl\{\{i\}\bigr\}, \quad
				V_{i,2} = \{\varnothing\},
			\end{align*}
			where~$V_{i,1}=\varnothing$ for~$\gamma_i=0$ and~$V_{i,2}=\varnothing$ for~$b_i=0$.
			Hence, although the number of discretization points and hence the dimension $d$ may be large, every interaction set satisfies~$|\bv|\leq 2$.
			We remark this low-dimensional structure is also present in the context of other boundary conditions, e.\,g., of Dirichlet or Neumann type, only modify the stencil at the two boundary nodes~$i\in\{1,d\}$. 
		\end{example}

		In general, the vector field $F$ and the columns of $G$ cannot be sampled separately, such that an off-the-shelf application of the ANOVA approximation of Section~\ref{sec:ANOVA} to the individual functions in~\eqref{eq:ANOVA_struct_F,G} is not possible. Therefore, we aim for a coupled reconstruction of the dynamics map~\eqref{eq:system_H}.
		However, applying the ANOVA approximation as a black box to this coupled map does not leverage the structure, since ignoring the control-affine structure would spend considerable
		computational effort on interactions of the control components that are excluded a priori.

		In many data-driven scenarios, one merely observes the system under some given excitation, i.\,e., one is provided with samples~$(\bx^i,\bu^i,H(\bx^i,\bu^i))$. These pairs of samples correspond to state, the employed action, and the resulting successor state.
		In the following, we therefore develop an approach that works with i.i.d.\ sampling points~$(\bx^i,\bu^i)$ drawn from a tensor-product density on~$\Omega\times\U$ and computes an approximation by coupled least squares leveraging the control-affine structure.
		To this end, we assume that both the control domain~$\U=\bigtimes_{\ell=1}^m\U_\ell$ and the sampling density~$\widetilde\omega=\prod_{\ell=1}^m\omega_{d+\ell}$ are of tensor-product structure. 
		The main idea is to encode the control-affine structure of~\eqref{eq:system_H} directly into the approximation ansatz. To this end, we employ a dictionary consisting of basis functions in~$\bx$ only, augmented by the same basis functions multiplied by the components~$u_\ell$ of the control.
		Thereby, the ansatz is affine linear in~$\bu$ by construction, and no spurious interactions between the components of~$\bu$ have to be computed.
		We first discuss the case of a univariate control in Section~\ref{sec:1d}, including the resulting coupled least squares problem and its interpretation as an orthogonal projection. In Section~\ref{sec:error_coupled}, we prove an error bound for finitely many samples in the spirit of the bound for uncontrolled scalar functions recalled in Theorem~\ref{Thm:ANOVA_cheb}. Afterwards, we generalize the ansatz and the error bound to multivariate controls in Section~\ref{sec:2d} and refer the reader to Appendix~\ref{sec:complete_func} for an implementation of the resulting approximation via preprocessed ANOVA index sets.

		\subsection[]{Coupled least squares approximation for univariate control}\label{sec:linearity}\label{sec:1d}

		We first examine the
		case~$m=1$ in~\eqref{eq:system_H}, i.\,e., single-input systems.
		In this case, the system to be approximated reads
		\begin{align*}
			\begin{pmatrix} H_1(\bx,u) \\ \vdots \\ H_d(\bx,u) \end{pmatrix} = H(\bx,u) = F(\bx) + G(\bx)u, \quad u\in\R.
		\end{align*}
		Since the ANOVA decomposition of Section~\ref{sec:ANOVA} is defined for scalar-valued functions, we proceed rowwise, i.\,e., our objective is to approximate each row of~\eqref{eq:system_H} separately by means of the ANOVA approximation of Section~\ref{sec:ANOVA_implementation}. Consistent with the structure of the map $H$, we choose an ansatz of the form
		\begin{align}
			\label{eq:approx_H_rowwise}
			H_j(\bx,u) \approx \sum_{\bk\in\I(V_1)} c_\bk^1 \varphi^1_\bk(\bx) + \sum_{\bs\in\I(V_2)} c_\bs^2 \varphi^2_\bs(\bx) u, \quad j=1,\dots,d,
		\end{align}
		where~$\{\varphi^1_\bk\}_{\bk\in\I(V_1)}$ and~$\{\varphi^2_\bs\}_{\bs\in\I(V_2)}$ are two possibly different orthonormal systems.
		To compute the coefficients, we sample
		system~\eqref{eq:system_H} at i.i.d.\ data points~$(\bx^i,u^i)\in\Omega\times\U \subset \R^{d+1}$, $i=1,\dots,M$, drawn according to the tensor product density~$(\omega \cdot \omega_{d+1})(\bx,u)\coloneqq\omega(\bx)\,\omega_{d+1}(u)$ with~$M\in\N$ and~$\omega_{d+1}$ being the density corresponding to~$\U\subset\R$.
		Throughout this section, we abbreviate the first two moments of the control distribution by
		\begin{align}
			\label{eq:eta_nu}
			\eta \coloneqq \langle 1,u \rangle_{L_2(\U,\omega_{d+1})} = \int_\U u \,\omega_{d+1}(u) \,\mathrm du,
			\qquad
			\nu \coloneqq \|u\|^2_{L_2(\U,\omega_{d+1})} = \int_\U u^2 \,\omega_{d+1}(u) \,\mathrm du.
		\end{align}
		Here, $\eta$ is the mean and $\nu$ the second moment of a control drawn according to~$\omega_{d+1}$, such that~$\nu-\eta^2$ is its variance.
		Note that~$\|1\|^2_{L_2(\U,\omega_{d+1})} = \int_\U \omega_{d+1}(u)\,\mathrm du = 1$, since~$\omega_{d+1}$ is a probability density, such that expanding the square yields
		\begin{align}
			\label{eq:cs_control}
			0\leq 	\|u-\eta\|^2_{L_2(\U,\omega_{d+1})}
			= \nu - 2\eta\,\langle 1,u \rangle_{L_2(\U,\omega_{d+1})} + \eta^2 \,\|1\|^2_{L_2(\U,\omega_{d+1})}
			= \nu-\eta^2.
		\end{align}
		We therefore have $\eta^2\leq\nu$, and equality holds if and only if~$\|u-\eta\|_{L_2(\U,\omega_{d+1})}=0$, i.\,e., if and only if~$u$ equals its mean~$\eta$ $\omega_{d+1}$-almost everywhere.
		Negating this equality case shows that the following three conditions are equivalent:
		\begin{enumerate}[label=(\roman*)]
			\item The control~$u$ is not $\omega_{d+1}$-almost everywhere constant.
			\item The variance~$\nu-\eta^2$ of the control distribution is positive.
			\item The control~$u$ attains at least two distinct values on sets of positive measure.
		\end{enumerate}
		Indeed, (i) is the negation of the equality case just discussed, (ii) is equivalent to (i) by~\eqref{eq:cs_control}, and (iii) restates~(i) for a real-valued random variable. We call a control distribution with these properties \emph{non-degenerate}.
		For later reference, we state this property as an assumption.
		\begin{assumption}[Non-degenerate control distribution]
			\label{Ass:nondegenerate}
			The control distribution is non-degenerate, i.\,e., $\nu-\eta^2>0$ with~$\eta$ and~$\nu$ from~\eqref{eq:eta_nu}.
		\end{assumption}
		This assumption will be used repeatedly below. It guarantees that the control samples carry enough information to separate the contributions of~$F_j$ and~$G_j$, whereas these contributions are indistinguishable for an almost everywhere constant control.

		For each~$j=1,\dots,d$, we define the vectors~$\boldsymbol h \coloneqq \left(H_j(\bx^i,u^i)\right)_{i=1}^M \in \R^M$,
		$\bc^1 \coloneqq \left(c_\bk^1\right)_{\bk\in\I(V_1)} \in \R^{|\I(V_1)|}$,
		and~$\bc^2 \coloneqq \left(c_\bs^2\right)_{\bs\in\I(V_2)} \in \R^{|\I(V_2)|}$,
		as well as the matrices
		\begin{align}
			\label{eq:matrices_A1A2}
			\bA_1 \coloneqq \left(\varphi^1_\bk(\bx^i)\right)_{i=1,\,\bk\in\I(V_1)}^{M} \in \R^{M\times|\I(V_1)|}, 
			\qquad
			\bA_2 \coloneqq \left(\varphi^2_\bs(\bx^i)\right)_{i=1,\,\bs\in\I(V_2)}^{M} \in \R^{M\times|\I(V_2)|},
		\end{align}
		and
		\begin{align}
			\label{eq:matrix_U}
			\bU \coloneqq \operatorname{diag}(u^1,\ldots,u^M)\in \R^{M\times M}.
		\end{align}
		Then, for fixed~$j=1,\dots,d$, the approximation~\eqref{eq:approx_H_rowwise} can be written in matrix-vector notation as
		\begin{align}
			\label{eq:matrixform_m1}
			\boldsymbol h
			\approx
			\begin{pmatrix}
				\bA_1 & \bU \bA_2
			\end{pmatrix}
			\begin{pmatrix}
				\bc^1\\
				\bc^2
			\end{pmatrix}.
		\end{align}
		The coefficients~$\bc^1, \bc^2$ are then computed analogously to~\eqref{eq:lsqr_problem} by solving the least squares problem
		\begin{align}
			\label{eq:opt_fullanova}
			\underset{(\tilde\bc^1,\tilde\bc^2)^\top\in\R^{|\I(V_1)|+|\I(V_2)|}}{\min} \left\| \begin{pmatrix} \bA_1 & \bU \bA_2 \end{pmatrix}  \begin{pmatrix} \tilde\bc^1 \\ \tilde\bc^2 \end{pmatrix} - \boldsymbol h \right\|_2^2.
		\end{align}
		For the least squares solution to be computable by the normal equations
		\begin{align}
			\label{eq:neq_control}
			\begin{pmatrix}
				\bA_1^\top \\
				\bA_2^\top \bU
			\end{pmatrix}
			\begin{pmatrix}
				\bA_1 & \bU \bA_2
			\end{pmatrix}
			\begin{pmatrix}
				\tilde\bc^1\\
				\tilde\bc^2
			\end{pmatrix}
			=
			\begin{pmatrix}
				\bA_1^\top \\
				\bA_2^\top \bU
			\end{pmatrix}
			\boldsymbol h,
		\end{align}
		we need to ensure that the system matrix~$\begin{pmatrix} \bA_1 & \bU \bA_2 \end{pmatrix} \in\R^{M\times(|\I(V_1)|+|\I(V_2)|)}$ has full column rank.
		To this end, we necessarily
		require $M \geq |\I(V_1)|+|\I(V_2)|$, i.\,e., sufficiently many sample points~$\bx^i$, $i=1,\dots,M$.
		In addition, we may assume that~$\bA_1, \bA_2$ have full column rank, see~Section~\ref{sec:ANOVA_implementation}.
		Whether the composed matrix has full column rank depends on the sampled control values~$u^i$ and thus on their underlying sampling distribution. For instance, for~$\varphi^1_{\boldsymbol 0}=\varphi^2_{\boldsymbol 0}\equiv 1$ and constant control samples~$u^1=\dots=u^M=c\neq 0$ the matrix~$\begin{pmatrix} \bA_1 & \bU \bA_2 \end{pmatrix} = \begin{pmatrix} \boldsymbol 1 & c\boldsymbol 1 \end{pmatrix}$ is rank deficient. 
		However, by Proposition~\ref{Prop:limit_projection} below the scaled Gram matrix converges to a positive-definite limit under Assumption~\ref{Ass:nondegenerate}. Thus, for almost every realization the system matrix has full column rank for all sufficiently large~$M$. A quantitative version of this statement, which holds with high probability, is given in Proposition~\ref{Prop:MZ} below.

		Before turning to the error analysis, we clarify the main objective of the coupled least squares problem~\eqref{eq:opt_fullanova}.
		To this end, we consider the infinite-data limit.
		Roughly speaking and after scaling by~$\frac1M$, the entries of the normal equations~\eqref{eq:neq_control} are Monte-Carlo estimates of the corresponding inner products in~$L_2(\Omega\times\U,\omega\cdot\omega_{d+1})$. Hence, the scaled discrete system converges to a continuous one as the number~$M$ of samples grows.
		The following proposition identifies this limit system as the one characterizing the orthogonal projection of~$H_j$ onto the span of a particular dictionary corresponding to the ansatz space of~\eqref{eq:approx_H_rowwise}. It also states that this system is uniquely solvable and shows that the coefficients computed from finitely many samples converge to those of this projection.
		Consequently, the approach of this section is an empirical least squares approximation with respect to this dictionary. This interpretation is used throughout the remainder of the section. It may be used as a guidance for the choice of the control density (see Remark~\ref{Remark:schur}), allows the comparison with a separate ANOVA approximation of~$F$ and~$G$ in Remark~\ref{Remark:anova_separate}, and leads to the multivariate analogue in Section~\ref{sec:2d}.

		\begin{proposition}
			\label{Prop:limit_projection}
			Let~$(\bx^i,u^i)$, $i=1,\dots,M$, be i.i.d.\ samples drawn according to the density~$\omega\cdot\omega_{d+1}$, let Assumption~\ref{Ass:nondegenerate} hold, and let~$F_j,\,G_j\in L_2(\Omega,\omega)$.
			Then, abbreviating by
			\begin{align*}
				\mathcal{\widetilde D} = \{\varphi^1_\bk \colon \bk\in\I(V_1)\} \cup \{\varphi^2_\bs u \colon \bs\in\I(V_2)\}
			\end{align*}
			the dictionary of the ansatz~\eqref{eq:approx_H_rowwise}, the following assertions hold.
			\begin{enumerate}[label=(\roman*)]
				\item After scaling by~$\frac1M$, the normal equations~\eqref{eq:neq_control} converge for~$M\to\infty$, entrywise and almost surely, to a limit system.
				\item This limit system is the Gram system of~$\mathcal{\widetilde D}$ in~$L_2(\Omega\times\U,\omega\cdot\omega_{d+1})$. It is uniquely solvable, and its solution consists of the coefficients of the orthogonal projection~$P_{\mathcal{\widetilde D}}H_j$ of~$H_j$ onto~$\operatorname{span}\mathcal{\widetilde D}$.
				\item For almost every realization of the samples, the scaled normal equations are uniquely solvable for all sufficiently large~$M$, and their solutions converge to the coefficients of~$P_{\mathcal{\widetilde D}}H_j$.
			\end{enumerate}
		\end{proposition}

		\begin{proof}
			We prove the three assertions in the order in which they are stated.

			\emph{Convergence of the sampled inner products.}
			The normal equations~\eqref{eq:neq_control} may also be written as
			\begin{align}
				\label{eq:normal_system_linear}
				\begin{pmatrix}
					\bA_1^\top \bA_1 & \bA_1^\top \bU \bA_2 \\
					\bA_2^\top \bU \bA_1 & \bA_2^\top \bU^2 \bA_2
				\end{pmatrix}  
				\begin{pmatrix}
					\tilde\bc^1 \\
					\tilde\bc^2
				\end{pmatrix}
				=
				\begin{pmatrix}
					\bA_1^\top \boldsymbol h \\
					\bA_2^\top \bU \boldsymbol h
				\end{pmatrix}.
			\end{align}
			Then the strong law of large numbers yields for i.i.d.\ data points~$(\bx^i,u^i)\in\Omega\times\U \subset \R^{d+1}$, $i=1,\dots,M$, drawn according to the density~$\omega \cdot \omega_{d+1}$ that, after scaling, the entries of the block matrix converge almost surely to
			\begin{align*}
				\lim_{M\to\infty} \frac{1}{M}[\bA_1^\top \bA_1]_{\bk,\boldsymbol n} &= \lim_{M\to\infty} \frac1M \sum_{i=1}^M \varphi^1_\bk(\bx^i) \varphi^1_{\boldsymbol n}(\bx^i) = \langle \varphi^1_\bk,\varphi^1_{\boldsymbol n} \rangle_{L_2(\Omega,\omega)} = \delta_{\bk,\boldsymbol n}, \bk,\,\boldsymbol n\in\I(V_1), \\
				\lim_{M\to\infty} \frac1M [\bA_1^\top \bU \bA_2]_{\bk,\boldsymbol s} &= \lim_{M\to\infty} \frac1M \sum_{i=1}^M \varphi^1_\bk(\bx^i) u^i \varphi^2_{\boldsymbol s}(\bx^i) = \langle \varphi^1_\bk, u\varphi^2_{\boldsymbol s} \rangle_{L_2(\Omega\times\U,\omega \cdot \omega_{d+1})}, \bk\in\I(V_1),\,\boldsymbol s\in\I(V_2), \\
				\lim_{M\to\infty} \frac1M [\bA_2^\top \bU^2 \bA_2]_{\boldsymbol r,\boldsymbol s} &= \lim_{M\to\infty} \frac1M \sum_{i=1}^M u^i \varphi^2_{\boldsymbol r}(\bx^i) u^i \varphi^2_{\boldsymbol s}(\bx^i) = \langle u\varphi^2_{\boldsymbol r},u\varphi^2_{\boldsymbol s} \rangle_{L_2(\Omega\times\U,\omega \cdot \omega_{d+1})}, \boldsymbol r,\,\boldsymbol s\in\I(V_2).
			\end{align*}
			Here, the strong law of large numbers is applicable since all involved random variables are integrable. Indeed, the products of two basis functions are integrable by the Cauchy--Schwarz inequality, the additional factors~$u^i$ are bounded by the compactness of~$\U$, and~$H_j\in L_2(\Omega\times\U,\omega\cdot\omega_{d+1})$.
			Moreover, all integrands appearing above are absolutely integrable. By Tonelli's theorem, the Cauchy--Schwarz inequality and the orthonormality of the two systems we have
			\begin{align*}
				\int_{\Omega\times\U} \big|\varphi^1_\bk(\bx) \varphi^2_{\boldsymbol s}(\bx)\, u\big| \,(\omega \cdot \omega_{d+1})(\bx,u) \,\mathrm{d}(\bx,u)
				\leq \|\varphi^1_\bk\|_{L_2(\Omega,\omega)} \|\varphi^2_{\boldsymbol s}\|_{L_2(\Omega,\omega)} \max_{u\in\U}|u| < \infty.
			\end{align*}
			The remaining integrands can be treated analogously, where we used the compactness of~$\U$.
			Hence, Fubini's theorem yields
			\begin{align}
				\langle \varphi^1_\bk, u\varphi^2_{\boldsymbol s} \rangle_{L_2(\Omega\times\U,\omega \cdot \omega_{d+1})}
				&= \int_{\Omega\times\U} \varphi^1_\bk(\bx) \varphi^2_{\boldsymbol s}(\bx) u \,(\omega \cdot \omega_{d+1})(\bx,u) \,\mathrm{d}(\bx,u) \notag \\
				&= \int_{\Omega}\int_{\U} \varphi^1_\bk(\bx) \varphi^2_{\boldsymbol s}(\bx) u \,\omega(\bx) \,\omega_{d+1}(u) \,\mathrm{d}u\,\mathrm{d}\bx \notag \\[1ex]
				&= \langle \varphi^1_\bk,\varphi^2_{\boldsymbol s} \rangle_{L_2(\Omega,\omega)} \langle 1,u \rangle_{L_2(\U,\omega_{d+1})}, \label{eq:Fubini_u} \\
				\langle u\varphi^2_{\boldsymbol r},u\varphi^2_{\boldsymbol s} \rangle_{L_2(\Omega\times\U,\omega \cdot \omega_{d+1})}
				&= \int_{\Omega\times\U} \varphi^2_{\boldsymbol r}(\bx) \varphi^2_{\boldsymbol s}(\bx) u^2 \,(\omega \cdot \omega_{d+1})(\bx,u) \,\mathrm{d}(\bx,u) \notag \\
				&= \int_{\Omega}\int_{\U} \varphi^2_{\boldsymbol r}(\bx) \varphi^2_{\boldsymbol s}(\bx) u^2 \,\omega(\bx) \,\omega_{d+1}(u) \,\mathrm{d}u\,\mathrm{d}\bx \notag \\[1ex]
				&= \langle \varphi^2_{\boldsymbol r},\varphi^2_{\boldsymbol s} \rangle_{L_2(\Omega,\omega)} \|u\|^2_{L_2(\U,\omega_{d+1})}. \label{eq:Fubini_u^2}
			\end{align}
			For the right-hand side we obtain similarly, now using in addition that~$F_j,G_j\in L_2(\Omega,\omega)$,
			\begin{align*}
				\lim_{M\to\infty} \frac1M [\bA_1^\top \boldsymbol h]_\bk
				&= \lim_{M\to\infty} \frac1M \sum_{i=1}^M \varphi^1_\bk(\bx^i)H_j(\bx^i,u^i) 
				= \langle H_j,\varphi^1_\bk \rangle_{L_2(\Omega\times\U,\omega \cdot \omega_{d+1})} \\
				&= \langle F_j,\varphi^1_\bk \rangle_{L_2(\Omega,\omega)} + \langle G_j,\varphi^1_\bk \rangle_{L_2(\Omega,\omega)} \langle 1,u\rangle_{L_2(\U,\omega_{d+1})}
			\end{align*}
			as well as
			\begin{align*}
				\lim_{M\to\infty} \frac1M [\bA_2^\top \bU \boldsymbol h]_{\boldsymbol s}
				&= \lim_{M\to\infty} \frac1M \sum_{i=1}^M u^i\varphi^2_{\boldsymbol s}(\bx^i)H_j(\bx^i,u^i) 
				= \langle H_j,u\varphi^2_{\boldsymbol s} \rangle_{L_2(\Omega\times\U,\omega \cdot \omega_{d+1})} \\
				&= \langle F_j,\varphi^2_{\boldsymbol s} \rangle_{L_2(\Omega,\omega)}\langle 1,u \rangle_{L_2(\U,\omega_{d+1})} + \langle G_j,\varphi^2_{\boldsymbol s} \rangle_{L_2(\Omega,\omega)} \|u\|^2_{L_2(\U,\omega_{d+1})}.
			\end{align*}
			Thus, by~\eqref{eq:Fubini_u} and~\eqref{eq:Fubini_u^2}, for~$M\to\infty$ the linear system~\eqref{eq:normal_system_linear} converges to
			\begin{align}
				\label{eq:convergence_neq}
				&
				\begin{pmatrix}
					\bI_{|\I(V_1)|} & \langle 1,u \rangle_{L_2(\U,\omega_{d+1})} \cdot \boldsymbol\Phi \\
					\langle 1,u \rangle_{L_2(\U,\omega_{d+1})} \cdot \boldsymbol\Phi^\top & \|u\|^2_{L_2(\U,\omega_{d+1})} \cdot \bI_{|\I(V_2)|} 
				\end{pmatrix}  
				\begin{pmatrix}
					\tilde\bc^1\\
					\tilde\bc^2
				\end{pmatrix} \notag \\[1ex]
				&\hspace{3em} =
				\begin{pmatrix}
					\left(\langle F_j,\varphi^1_\bk \rangle_{L_2(\Omega,\omega)} + \langle G_j,\varphi^1_\bk \rangle_{L_2(\Omega,\omega)} \langle 1,u \rangle_{L_2(\U,\omega_{d+1})}\right)_{\bk\in\I(V_1)}
					\\
					\!\left(\langle F_j,\varphi^2_{\boldsymbol s} \rangle_{L_2(\Omega,\omega)}\langle 1,u \rangle_{L_2(\U,\omega_{d+1})} + \langle G_j,\varphi^2_{\boldsymbol s} \rangle_{L_2(\Omega,\omega)} \|u\|^2_{L_2(\U,\omega_{d+1})}\right)_{\boldsymbol s\in\I(V_2)}
				\end{pmatrix}
			\end{align}
			with~$\boldsymbol\Phi\coloneqq\left(\langle \varphi^1_\bk,\varphi^2_{\boldsymbol s} \rangle_{L_2(\Omega,\omega)}\right)_{\bk\in\I(V_1),\,\boldsymbol s\in\I(V_2)}$, which shows~(i).

			\emph{Identification with the orthogonal projection.}
			We consider the dictionary
			\begin{align}
				\label{eq:dictionary}
				\mathcal{\widetilde D} \coloneqq \{\psi_{\boldsymbol n} \colon \boldsymbol n\in\I\}
				= \{\varphi^1_\bk \colon \bk\in\I(V_1)\} \cup \{\varphi^2_\bs u \colon \bs\in\I(V_2)\},
				\quad \I \coloneqq \I(V_1) \cup \I(V_2),
			\end{align}
			where the unions are to be understood as disjoint unions, since an index~$\bk\in\I(V_1)\cap\I(V_2)$ refers to the two different functions~$\varphi^1_\bk$ and~$\varphi^2_\bk u$.
			The corresponding orthogonal projection of a function~$f\in L_2(\Omega\times\U,\omega \cdot \omega_{d+1})$ onto~$\operatorname{span}\mathcal{\widetilde D}$ is given by
			\begin{align*}
				P_{\mathcal{\widetilde D}}f(\bx,u)
				\coloneqq \sum_{\boldsymbol n\in\I} c_{\boldsymbol n} \psi_{\boldsymbol n}(\bx,u),
			\end{align*}
			where the coefficients are given as the solution of the linear system
			\begin{align}
				\label{eq:lgs_orthproj}
				(\langle \psi_{\boldsymbol p},\psi_{\boldsymbol n} \rangle_{L_2(\Omega\times\U,\omega \cdot \omega_{d+1})})_{\boldsymbol p,\boldsymbol n\in\I} \cdot (c_{\boldsymbol n})_{\boldsymbol n\in\I} 
				= (\langle f,\psi_{\boldsymbol p} \rangle_{L_2(\Omega\times\U,\omega \cdot \omega_{d+1})})_{\boldsymbol p\in\I}.
			\end{align}
			By the structure of~$\mathcal{\widetilde D}$, every element of~$\operatorname{span}\mathcal{\widetilde D}$ is of the form
			$\sum_{\bk\in\I(V_1)} c^1_\bk \varphi^1_\bk(\bx) + \sum_{\boldsymbol s\in\I(V_2)} c^2_{\boldsymbol s} \varphi^2_{\boldsymbol s}(\bx)\,u$,
			i.\,e., precisely of the form of our approximation~\eqref{eq:approx_H_rowwise}.
			Moreover, ordering the dictionary according to this splitting and choosing~$f=H_j$, the Gram system~\eqref{eq:lgs_orthproj} coincides with the limit system~\eqref{eq:convergence_neq}, since its entries~$\langle \psi_{\boldsymbol p},\psi_{\boldsymbol n} \rangle$ and~$\langle H_j,\psi_{\boldsymbol p} \rangle$ are the limits computed in the first step. For functions not depending on~$u$, the inner products of~$L_2(\Omega\times\U,\omega \cdot \omega_{d+1})$ and of~$L_2(\Omega,\omega)$ agree, since~$\omega_{d+1}$ is a probability density.
			Recall that a Gram matrix is positive definite if and only if the underlying family is linearly independent. It therefore remains to verify the latter, which holds by Assumption~\ref{Ass:nondegenerate}. If
			\begin{align*}
				\sum_{\bk\in\I(V_1)} a_\bk \varphi^1_\bk(\bx) + \bigg(\sum_{\boldsymbol s\in\I(V_2)} b_{\boldsymbol s} \varphi^2_{\boldsymbol s}(\bx)\bigg) u = 0
			\end{align*}
			for $(\omega\cdot\omega_{d+1})$-almost all~$(\bx,u)$, then, since~$u$ attains at least two distinct values on a set of positive measure, the left-hand side is an affine function of~$u$ vanishing at two distinct points for almost every~$\bx$. Hence, both sums vanish separately.
			The orthonormality of the two systems then yields~$a_\bk=b_{\boldsymbol s}=0$.\footnote{Note that for an $\omega_{d+1}$-almost everywhere constant control the two parts of the dictionary may overlap and this argument breaks down, in accordance with Assumption~\ref{Ass:nondegenerate}.}
			Altogether, the limit system~\eqref{eq:convergence_neq} and the projection system~\eqref{eq:lgs_orthproj} for~$f=H_j$ are one and the same system, and it is invertible. Hence, both systems have the same unique solution.
			Viewing the index set~$\I$ in~\eqref{eq:dictionary} as the disjoint union of~$\I(V_1)$ and~$\I(V_2)$, this solution is given by the pair~$(\tilde\bc^1,\tilde\bc^2)$ and, equivalently, by the coefficient vector~$(c_{\boldsymbol n})_{\boldsymbol n\in\I}$ of~$P_{\mathcal{\widetilde D}}H_j$. Thus, the two coefficient representations coincide, which proves~(ii).

			\emph{Convergence of the solutions.}
			The determinant depends continuously on the entries of a matrix, such that the almost sure convergence of the scaled system matrix to the invertible limit matrix implies that, for almost every realization, the scaled system matrix is invertible for all sufficiently large~$M$.
			The threshold~$M$ from which on this holds depends on the realization, so this argument does not yield a quantitative statement, cf.\ Remark~\ref{Remark:schur} and Proposition~\ref{Prop:MZ} below.
			Since matrix inversion is continuous as well, the solutions~$(\tilde\bc^1,\tilde\bc^2)$ of~\eqref{eq:opt_fullanova} then converge almost surely to the coefficients of~$P_{\mathcal{\widetilde D}}H_j$, which shows~(iii).
		\end{proof}

		The following remark quantifies this solvability in terms of the control distribution.
		\begin{remark}[Quantitative solvability and conditioning]
			\label{Remark:schur}
			The unique solvability of the limit system~\eqref{eq:convergence_neq} was obtained in the proof of Proposition~\ref{Prop:limit_projection} from the linear independence of the dictionary, which is a purely qualitative statement.
			It may be quantified in terms of the first two moments~\eqref{eq:eta_nu} of the control distribution alone.
			To this end, we abbreviate~$g\coloneqq\sum_{\bk\in\I(V_1)} a_\bk\varphi^1_\bk$ and~$h\coloneqq\sum_{\boldsymbol s\in\I(V_2)} y_{\boldsymbol s}\varphi^2_{\boldsymbol s}$ for coefficient vectors~$\boldsymbol a,\by$, so that~$\|\boldsymbol a\|_2=\|g\|_{L_2(\Omega,\omega)}$ and~$\|\by\|_2=\|h\|_{L_2(\Omega,\omega)}$ by the orthonormality of the two systems.
			Denoting the system matrix of~\eqref{eq:convergence_neq} by~$\mathcal G$ and recalling from the proof of Proposition~\ref{Prop:limit_projection} that it is the Gram matrix of the dictionary~\eqref{eq:dictionary}, its quadratic form reads
			\begin{align*}
				(\boldsymbol a,\by)^\top \mathcal G\, (\boldsymbol a,\by)
				= \|g+hu\|^2_{L_2(\Omega\times\U,\omega\cdot\omega_{d+1})}
				= \|g\|^2_{L_2(\Omega,\omega)} + 2\eta\,\langle g,h\rangle_{L_2(\Omega,\omega)} + \nu\,\|h\|^2_{L_2(\Omega,\omega)}.
			\end{align*}
			By the Cauchy--Schwarz inequality, $|\langle g,h\rangle_{L_2(\Omega,\omega)}|\leq\|g\|_{L_2(\Omega,\omega)}\|h\|_{L_2(\Omega,\omega)}$, such that the right-hand side is bounded from below and from above by the quadratic forms of the symmetric matrices~$\bigl(\begin{smallmatrix}1&\mp|\eta|\\ \mp|\eta|&\nu\end{smallmatrix}\bigr)$ evaluated at~$(\|g\|_{L_2(\Omega,\omega)},\|h\|_{L_2(\Omega,\omega)})$.
			Both matrices have trace~$1+\nu$ and determinant~$\nu-\eta^2$, hence the same eigenvalues, and we obtain
			\begin{align}
				\label{eq:limit_eigenvalues}
				\lambda_-\,\bI \preceq \mathcal G \preceq \lambda_+\,\bI,
				\qquad
				\lambda_\pm \coloneqq \frac{(1+\nu)\pm\sqrt{(1-\nu)^2+4\eta^2}}{2},
				\qquad
				\lambda_-\lambda_+ = \nu-\eta^2.
			\end{align}
			Therein, $\nu-\eta^2$ is the variance of the control, cf.~\eqref{eq:eta_nu}, such that~$\mathcal G$ is positive definite if and only if the control distribution is non-degenerate, and its smallest eigenvalue degenerates as the control variance tends to zero.
			Note that~\eqref{eq:limit_eigenvalues} holds for arbitrary basis systems~$\varphi^1$ and~$\varphi^2$, and that both bounds are attained for~\mbox{$\{\varphi^1_\bk\}_{\bk\in\I(V_1)}=\{\varphi^2_{\boldsymbol s}\}_{\boldsymbol s\in\I(V_2)}$}, where~$\mathcal G = \bigl(\begin{smallmatrix}1&\eta\\ \eta&\nu\end{smallmatrix}\bigr)\otimes\bI_{|\I(V_1)|}$.
			Moreover, the condition number of~$\mathcal G$ is bounded by the fraction~$\lambda_+/\lambda_-$, which for fixed~$\nu$ is minimal for~$\eta=0$. If in addition~$\nu=1$, then~$\lambda_\pm=1$, the limit matrix is the identity and the least squares problem is asymptotically perfectly conditioned.
			This provides a practical guideline for the choice of the control density~$\omega_{d+1}$. 
		\end{remark}

		The next remark identifies the situation in which the coupled reconstruction coincides with separate ANOVA approximations of the maps $F$ and $G$.
		\begin{remark}
			\label{Remark:anova_separate}
			If one has access to the samples
			\begin{align}
				\label{eq:samples_FG}
				\boldsymbol f \coloneqq \left(F_j(\bx^i)\right)_{i=1}^M
				\quad\text{and}\quad
				\boldsymbol g \coloneqq \left(G_j(\bx^i)\right)_{i=1}^M,    
			\end{align}
			one could apply the ANOVA approximation separately to the functions~$F$ and~$G$.
			Therein, one would again use the matrices~$\bA_1$ and~$\bA_2$ in~\eqref{eq:matrices_A1A2} and seek coefficients~$\bc^{1,\star}$ and~$\bc^{2,\star}$ with
			\begin{align}
				\label{eq:opt_coeffs_sep}
				\bc^{1,\star}= \underset{\bc\in\R^{|\I(V_1)|}}{\operatorname{argmin}} \| \bA_1 \bc - \boldsymbol f \|_2^2 ,
				\qquad
				\bc^{2,\star} = \underset{\bc\in\R^{|\I(V_2)|}}{\operatorname{argmin}} \| \bA_2 \bc - \boldsymbol g \|_2^2,
			\end{align}
			cf.\ Section~\ref{sec:ANOVA_implementation}.
			In the limit~$M\to\infty$ they would result from the normal equations as
			\begin{align*}
				\bI_{|\I(V_1)|} \cdot \bc^{1,\star} = \begin{pmatrix} \langle F_j,\varphi_\bk^1 \rangle_{L_2(\Omega,\omega)} \end{pmatrix}_{\bk\in\I(V_1)},
				\qquad
				\bI_{|\I(V_2)|} \cdot \bc^{2,\star} = \begin{pmatrix} \langle G_j,\varphi_\bs^2 \rangle_{L_2(\Omega,\omega)} \end{pmatrix}_{\bs\in\I(V_2)},
			\end{align*}
			or combined from the system
			\begin{align*}
				\begin{pmatrix}
					\bI_{|\I(V_1)|} & \boldsymbol 0 \\ \boldsymbol 0 & \bI_{|\I(V_2)|}
				\end{pmatrix}
				\begin{pmatrix}
					\bc^{1,\star} \\ \bc^{2,\star}
				\end{pmatrix}
				=
				\begin{pmatrix}
					\begin{pmatrix} \langle F_j,\varphi_\bk^1 \rangle_{L_2(\Omega,\omega)} \end{pmatrix}_{\bk\in\I(V_1)} \\
					\begin{pmatrix} \langle G_j,\varphi_\bs^2 \rangle_{L_2(\Omega,\omega)} \end{pmatrix}_{\bs\in\I(V_2)}
				\end{pmatrix}.
			\end{align*}
			Inserting this equality into the left-hand side of the system~\eqref{eq:convergence_neq}, we obtain
			\begin{align*}
				\begin{pmatrix}
					\left(\langle F_j,\varphi^1_\bk \rangle_{L_2(\Omega,\omega)}\right)_{\bk\in\I(V_1)} + \langle 1,u \rangle_{L_2(\U,\omega_{d+1})} \cdot \boldsymbol\Phi\! \left(\langle G_j,\varphi^2_\bs \rangle_{L_2(\Omega,\omega)}\right)_{\bs\in\I(V_2)}
					\\
					\langle 1,u \rangle_{L_2(\U,\omega_{d+1})} \cdot \boldsymbol\Phi^\top\! \left(\langle F_j,\varphi^1_\bk \rangle_{L_2(\Omega,\omega)}\right)_{\bk\in\I(V_1)} + \|u\|^2_{L_2(\U,\omega_{d+1})} \left(\langle G_j,\varphi^2_{\boldsymbol s} \rangle_{L_2(\Omega,\omega)}\right)_{\boldsymbol s\in\I(V_2)}
				\end{pmatrix}.
			\end{align*}
			In other words, the coupled reconstruction approach above coincides with the separate application of the ANOVA approximation, if
			\begin{align*}
				\boldsymbol\Phi\! \left(\langle G_j,\varphi^2_\bs \rangle_{L_2(\Omega,\omega)}\right)_{\bs\in\I(V_2)}
				&=
				\left(\sum_{\boldsymbol s\in\I(V_2)} \langle \varphi^1_\bk,\varphi^2_\bs \rangle_{L_2(\Omega,\omega)} \langle G_j,\varphi^2_\bs \rangle_{L_2(\Omega,\omega)}\right)_{\bk\in\I(V_1)} \\
				&=
				\left(\langle G_j,\varphi^1_\bk \rangle_{L_2(\Omega,\omega)}\right)_{\bk\in\I(V_1)}
			\end{align*}
			and
			\begin{align*}
				\boldsymbol\Phi^\top\! \left(\langle F_j,\varphi^1_\bk \rangle_{L_2(\Omega,\omega)}\right)_{\bk\in\I(V_1)}
				&=
				\left(\sum_{\bk\in\I(V_1)} \langle \varphi^1_\bk,\varphi^2_\bs \rangle_{L_2(\Omega,\omega)} \langle F_j,\varphi^1_\bk \rangle_{L_2(\Omega,\omega)}\right)_{\boldsymbol s\in\I(V_2)} \\
				&=
				\left(\langle F_j,\varphi^2_{\boldsymbol s} \rangle_{L_2(\Omega,\omega)}\right)_{\boldsymbol s\in\I(V_2)}.
			\end{align*}
			Both conditions are fulfilled if~$\{\varphi^1_\bk\}_{\bk\in\I(V_1)}=\{\varphi^2_{\boldsymbol s}\}_{\boldsymbol s\in\I(V_2)}$, as then~$\boldsymbol\Phi = \left(\langle \varphi^1_\bk,\varphi^2_{\boldsymbol s} \rangle_{L_2(\Omega,\omega)}\right)_{\bk\in\I(V_1),\,\boldsymbol s\in\I(V_2)}$ is the identity matrix.
			Two refinements of this conclusion are in order.
			First, in the two conditions above the common factor~$\langle 1,u\rangle_{L_2(\U,\omega_{d+1})}$ has been cancelled. If the control density has zero mean, i.\,e., $\langle 1,u\rangle_{L_2(\U,\omega_{d+1})}=0$, the off-diagonal blocks in~\eqref{eq:convergence_neq} vanish and the coupled approach coincides in the limit with the separate ANOVA approximation for \emph{arbitrary} basis systems. This provides yet another argument in favour of zero-mean control densities.
			Second, for~$\langle 1,u\rangle_{L_2(\U,\omega_{d+1})}\neq 0$, if the two conditions are required to hold for all~$F_j,G_j$, they are equivalent to~$\operatorname{span}\{\varphi^1_\bk\}_{\bk\in\I(V_1)}=\operatorname{span}\{\varphi^2_{\boldsymbol s}\}_{\boldsymbol s\in\I(V_2)}$. The systems need not be identical, since any orthonormal basis of the same subspace suffices.
		\end{remark}

		\subsection[]{Error analysis}\label{sec:error_coupled}

		We now turn to the error analysis of the coupled approach for finitely many samples.
		The empirical least squares problem~\eqref{eq:opt_fullanova} controls the residual of~$H_j$ at the sampling points, whereas we aim to bound the $L_2$-error of the resulting approximation.
		The link between these two quantities is a norm equivalence between the discrete and the continuous norm.
		For a \emph{fixed} bounded function~$f$, such an equivalence holds with high probability by Hoeffding's inequality~\cite{Hoeffding63}, since~$\frac1M\sum_{i=1}^M|f(\bz^i)|^2$ is the sample mean of the i.i.d.\ random variables~$|f(\bz^i)|^2$ with expectation~$\|f\|^2_{L_2(\Omega\times\U,\omega\cdot\omega_{d+1})}$.
		The residual~$H_j-\tilde H_j$, however, depends on the samples through the computed coefficients, such that a statement for a fixed function does not suffice.
		Instead, we establish a norm equivalence which holds \emph{uniformly} on the ansatz space, i.\,e., a Marcinkiewicz--Zygmund inequality for the dictionary~$\mathcal{\widetilde D}$ under random sampling, cf.~\cite{CoDaLe13,CoMi17}, where such norm equivalences form the core of the stability analysis of least squares approximation.
		In the proof of the error bound, it is applied to the component of the residual in~$\operatorname{span}\mathcal{\widetilde D}$, whereas the remaining component is a fixed function to which Hoeffding's inequality applies.

		In the following, we fix three conventions.
		First, the index sets~$\I(V_1)$ and~$\I(V_2)$ are understood as \emph{prescribed}, deterministic sets, i.\,e., we assume that the truncation step of Algorithm~\ref{alg:anova} is not carried out but that the index sets are built from the exact structure~\eqref{eq:ANOVA_struct_F,G} instead. We refer the reader to Remark~\ref{Remark:datadependence} for a discussion of the data-dependent case.
		Second, we use the Chebyshev basis~\eqref{eq:basis_cheb} in both systems, i.\,e., $\varphi^1=\varphi^2=\varphi^{\mathrm{cheb}}$ with possibly different index sets~$\I(V_1)$ and~$\I(V_2)$, we sample according to~$\omega^{\mathrm{cheb}}\cdot\omega_{d+1}$ with a probability density~$\omega_{d+1}$ on the compact set~$\U\subset\R$, and we let Assumption~\ref{Ass:nondegenerate} hold, such that~$\lambda_->0$ in~\eqref{eq:limit_eigenvalues}.
		Third, the estimates are stated for one fixed row~$j\in[d]$. A statement holding simultaneously for all~$j=1,\dots,d$ follows from a union bound over the~$d$ rows, which multiplies the failure probability by~$d$. We abbreviate
		\begin{align*}
			n_1\coloneqq|\I(V_1)|,\qquad n_2\coloneqq|\I(V_2)|,\qquad u_{\max}\coloneqq\max_{u\in\U}|u|,\qquad \alpha\coloneqq\max_{\bk\in\I(V_1)\cup\I(V_2)}|\supp\bk|,
		\end{align*}
		and denote by~$\mathcal G_M\coloneqq\frac1M\boldsymbol B^\top\boldsymbol B$ with~$\boldsymbol B\coloneqq(\bA_1\;\;\bU\bA_2)\in\R^{M\times(n_1+n_2)}$ the scaled system matrix of the normal equations~\eqref{eq:normal_system_linear}, whose limit for~$M\to\infty$ is the Gram matrix~$\mathcal G$ of~$\mathcal{\widetilde D}$ by Proposition~\ref{Prop:limit_projection}.
		For~$g=\sum_{\bk\in\I(V_1)}a_\bk\varphi_\bk+u\sum_{\bs\in\I(V_2)}y_\bs\varphi_\bs\in\operatorname{span}\mathcal{\widetilde D}$ with coefficient vector~$\bc=(\boldsymbol a,\by)^\top$ we have
		\begin{align}
			\label{eq:mz_forms}
			\|g\|^2_{L_2(\Omega\times\U,\omega^{\mathrm{cheb}}\cdot\omega_{d+1})}=\bc^\top\mathcal G\,\bc
			\qquad\text{and}\qquad
			\frac1M\sum_{i=1}^M|g(\bx^i,u^i)|^2=\frac1M\|\boldsymbol B\bc\|_2^2=\bc^\top\mathcal G_M\,\bc .
		\end{align}

		\begin{proposition}[Marcinkiewicz--Zygmund inequality for the coupled dictionary]
			\label{Prop:MZ}
			Let~$\delta\in(0,1)$, $r>0$ and~$M>3$ satisfy
			\begin{align}
				\label{eq:mz_condition}
				n_1+u_{\max}^2\,n_2\;\leq\;\frac{c_\delta\,\lambda_-\,M}{2^\alpha\,(r+1)\log M},
				\qquad
				c_\delta\coloneqq(1+\delta)\log(1+\delta)-\delta ,
			\end{align}
			where~$n_1$, $n_2$, $u_{\max}$ and~$\alpha$ are given above, and where
			\begin{align*}
				\lambda_-=\frac{(1+\nu)-\sqrt{(1-\nu)^2+4\eta^2}}{2}>0
			\end{align*}
			is the lower bound on the spectrum of the Gram matrix~$\mathcal G$ from~\eqref{eq:limit_eigenvalues} with the moments~$\eta,\nu$ of the control distribution from~\eqref{eq:eta_nu}.
			Then, with probability at least~$1-2M^{-r}$,
			\begin{align}
				\label{eq:mz}
				(1-\delta)\,\|g\|^2_{L_2(\Omega\times\U,\omega^{\mathrm{cheb}}\cdot\omega_{d+1})}
				\;\leq\;\frac1M\sum_{i=1}^M|g(\bx^i,u^i)|^2
				\;\leq\;(1+\delta)\,\|g\|^2_{L_2(\Omega\times\U,\omega^{\mathrm{cheb}}\cdot\omega_{d+1})}
			\end{align}
			for all $g\in\operatorname{span}\mathcal{\widetilde D}$, or, equivalently, $(1-\delta)\,\mathcal G\preceq\mathcal G_M\preceq(1+\delta)\,\mathcal G$.
		\end{proposition}

		\begin{proof}
			Let~$\boldsymbol v(\bz)\in\R^{n_1+n_2}$ denote the vector of all elements of~$\mathcal{\widetilde D}$ evaluated at~$\bz=(\bx,u)$, such that~$\mathcal G=\mathbb E[\boldsymbol v(\bz)\boldsymbol v(\bz)^\top]$ and~$\mathcal G_M=\frac1M\sum_{i=1}^M\boldsymbol v(\bz^i)\boldsymbol v(\bz^i)^\top$.
			The matrices~$\boldsymbol X_i\coloneqq\mathcal G^{-1/2}\boldsymbol v(\bz^i)\boldsymbol v(\bz^i)^\top\mathcal G^{-1/2}$ are i.i.d.\ and symmetric positive semidefinite and satisfy~$\mathbb E[\boldsymbol X_i]=\mathcal G^{-1/2}\,\mathbb E[\boldsymbol v(\bz^i)\boldsymbol v(\bz^i)^\top]\,\mathcal G^{-1/2}=\bI$ as well as
			\begin{align*}
				\|\boldsymbol X_i\|_2=\|\mathcal G^{-1/2}\boldsymbol v(\bz^i)\|_2^2\leq K\coloneqq\sup_{\bz\in\Omega\times\U}\|\mathcal G^{-1/2}\boldsymbol v(\bz)\|_2^2 .
			\end{align*}
			Hence, the matrix Chernoff inequality~\cite[Theorem~1.1]{Tropp12} yields, with~$n\coloneqq n_1+n_2$,
			\begin{align*}
				\mathbb P\Big(\lambda_{\min}\Big(\sum_{i=1}^M\boldsymbol X_i\Big)\leq(1-\delta)M\Big)\leq n\,\mathrm e^{-M\tilde c_\delta/K},
				\qquad
				\mathbb P\Big(\lambda_{\max}\Big(\sum_{i=1}^M\boldsymbol X_i\Big)\geq(1+\delta)M\Big)\leq n\,\mathrm e^{-Mc_\delta/K},
			\end{align*}
			where~$\tilde c_\delta\coloneqq\delta+(1-\delta)\log(1-\delta)\geq c_\delta$.
			Since~$\mathcal G^{-1/2}\mathcal G_M\mathcal G^{-1/2}=\frac1M\sum_{i=1}^M\boldsymbol X_i$ is symmetric, the event
			\begin{align*}
				E_\delta\coloneqq\Big\{\big\|\mathcal G^{-1/2}\mathcal G_M\mathcal G^{-1/2}-\bI\big\|_2>\delta\Big\}
			\end{align*}
			occurs if and only if the smallest eigenvalue of this matrix is smaller than~$1-\delta$ or its largest eigenvalue is larger than~$1+\delta$. Since the probability of a union of two events is at most the sum of their probabilities, we get~$\mathbb P(E_\delta)\leq2n\,\mathrm e^{-Mc_\delta/K}$.

			It remains to bound~$K = \sup_{\bz\in\Omega\times\U}\|\mathcal G^{-1/2}\boldsymbol v(\bz)\|_2^2$.
			By~\eqref{eq:limit_eigenvalues}, $\|\mathcal G^{-1/2}\boldsymbol v(\bz)\|_2^2=\boldsymbol v(\bz)^\top\mathcal G^{-1}\boldsymbol v(\bz)\leq\|\boldsymbol v(\bz)\|_2^2/\lambda_-$.
			Moreover, the Chebyshev basis~\eqref{eq:basis_cheb} satisfies~$|\varphi_\bk(\bx)|^2\leq2^{|\supp\bk|}\leq2^\alpha$ for all~$\bx\in[-1,1]^d$ and~$\bk\in\I(V_1)\cup\I(V_2)$, such that
			\begin{align*}
				\|\boldsymbol v(\bz)\|_2^2=\sum_{\bk\in\I(V_1)}|\varphi_\bk(\bx)|^2+u^2\sum_{\bs\in\I(V_2)}|\varphi_\bs(\bx)|^2\leq2^\alpha\big(n_1+u_{\max}^2\,n_2\big).
			\end{align*}
			Hence, $K\leq2^\alpha(n_1+u_{\max}^2n_2)/\lambda_-$, and~\eqref{eq:mz_condition} yields~$K\leq c_\delta M/((r+1)\log M)$, i.\,e., $\mathrm e^{-Mc_\delta/K}\leq M^{-(r+1)}$.
			Since~$n=\operatorname{tr}(\bI)=\mathbb E[\operatorname{tr}\boldsymbol X_i]=\mathbb E\|\mathcal G^{-1/2}\boldsymbol v(\bz^i)\|_2^2\leq K\leq M$ (note that~$c_\delta<1$ and~$(r+1)\log M>1$ for~$M>3$), we obtain~$\mathbb P(E_\delta)\leq2nM^{-(r+1)}\leq2M^{-r}$.
			On the complement of~$E_\delta$, we have~$(1-\delta)\bI\preceq\mathcal G^{-1/2}\mathcal G_M\mathcal G^{-1/2}\preceq(1+\delta)\bI$, which is equivalent to~$(1-\delta)\mathcal G\preceq\mathcal G_M\preceq(1+\delta)\mathcal G$ and, by~\eqref{eq:mz_forms}, to~\eqref{eq:mz}.
		\end{proof}

        We now may prove an error bound for the ANOVA-based approximation tailored to control-affine systems.
		For an index set~$\I\subset\N_0^d$ we abbreviate the two tail quantities of the weight~$\beta$ from~\eqref{eq:weight_Hmix} by 
		\begin{align*}
			\tau_{\I}\coloneqq\sup_{\bk\in\N_0^d\setminus\I}\frac1{\beta(\bk)}
			\qquad\text{and}\qquad
			\rho_{\I}\coloneqq\sum_{\bk\in\N_0^d\setminus\I}\frac1{\beta(\bk)},
		\end{align*}
		where~$\rho_\I<\infty$ for~$s>\frac12$, cf.\ the discussion following Theorem~\ref{Thm:ANOVA_cheb}.

		\begin{theorem}[Error bound for the coupled approach]
			\label{Thm:err_coupled_MZ}
			Let~$m=1$, let~$F_j,G_j\in H^s_{\mathrm{mix}}([-1,1]^d)$ with~$s>\frac12$, and let~$\delta\in(0,1)$, $r>0$ and~$M>3$ satisfy~\eqref{eq:mz_condition}.
			Let~$\tilde H_j$ be the approximation~\eqref{eq:approx_H_rowwise} with the coefficients obtained from the least squares problem~\eqref{eq:opt_fullanova}, and abbreviate~$\|\cdot\|_{H^s_{\mathrm{mix}}}\coloneqq\|\cdot\|_{H^s_{\mathrm{mix}}([-1,1]^d)}$.
			Then, with probability at least~$1-3M^{-r}$,
			\begin{align}
				\label{eq:err_coupled_MZ}
				\begin{split}
					\|H_j-\tilde H_j\|^2_{L_2([-1,1]^d\times\U,\omega^{\mathrm{cheb}}\cdot\omega_{d+1})}
					&\leq\frac{2(2-\delta)}{1-\delta}\Big(\|F_j\|_{H^s_{\mathrm{mix}}}\,\tau_{\I(V_1)}^{1/2}+\sqrt\nu\,\|G_j\|_{H^s_{\mathrm{mix}}}\,\tau_{\I(V_2)}^{1/2}\Big)^2\\
					&\quad+\frac{2^{d+1}}{1-\delta}\sqrt{\frac{r\log M}{2M}}\Big(\|F_j\|_{H^s_{\mathrm{mix}}}\,\rho_{\I(V_1)}^{1/2}+u_{\max}\,\|G_j\|_{H^s_{\mathrm{mix}}}\,\rho_{\I(V_2)}^{1/2}\Big)^2 .
				\end{split}
			\end{align}
		\end{theorem}

		\begin{proof}
			The proof consists of four steps. We write~$\|\cdot\|_{L_2}$ for the norm of~$L_2([-1,1]^d\times\U,\omega^{\mathrm{cheb}}\cdot\omega_{d+1})$, $\|\cdot\|_{C}$ for the supremum norm on~$[-1,1]^d\times\U$, and
				$\| w \|^2_M \coloneqq \frac1M\sum_{i=1}^M|w(\bx^i,u^i)|^2 $
			for the squared discrete seminorm induced by the samples.

			\emph{Step 1.}
			With the orthogonal projections~\eqref{eq:projection}, we set
			\begin{align*}
				g\coloneqq P_{\I(V_1)}F_j+u\,P_{\I(V_2)}G_j\in\operatorname{span}\mathcal{\widetilde D}
				\qquad\text{and}\qquad
				e\coloneqq H_j-g=\big(F_j-P_{\I(V_1)}F_j\big)+u\,\big(G_j-P_{\I(V_2)}G_j\big).
			\end{align*}
			Note that~$g$ and~$e$ do not depend on the samples.
			For~$f\in H^s_{\mathrm{mix}}([-1,1]^d)$ and an index set~$\I$, Parseval's identity and~\eqref{eq:Hsmix_norm} yield
			\begin{align*}
				\|f-P_\I f\|^2_{L_2([-1,1]^d,\omega^{\mathrm{cheb}})}=\sum_{\bk\notin\I}|c_\bk(f)|^2\leq\tau_\I\,\|f\|^2_{H^s_{\mathrm{mix}}([-1,1]^d)},
			\end{align*}
			and a Cauchy--Schwarz inequality with weights~$\beta(\bk)^{1/2}\beta(\bk)^{-1/2}$ gives~$\sum_{\bk\notin\I}|c_\bk(f)|\leq\rho_\I^{1/2}\|f\|_{H^s_{\mathrm{mix}}([-1,1]^d)}$.
			In particular, $\sum_{\bk\in\N_0^d}|c_\bk(f)|<\infty$ for~$s>\frac12$, such that the Chebyshev series of~$f$ converges absolutely and uniformly, $f$ coincides with a continuous function (which is the representative being sampled), and~$f-P_\I f=\sum_{\bk\notin\I}c_\bk(f)\varphi_\bk$ holds pointwise.
			Since~$|\varphi_\bk(\bx)|\leq2^{|\supp\bk|/2}\leq2^{d/2}$ for all~$\bx\in[-1,1]^d$ by~\eqref{eq:basis_cheb}, the triangle inequality yields
			\begin{align*}
				\|f-P_\I f\|_{C([-1,1]^d)}\leq2^{d/2}\sum_{\bk\notin\I}|c_\bk(f)|\leq2^{d/2}\,\rho_\I^{1/2}\,\|f\|_{H^s_{\mathrm{mix}}([-1,1]^d)} .
			\end{align*}
			Moreover, for~$a,b\in L_2([-1,1]^d,\omega^{\mathrm{cheb}})$ we have~$\|a+ub\|^2_{L_2}=\|a\|^2+2\eta\langle a,b\rangle+\nu\|b\|^2\leq(\|a\|+\sqrt\nu\|b\|)^2$ with the norm and inner product of~$L_2([-1,1]^d,\omega^{\mathrm{cheb}})$, since~$\eta^2\leq\nu$ by the Cauchy--Schwarz inequality.
			Applying these three estimates to~$e$ yields
			\begin{align}
				\label{eq:mz_e_bounds}
				\begin{split}
			\|e\|_{L_2}&\leq\|F_j\|_{H^s_{\mathrm{mix}}}\tau_{\I(V_1)}^{1/2}+\sqrt\nu\,\|G_j\|_{H^s_{\mathrm{mix}}}\tau_{\I(V_2)}^{1/2},
\\
\|e\|_{C}&\leq2^{d/2}\Big(\|F_j\|_{H^s_{\mathrm{mix}}}\rho_{\I(V_1)}^{1/2}+u_{\max}\|G_j\|_{H^s_{\mathrm{mix}}}\rho_{\I(V_2)}^{1/2}\Big).
				\end{split}
			\end{align}

			\emph{Step 2.}
			On the event where~\eqref{eq:mz} holds, $\mathcal G_M$ is positive definite, hence~$\boldsymbol B$ has full column rank and~\eqref{eq:opt_fullanova} has a unique solution~$\bc^\star$.
			Let~$\boldsymbol\Pi\in\R^{M\times M}$ denote the orthogonal projector onto the range of~$\boldsymbol B$, let~$\bc_g$ be the coefficient vector of~$g$, and let~$\boldsymbol e\coloneqq(e(\bx^i,u^i))_{i=1}^M$.
			Then~$\boldsymbol B\bc^\star=\boldsymbol\Pi\boldsymbol h$ with the sample vector~$\boldsymbol h=(H_j(\bx^i,u^i))_{i=1}^M$ from~\eqref{eq:opt_fullanova} and~$\boldsymbol B\bc_g=\boldsymbol\Pi\boldsymbol B\bc_g$, such that~$\boldsymbol B(\bc^\star-\bc_g)=\boldsymbol\Pi(\boldsymbol h-\boldsymbol B\bc_g)=\boldsymbol\Pi\boldsymbol e$ and therefore
			\begin{align*}
				\|\tilde H_j-g\|_M=M^{-1/2}\|\boldsymbol B(\bc^\star-\bc_g)\|_2=M^{-1/2}\|\boldsymbol\Pi\boldsymbol e\|_2\leq M^{-1/2}\|\boldsymbol e\|_2=\|e\|_M .
			\end{align*}
			As~$\tilde H_j-g\in\operatorname{span}\mathcal{\widetilde D}$, the lower bound in~\eqref{eq:mz} gives~$\|\tilde H_j-g\|_{L_2}\leq(1-\delta)^{-1/2}\|\tilde H_j-g\|_M$, and the triangle inequality yields
			\begin{align}
				\label{eq:mz_chain}
				\begin{split}
					\|H_j-\tilde H_j\|_{L_2}&\leq\|e\|_{L_2}+\|g-\tilde H_j\|_{L_2}\leq\|e\|_{L_2}+(1-\delta)^{-1/2}\|e\|_M
					\qquad\text{and hence}\\
					\|H_j-\tilde H_j\|^2_{L_2}&\leq2\|e\|^2_{L_2}+\frac{2}{1-\delta}\|e\|^2_M .
				\end{split}
			\end{align}

			\emph{Step 3.}
			The quantity~$\|e\|_M^2$ is the sample mean of the i.i.d.\ random variables~$|e(\bx^i,u^i)|^2\in[0,\|e\|_C^2]$ with expectation~$\|e\|^2_{L_2}$.
			Hence, Hoeffding's inequality~\cite{Hoeffding63} yields~$\mathbb P\big(\|e\|_M^2-\|e\|_{L_2}^2\geq t\big)\leq\exp(-2Mt^2/\|e\|_C^4)$ for every~$t>0$, and the choice~$t\coloneqq\|e\|_C^2\sqrt{r\log M/(2M)}$ gives
			\begin{align}
				\label{eq:mz_hoeffding}
				\|e\|_M^2\leq\|e\|_{L_2}^2+\|e\|_C^2\sqrt{\frac{r\log M}{2M}}
			\end{align}
			with probability at least~$1-M^{-r}$ (if~$\|e\|_C=0$, there is nothing to show).

			\emph{Step 4: Conclusion.}
			By Boole's inequality, \eqref{eq:mz} and~\eqref{eq:mz_hoeffding} hold simultaneously with probability at least~$1-2M^{-r}-M^{-r}$.
			Inserting~\eqref{eq:mz_hoeffding} into~\eqref{eq:mz_chain} and using~$2+\frac2{1-\delta}=\frac{2(2-\delta)}{1-\delta}$, we obtain
			\begin{align*}
				\|H_j-\tilde H_j\|^2_{L_2}\leq\frac{2(2-\delta)}{1-\delta}\|e\|^2_{L_2}+\frac{2}{1-\delta}\|e\|_C^2\sqrt{\frac{r\log M}{2M}},
			\end{align*}
			and~\eqref{eq:mz_e_bounds} yields the assertion.
		\end{proof}

		We give a short interpretation of the estimate~\eqref{eq:err_coupled_MZ}, which features two central components.
		The first term is independent of~$M$ and bounds the squared distance~$\|e\|^2_{L_2}$ of~$H_j$ to the comparison function~$g$, i.\,e., the truncation and projection error of Remark~\ref{Remark:lsqr_error} in the coupled setting. Note that~$g$ consists of the $L_2$-projections of~$F_j$ and~$G_j$, i.\,e., of their separate ANOVA approximations in the limit~$M\to\infty$, cf.\ Remark~\ref{Remark:anova_separate}. For~$\eta=0$, the Gram matrix~$\mathcal G$ is block diagonal and~$g$ coincides with the orthogonal projection~$P_{\mathcal{\widetilde D}}H_j$ of Proposition~\ref{Prop:limit_projection}, such that~$\|e\|_{L_2}$ is the best-approximation error of~$H_j$ in~$\operatorname{span}\mathcal{\widetilde D}$ itself. The factor~$2(2-\delta)/(1-\delta)$ stems from the binomial estimate in~\eqref{eq:mz_chain} and from the lower bound in~\eqref{eq:mz}. The first term can only be reduced by enlarging the index sets, and the factor~$\sqrt\nu$ weights the contribution of~$G_j$ with the second moment of the control.
		The second term is the estimation error and stems from replacing the~$L_2$-norm of the fixed remainder~$e$ by its discrete counterpart, decays like~$(\log M/M)^{1/2}$ for fixed index sets, and is proportional to the squared supremum norm of~$e$. This is why it involves the~$\ell_1$-tails~$\rho_\I$ instead of the~$\ell_\infty$-tails~$\tau_\I$ and carries the factor~$2^d$.

		We briefly comment on the choice of the index sets in the above theorem. 
		\begin{remark}[Deterministic versus data-dependent index sets]
			\label{Remark:datadependence}
			Theorem~\ref{Thm:err_coupled_MZ} (and likewise Corollary~\ref{Cor:err_est_fullanova_m} below) is formulated for index sets that are fixed \emph{before} the data are drawn.
			This is consistent with the sampling condition, which is a statement about the random matrix~$\boldsymbol B$ associated with the prescribed dictionary~$\mathcal{\widetilde D}$, cf.\ Proposition~\ref{Prop:MZ} and~\cite{KaUlVo19,SchmischkeDiss}.
			In Algorithm~\ref{alg:anova}, however, the set~$V$ of relevant terms in~\eqref{eq:indexset_V} is determined from the global sensitivity indices of a first approximation, i.\,e., from the same data that are subsequently used for the final least squares fit.
			Strictly speaking, the resulting index sets are then random and depend on the samples, so that the cited concentration results are not directly applicable and the theorem does not cover this variant.
			Hence, we briefly discuss two possible remedies that are standard in the ANOVA literature. Either the truncation step is skipped and the structure~\eqref{eq:ANOVA_struct_F,G} is prescribed, in which case the theorem applies as stated. Alternatively, the data are split into two independent batches. The first one is used for the detection of the active sets, while the second, conditionally independent batch is used for the final fit. In this case, the theorem applies conditionally on the first batch.
		\end{remark}

		The above results are statements about the combined map~$H_j$. They also yield bounds for the drift and the input map separately, which we make precise in the following remark.
		\begin{remark}[Separation of drift and input map]
			\label{Remark:no_separation}
			First note that computation in Remark~\ref{Remark:schur} does not use that the functions~$g$ and~$h$ therein belong to the spans of the two systems. Hence, for all~$a,b\in L_2(\Omega,\omega)$, we have
			\begin{align*}
				\|a+b\,u\|^2_{L_2(\Omega\times\U,\omega\cdot\omega_{d+1})}
				=\|a\|^2_{L_2(\Omega,\omega)}+2\eta\,\langle a,b\rangle_{L_2(\Omega,\omega)}+\nu\,\|b\|^2_{L_2(\Omega,\omega)}
				\geq\lambda_-\big(\|a\|^2_{L_2(\Omega,\omega)}+\|b\|^2_{L_2(\Omega,\omega)}\big).
			\end{align*}
			Choosing~$a=F_j-\tilde F_j$ and~$b=G_j-\tilde G_j$ with~$\tilde F_j\coloneqq\sum_{\bk\in\I(V_1)}c^1_\bk\varphi^1_\bk$ and~$\tilde G_j\coloneqq\sum_{\bs\in\I(V_2)}c^2_\bs\varphi^2_\bs$.  yields
			\begin{align*}
				\|F_j-\tilde F_j\|^2_{L_2(\Omega,\omega)}+\|G_j-\tilde G_j\|^2_{L_2(\Omega,\omega)}\leq\frac1{\lambda_-}\,\|H_j-\tilde H_j\|^2_{L_2(\Omega\times\U,\omega\cdot\omega_{d+1})}.
			\end{align*}
			Consequently, Theorem~\ref{Thm:err_coupled_MZ} yields error bounds for~$F_j$ and~$G_j$ individually, with the same probability and the additional factor~$1/\lambda_-$.
			In this sense, the coupled ansatz separates the drift from the input map for finitely many samples under Assumption~\ref{Ass:nondegenerate}. 
		\end{remark}

		\subsection[]{Multivariate control}\label{sec:2d}

		For~$m>1$ note that the rows of~\eqref{eq:system_H} can be written as
		\begin{align*}
			H_j(\bx,\bu) = F_j(\bx) + \sum_{\ell=1}^{m} G_{j\ell}(\bx) u_\ell, \quad j=1,\dots,d,
		\end{align*}
		with~$\bu = (u_\ell)_{\ell=1}^m \in \R^m$.
		Thus, we generalize the above idea of~\eqref{eq:approx_H_rowwise} as
		\begin{align}
			\label{eq:approx_H_rowwise_m}
			H_j(\bx,\bu) \approx \sum_{\bk\in\I(V_1)} c_\bk^1 \varphi^1_\bk(\bx) + \sum_{\ell=1}^{m} \left( \sum_{\bs\in\I(V_{\ell+1})} c_\bs^{\ell+1} \varphi^{\ell+1}_\bs(\bx) \right) u_\ell, \quad j=1,\dots,d.
		\end{align}
		Again we sample the whole system~\eqref{eq:system_H} at i.i.d.\ data points~$(\bx^i,\bu^i)\in\Omega\times\U \subset \R^{d+m}$, $i=1,\dots,M$, drawn according to the tensor product density~$\omega \cdot \widetilde\omega$ with~$M\in\N$ and~$\widetilde\omega \coloneqq \prod_{\ell=1}^{m} \omega_{d+\ell}$ being the density corresponding to~$\U=\bigtimes_{\ell=1}^m \U_\ell\subset\R^m$.
		Analogously to~\eqref{eq:eta_nu}, we abbreviate the first two moments of the components of the control by
		\begin{align}
			\label{eq:eta_nu_m}
			\eta_\ell \coloneqq \langle 1,u_\ell \rangle_{L_2(\U_\ell,\omega_{d+\ell})}, \qquad
			\nu_\ell \coloneqq \|u_\ell\|^2_{L_2(\U_\ell,\omega_{d+\ell})},
			\quad \ell=1,\dots,m,
		\end{align}
		and replace Assumption~\ref{Ass:nondegenerate} by its multivariate counterpart.
		\begin{assumption}[Non-degenerate control distribution, multivariate case]
			\label{Ass:nondegenerate_m}
			The distribution of each component~$u_\ell$ of the control is non-degenerate, i.\,e., $\nu_\ell-\eta_\ell^2>0$ for~$\ell=1,\dots,m$.
		\end{assumption}
		Introduce the vectors
		\begin{align*}
			\bc^{\ell+1} &\coloneqq \left(c_\bk^{\ell+1}\right)_{\bk\in\I(V_{\ell+1})} \in \R^{|\I(V_{\ell+1})|}, \quad \ell=0,\dots,m,
		\end{align*}
		and matrices
		\begin{align}
			\label{eq:matrices_Aell}
			\bA_{\ell+1} &\coloneqq \left(\varphi^{\ell+1}_\bk(\bx^i)\right)_{i=1,\,\bk\in\I(V_{\ell+1})}^{M} \in \R^{M\times|\I(V_{\ell+1})|}, \quad \ell=0,\dots,m,
		\end{align}
		as well as%
		\begin{align}
			\label{eq:matrices_Uell}
			\bU_\ell \coloneqq \operatorname{diag}(u^1_\ell,\ldots,u^M_\ell)\in \R^{M\times M},
			\quad \ell=1,\dots,m.
		\end{align}
		Then, the same technique as pursued for the single-input case $m=1$ applies to the individual components of~$\bu^i$. For fixed~$j=1,\dots,d$, we obtain the matrix-vector form
		\begin{align*}
			\boldsymbol h = \left(H_j(\bx^i,\bu^i)\right)_{i=1}^M
			\approx
			\begin{pmatrix}
				\bA_1 & \bU_1 \bA_2 & \dots & \bU_m \bA_{m+1}
			\end{pmatrix}
			\begin{pmatrix}
				\bc^1\\
				\vdots\\
				\bc^{m+1}
			\end{pmatrix}.
		\end{align*}
		The least squares solution to this problem is again computed by means of the normal equations.
		Thus, we need to ensure that the system matrix has full column rank.
		Since we may assume that all~$\bA_{\ell+1}$ have full column rank by Section~\ref{sec:ANOVA_implementation}, we require an overdetermined system with~$M \geq |\I(V_1)|+\dots+|\I(V_{m+1})|$.
		Again, the scaled Gram matrix converges almost surely, entrywise, to the positive definite limit matrix of the following remark, provided that Assumption~\ref{Ass:nondegenerate_m} holds.
		Since the limit matrix is invertible, the event of full column rank has asymptotic probability one. A quantitative statement, which holds with high probability, follows from the multivariate analogue of Proposition~\ref{Prop:MZ}, see the proof of Corollary~\ref{Cor:err_est_fullanova_m} below.
		Concerning the conditioning, zero-mean and normalized control densities again yield the identity as limit matrix, cf.\ Section~\ref{sec:1d}.

		\begin{remark}
			In analogy to the proof of Proposition~\ref{Prop:limit_projection}, the normal equations converge after scaling by~$\frac1M$ for~$M\to\infty$ to the limit system
			\begin{align*}
				&
				\begin{pmatrix}
					\bI_{|\I(V_1)|} & \eta_1 \cdot \boldsymbol\Phi_{\{1,2\}} & \dots & \eta_m \cdot \boldsymbol\Phi_{\{1,m+1\}} \\
					\eta_1 \cdot \boldsymbol\Phi_{\{1,2\}}^\top & \nu_1 \cdot \bI_{|\I(V_2)|} & \dots & \eta_1 \eta_m \cdot \boldsymbol\Phi_{\{2,m+1\}} \\
					\vdots & \vdots & \ddots & \vdots \\
					\eta_m \cdot \boldsymbol\Phi_{\{1,m+1\}}^\top & \eta_1 \eta_m \cdot \boldsymbol\Phi_{\{2,m+1\}}^\top & \dots & \nu_m \cdot \bI_{|\I(V_{m+1})|}
				\end{pmatrix}  
				\begin{pmatrix}
					\bc^1\\
					\bc^2\\
					\vdots\\
					\bc^{m+1}
				\end{pmatrix} \notag \\[1ex]
				& =
				\begin{pmatrix}
					\Big(\langle F_j,\varphi^1_\bk \rangle_{L_2(\Omega,\omega)} + \sum_{n=1}^m \eta_n \,\langle G_{jn},\varphi^1_\bk \rangle_{L_2(\Omega,\omega)}\Big)_{\bk\in\I(V_1)}
					\\
					\Big(\eta_1\,\langle F_j,\varphi^2_{\boldsymbol s} \rangle_{L_2(\Omega,\omega)} + \nu_1\, \langle G_{j1},\varphi^2_{\boldsymbol s} \rangle_{L_2(\Omega,\omega)} + \eta_1 \sum_{n\neq 1} \eta_n\, \langle G_{jn},\varphi^2_{\boldsymbol s} \rangle_{L_2(\Omega,\omega)}\Big)_{\boldsymbol s\in\I(V_2)} \\
					\vdots \\
					\Big(\eta_m\,\langle F_j,\varphi^{m+1}_{\boldsymbol s} \rangle_{L_2(\Omega,\omega)} + \nu_m\, \langle G_{jm},\varphi^{m+1}_{\boldsymbol s} \rangle_{L_2(\Omega,\omega)} + \eta_m \sum_{n\neq m} \eta_n\, \langle G_{jn},\varphi^{m+1}_{\boldsymbol s} \rangle_{L_2(\Omega,\omega)}\Big)_{\boldsymbol s\in\I(V_{m+1})}
				\end{pmatrix}
			\end{align*}
			with the matrices
			\begin{align*}
				\boldsymbol\Phi_{\{\ell+1,n+1\}} \coloneqq \left(\langle \varphi^{\ell+1}_\bk,\varphi^{n+1}_{\boldsymbol s} \rangle_{L_2(\Omega,\omega)}\right)_{\bk\in\I(V_{\ell+1}),\,\boldsymbol s\in\I(V_{n+1})},
				\quad 0\leq\ell<n\leq m,
			\end{align*}
			and the moments~$\eta_\ell$, $\nu_\ell$ from~\eqref{eq:eta_nu_m}.
			This limit system again coincides with the linear system characterizing the orthogonal projection of~$H_j$ onto the span of the corresponding dictionary, cf.~\eqref{eq:convergence_neq} and the projection interpretation given there.
			Here, we used~\mbox{$H_j = F_j + \sum_{n=1}^m G_{jn}u_n$} and the independence of the components of~$\bu$, i.\,e., $\langle u_n,u_\ell\rangle = \eta_n\eta_\ell$ for~$n\neq\ell$. For~$m=1$, the system indeed reduces to~\eqref{eq:convergence_neq}.
			Concerning solvability, the interpretation as an orthogonal projection carries over from the proof of Proposition~\ref{Prop:limit_projection}. The limit system is the Gram system of the dictionary
			$\{\varphi^1_\bk\}_{\bk\in\I(V_1)} \cup \bigcup_{\ell=1}^m \{\varphi^{\ell+1}_{\boldsymbol s} u_\ell\}_{\boldsymbol s\in\I(V_{\ell+1})}$
			in~$L_2(\Omega\times\U,\omega\cdot\widetilde\omega)$.
			Under Assumption~\ref{Ass:nondegenerate_m}, i.\,e., if none of the~$u_\ell$ is $\omega_{d+\ell}$-almost everywhere constant, the functions~$1,u_1,\dots,u_m$ are linearly independent in~$L_2(\U,\widetilde\omega)$, as the~$u_\ell$ are independent and non-degenerate. Therefore, the dictionary is linearly independent, its Gram matrix is positive definite, and the limit system is uniquely solvable.
			Finally, if all control densities have zero mean, i.\,e.,~$\eta_\ell=0$ for~$\ell=1,\dots,m$, all off-diagonal blocks, which carry the factors~$\eta_\ell$ or~$\eta_\ell\eta_n$, vanish. The limit system then becomes block diagonal, and the coupled reconstruction coincides in the limit with the separate ANOVA approximations of~$F_j$ and~$G_{j\ell}$ for arbitrary basis systems, in complete analogy to Remark~\ref{Remark:anova_separate}. 
		\end{remark}

		The error analysis of Section~\ref{sec:error_coupled} carries over to multivariate controls. To formulate it, we denote by
			\begin{align}
				\label{eq:moment_matrix}
				\boldsymbol N\coloneqq\int_{\U}\begin{pmatrix}1\\ \bu\end{pmatrix}\begin{pmatrix}1\\ \bu\end{pmatrix}^{\!\top}\widetilde\omega(\bu)\,\mathrm d\bu
				=\begin{pmatrix}1&\boldsymbol\eta^\top\\ \boldsymbol\eta&\boldsymbol\eta\boldsymbol\eta^\top+\operatorname{diag}\big(\nu_1-\eta_1^2,\dots,\nu_m-\eta_m^2\big)\end{pmatrix}\in\R^{(m+1)\times(m+1)}
			\end{align}
			with $ \boldsymbol\eta\coloneqq(\eta_\ell)_{\ell=1}^m$
			the Gram matrix of the functions~$1,u_1,\dots,u_m$ in~$L_2(\U,\widetilde\omega)$, where the second identity follows from the independence of the control components.
			Its Schur complement with respect to the top left entry is the diagonal matrix of the variances~$\nu_\ell-\eta_\ell^2$, such that~$\boldsymbol N$ is positive definite if and only if Assumption~\ref{Ass:nondegenerate_m} holds.
			For~$m=1$, the eigenvalues of~$\boldsymbol N=\bigl(\begin{smallmatrix}1&\eta\\ \eta&\nu\end{smallmatrix}\bigr)$ are~$\lambda_\pm$ from~\eqref{eq:limit_eigenvalues}, and for zero-mean controls we have~$\boldsymbol N=\operatorname{diag}(1,\nu_1,\dots,\nu_m)$.

		\begin{corollary}[Error estimate for multivariate control]
			\label{Cor:err_est_fullanova_m}

			Let~$m\geq1$,  $j\in [d]$, let Assumption~\ref{Ass:nondegenerate_m} hold, and let~$\lambda_->0$ denote the smallest eigenvalue of~$\boldsymbol N$ in~\eqref{eq:moment_matrix}.
			Let~$F_j,G_{j1},\dots,G_{jm}\in H^s_{\mathrm{mix}}([-1,1]^d)$ with~$s>\frac12$, let the sample points be drawn i.i.d.\ according to~$\omega^{\mathrm{cheb}}\cdot\widetilde\omega$, and use the Chebyshev basis~\eqref{eq:basis_cheb} in all~$m+1$ systems of the ansatz~\eqref{eq:approx_H_rowwise_m} with prescribed index sets~$\I(V_1),\dots,\I(V_{m+1})$ of cardinalities~$n_1,\dots,n_{m+1}$.
			With~$u_{\max,\ell}\coloneqq\max_{\bu\in\U}|u_\ell|$, $\alpha\coloneqq\max\{|\supp\bk|\colon\bk\in\I(V_1)\cup\dots\cup\I(V_{m+1})\}$ and~$c_\delta$ from~\eqref{eq:mz_condition}, let~$\delta\in(0,1)$, $r>0$ and~$M>3$ satisfy
			\begin{align*}
				n_1+\sum_{\ell=1}^m u_{\max,\ell}^2\,n_{\ell+1}\leq\frac{c_\delta\,\lambda_-\,M}{2^\alpha\,(r+1)\log M}.
			\end{align*}
			Let~$\tilde H_j=\tilde F_j+\sum_{\ell=1}^m\tilde G_{j\ell}u_\ell$ be the approximation~\eqref{eq:approx_H_rowwise_m} with the coefficients obtained from the corresponding least squares problem, where~$\tilde F_j$ and~$\tilde G_{j\ell}$ collect the terms with the coefficients~$\bc^1$ and~$\bc^{\ell+1}$, respectively.
			Then, with probability at least~$1-3M^{-r}$, we have
			\begin{align*}
				\|H_j-\tilde H_j\|^2_{L_2([-1,1]^d\times\U,\omega^{\mathrm{cheb}}\cdot\widetilde\omega)}
				&\leq\frac{2(2-\delta)}{1-\delta}\Big(\|F_j\|_{H^s_{\mathrm{mix}}}\,\tau_{\I(V_1)}^{1/2}+\sum_{\ell=1}^m\sqrt{\nu_\ell}\,\|G_{j\ell}\|_{H^s_{\mathrm{mix}}}\,\tau_{\I(V_{\ell+1})}^{1/2}\Big)^2\\
				&\quad+\frac{2^{d+1}}{1-\delta}\sqrt{\frac{r\log M}{2M}}\Big(\|F_j\|_{H^s_{\mathrm{mix}}}\,\rho_{\I(V_1)}^{1/2}+\sum_{\ell=1}^m u_{\max,\ell}\,\|G_{j\ell}\|_{H^s_{\mathrm{mix}}}\,\rho_{\I(V_{\ell+1})}^{1/2}\Big)^2
			\end{align*}
			with the abbreviations of Theorem~\ref{Thm:err_coupled_MZ}, as well as
			\begin{align*}
				\|F_j-\tilde F_j\|^2_{L_2([-1,1]^d,\omega^{\mathrm{cheb}})}+\sum_{\ell=1}^m\|G_{j\ell}-\tilde G_{j\ell}\|^2_{L_2([-1,1]^d,\omega^{\mathrm{cheb}})}\leq\frac1{\lambda_-}\,\|H_j-\tilde H_j\|^2_{L_2([-1,1]^d\times\U,\omega^{\mathrm{cheb}}\cdot\widetilde\omega)}.
			\end{align*}

		\end{corollary}

		\begin{proof}
			We indicate the changes in the proofs of Proposition~\ref{Prop:MZ}, Theorem~\ref{Thm:err_coupled_MZ} and Remark~\ref{Remark:no_separation}.
			Let~$a$, $b_1,\dots,b_m\in L_2([-1,1]^d,\omega^{\mathrm{cheb}})$ and~$\boldsymbol w\coloneqq(a,b_1,\dots,b_m)^\top$. Then, Fubini's theorem yields
			\begin{align}
				\label{eq:moment_bound}
				\Big\|a+\sum_{\ell=1}^m b_\ell u_\ell\Big\|^2_{L_2([-1,1]^d\times\U,\omega^{\mathrm{cheb}}\cdot\widetilde\omega)}
				=\int_{[-1,1]^d}\boldsymbol w(\bx)^\top\boldsymbol N\,\boldsymbol w(\bx)\,\omega^{\mathrm{cheb}}(\bx)\,\mathrm d\bx
				\geq\lambda_-\Big(\|a\|^2+\sum_{\ell=1}^m\|b_\ell\|^2\Big)
			\end{align}
			with the norms of~$L_2([-1,1]^d,\omega^{\mathrm{cheb}})$ on the right-hand side. Here, we used that
			\begin{align*}
				\int_\U\big|\boldsymbol w(\bx)^\top(1,\bu^\top)^\top\big|^2\,\widetilde\omega(\bu)\,\mathrm d\bu=\boldsymbol w(\bx)^\top\boldsymbol N\boldsymbol w(\bx)\geq\lambda_-\|\boldsymbol w(\bx)\|_2^2
				\qquad\text{for every }\bx\in[-1,1]^d.
			\end{align*}
			For~$a$ and~$b_\ell$ in the spans of the respective systems, the right-hand side equals~$\lambda_-$ times the squared Euclidean norm of the coefficient vector by orthonormality. Hence, the Gram matrix~$\mathcal G$ of the dictionary of~\eqref{eq:approx_H_rowwise_m} satisfies~$\mathcal G\succeq\lambda_-\bI$, which replaces~\eqref{eq:limit_eigenvalues} in the proof of Proposition~\ref{Prop:MZ}.
			Moreover, the vector~$\boldsymbol v(\bz)$ now consists of the~$\varphi_\bk(\bx)$, $\bk\in\I(V_1)$, and the~$\varphi_\bs(\bx)\,u_\ell$, $\bs\in\I(V_{\ell+1})$, $\ell=1,\dots,m$, such that~$\|\boldsymbol v(\bz)\|_2^2\leq2^\alpha\big(n_1+\sum_{\ell=1}^m u_{\max,\ell}^2\,n_{\ell+1}\big)$, and the remaining steps of the proof of Proposition~\ref{Prop:MZ} carry over.
			In the proof of Theorem~\ref{Thm:err_coupled_MZ}, we set~$g\coloneqq P_{\I(V_1)}F_j+\sum_{\ell=1}^m u_\ell\,P_{\I(V_{\ell+1})}G_{j\ell}$. The triangle inequality together with~$\|u_\ell b\|_{L_2([-1,1]^d\times\U,\omega^{\mathrm{cheb}}\cdot\widetilde\omega)}=\sqrt{\nu_\ell}\,\|b\|_{L_2([-1,1]^d,\omega^{\mathrm{cheb}})}$ and~$|u_\ell|\leq u_{\max,\ell}$ yields the analogues of~\eqref{eq:mz_e_bounds}, and Steps~2--4 follow analogously.
			Finally, the second estimate follows from~\eqref{eq:moment_bound} with~$a=F_j-\tilde F_j$ and~$b_\ell=G_{j\ell}-\tilde G_{j\ell}$.

		\end{proof}

		\section{Numerical Examples\label{sec:numerics}}

		In this section, we illustrate the coupled approach of Section~\ref{sec:control} by means of four examples. The first one is an academic test function with prescribed low-dimensional structure. The other three are discretized PDEs, namely a nonlinear advection-diffusion equation, the shallow water equations, and the viscous Burgers equation of Example~\ref{ex:burgers}.
		For each example, we illustrate that the relevant ANOVA terms of~$F$ and~$G$ are identified from the approximated sensitivity indices~\eqref{eq:gsis_approx} and analyze the accuracy of the approximation. 

		Throughout, we proceed rowwise on a single component~$H_j$ of~\eqref{eq:system_H}. We use the Chebyshev basis~\eqref{eq:basis_cheb} with i.i.d.\ samples from the Chebyshev density, and apply the two-stage procedure of Algorithm~\ref{alg:anova} with threshold~\mbox{$\varepsilon = 1/|V_{\mathrm{cand}}|$}, where~$V_{\mathrm{cand}}$ is the family of non-constant candidate terms. For~$d=8$ and~$q=2$, this yields~$\varepsilon=1/73$ for~$m=1$ and~$\varepsilon=1/110$ for~$m=2$.
		The coupled approach is realized via the preprocessed index sets of Algorithm~\ref{alg:preprocessed}, such that the ansatz~\eqref{eq:approx_H_rowwise_m} consists of~$m+1$ groups of terms with the same number of basis functions each. Correspondingly, we choose the number~$M$ of sample points~$(\bx^i,\bu^i)$ proportional to~$m+1$, such that the ratio of sample points to unknowns does not depend on~$m$.
		Since the coupled approach yields combined indices for~$h(\bz)=H_j(\bx,\bu)$, we split them for the visualization into a group of terms that depend only on~$\bx$ and groups of terms containing one~$u_\ell$, where the term depending on~$u_\ell$ only, i.\,e., the constant part of~$G_\ell$, is not displayed. We then renormalize the indices within each group, whereas the truncation itself is carried out with the combined indices and the threshold~$\varepsilon$.
		In all figures depicting sensitivity indices, we show the approximated indices with the terms for~$F$ on the left and those for the input maps~$G_\ell$ on the right.
		The computations were carried out with the package~\cite{ANOVAapprox,pyANOVAapprox}.

		Reconstruction errors are estimated by Monte-Carlo quadrature,
		\begin{align}
			\label{eq:mc_error}
			\|H_j-\tilde H_j\|_{L_2(\Omega\times\U,\omega \cdot \widetilde\omega)}
			\approx \left(\frac 1{M_{\mathrm{test}}} \sum_{i=1}^{M_{\mathrm{test}}} \big|H_j(\bx^i,\bu^i)-\tilde H_j(\bx^i,\bu^i)\big|^2\right)^{1/2},
		\end{align}
		using~$M_{\mathrm{test}}=200$ fresh points that are independent of the training data. In the following, $e$ denotes this error.
		Unless stated otherwise, we use~$\Omega=[-1,1]^d$, $\U=[-1,1]^m$, $d=8$, $q=2$, $\bN=(6,4)$, $M=10\,000\,(m+1)$, and no regularization. States and controls are drawn from the Chebyshev density, which gives~$\eta_\ell=0$, $\nu_\ell=\frac12$, \mbox{$\max_{\bu\in\U}|u_\ell|^2=1$}, and hence~$\lambda_-=\frac12$ in Proposition~\ref{Prop:MZ} and Corollary~\ref{Cor:err_est_fullanova_m}.

The Julia code that reproduces all numerical experiments and figures of this section is available at~
\begin{center}
    \url{https://gitlab.hrz.tu-chemnitz.de/manuels-at-tu-chemnitz.de/anova-for-control}.
\end{center}

		\subsection{Academic test function\label{sec:toy}}
		As a first test case we consider the example function
		\begin{align*}
			H_j(\bx,u) = (x_4-x_1)\,x_2 - x_3u
		\end{align*}
		with~$d=8$ and~$m=1$. This corresponds to one row of~\eqref{eq:system_H}, with~\mbox{$F_j(\bx)=(x_4-x_1)\,x_2$} and~$G_{j}(\bx)=-x_3$ in the notation of~\eqref{eq:ANOVA_struct_F,G}. Hence, the index sets to be identified are $V_{j,0}=\{\{1,2\},\{2,4\}\}$ and $V_{j,1}=\{\{3\}\}$, while the remaining four variables~$x_5,\dots,x_8$ are uninformative.
		Since this example function contains two-dimensional terms at maximum, it is sufficient to choose~$q=2$.

		\begin{figure}[htb]
			\centering
			\input{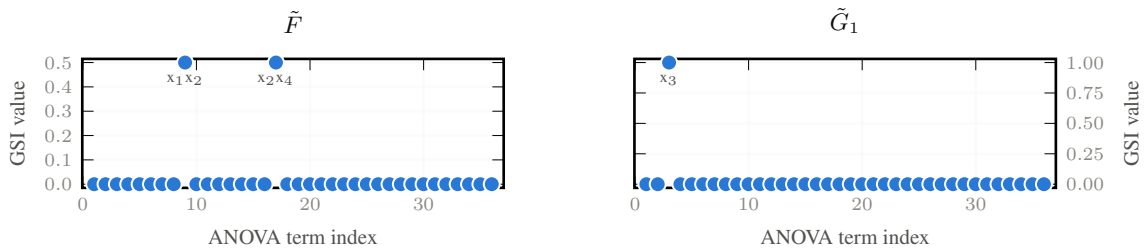}
			\caption{Approximated global sensitivity indices~\eqref{eq:gsis_approx} for the \emph{academic test function} with~$m=1$, where the relevant terms are~$\{1,2\}$, $\{2,4\}$ of~$F$ and~$\{3\}$ of~$G_1$.
				\label{fig:gsis}}
		\end{figure}
		In Figure~\ref{fig:gsis}, we depict the approximated global sensitivity indices.
		We observe that the relevant index sets are correctly identified, as expected for a function with prescribed structure. The resulting reconstruction is accurate as well. The $L_2$-error~\eqref{eq:mc_error} is~$e = 1.50\cdot 10^{-9}$, i.\,e., the function is reproduced up to the accuracy of the least squares solver.

		Moreover, we study the influence of the superposition dimension~$q$ (truncation error) and of the number of samples~$M$ (estimation error), cf.\ Remark~\ref{Remark:lsqr_error}. The corresponding results are depicted in Figure~\ref{fig:error_toy}, where the bandwidths are kept at the default~\mbox{$\bN=(6,4)$}.
		Since the example consists of ANOVA terms of order at most two, the truncation error vanishes for~$q\geq2$, as confirmed by the numerical results. 
		When varying~$M\in(m+1)\cdot\{10,20,50,\dots,50\,000\}$, the error remains above~$6\cdot10^{-1}$ up to~$M=100$, drops below~$10^{-6}$ at~$M=200$, and reaches the linear solver accuracy of~$10^{-9}$--$10^{-8}$ for larger~$M$.

		\begin{figure}[htb]
			\centering
			\mbox{

\begin{tikzpicture}[/tikz/background rectangle/.style={fill={rgb,1:red,1.0;green,1.0;blue,1.0}, fill opacity={1.0}, draw opacity={1.0}}, show background rectangle]
\begin{axis}[point meta max={nan}, point meta min={nan}, legend cell align={left}, legend columns={1}, legend style={color={rgb,1:red,0.0;green,0.0;blue,0.0}, draw opacity={1.0}, line width={1}, solid, fill={rgb,1:red,1.0;green,1.0;blue,1.0}, fill opacity={1.0}, text opacity={1.0}, font={{\fontsize{5 pt}{6 pt}\selectfont}}, text={rgb,1:red,0.3216;green,0.3176;blue,0.3059}, cells={anchor={center}}, at={(0.5, -0.62)}, anchor={north}}, axis background/.style={fill={rgb,1:red,1.0;green,1.0;blue,1.0}, opacity={1.0}}, anchor={north west}, xshift={0.4109mm}, yshift={-0.4044mm}, width={68.0000mm}, height={35.0000mm}, scaled x ticks={false}, xlabel={superposition dimension $q$}, x tick style={color={rgb,1:red,0.0;green,0.0;blue,0.0}, opacity={1.0}}, x tick label style={color={rgb,1:red,0.0;green,0.0;blue,0.0}, opacity={1.0}, rotate={0}}, xlabel style={at={(ticklabel cs:0.5)}, anchor=near ticklabel, at={{(ticklabel cs:0.5)}}, anchor={near ticklabel}, font={{\fontsize{8 pt}{10.4 pt}\selectfont}}, color={rgb,1:red,0.3216;green,0.3176;blue,0.3059}, draw opacity={1.0}, rotate={0.0}}, xmajorgrids={true}, xmin={0.9099999999999999}, xmax={4.09}, xticklabels={{$1$,$2$,$3$,$4$}}, xtick={{1.0,2.0,3.0,4.0}}, xtick align={inside}, xticklabel style={font={{\fontsize{7 pt}{9.1 pt}\selectfont}}, color={rgb,1:red,0.5373;green,0.5294;blue,0.5059}, draw opacity={1.0}, rotate={0.0}}, x grid style={color={rgb,1:red,0.8824;green,0.8784;blue,0.851}, draw opacity={0.1}, line width={1}, solid}, xticklabel pos={left}, x axis line style={color={rgb,1:red,0.0;green,0.0;blue,0.0}, draw opacity={1.0}, line width={1}, solid}, scaled y ticks={false}, ylabel={$L_2\ \mathrm{error}$}, y tick style={color={rgb,1:red,0.0;green,0.0;blue,0.0}, opacity={1.0}}, y tick label style={color={rgb,1:red,0.0;green,0.0;blue,0.0}, opacity={1.0}, rotate={0}}, ylabel style={at={(ticklabel cs:0.5)}, anchor=near ticklabel, at={{(ticklabel cs:0.5)}}, anchor={near ticklabel}, font={{\fontsize{8 pt}{10.4 pt}\selectfont}}, color={rgb,1:red,0.3216;green,0.3176;blue,0.3059}, draw opacity={1.0}, rotate={0.0}}, ymode={log}, log basis y={10}, ymajorgrids={true}, ymin={8.135512131092154e-10}, ymax={1.2497803720365674}, yticklabels={{$10^{-5}$,$10^{0}$}}, ytick={{1.0e-5,1.0}}, ytick align={inside}, yticklabel style={font={{\fontsize{7 pt}{9.1 pt}\selectfont}}, color={rgb,1:red,0.5373;green,0.5294;blue,0.5059}, draw opacity={1.0}, rotate={0.0}}, y grid style={color={rgb,1:red,0.8824;green,0.8784;blue,0.851}, draw opacity={0.1}, line width={1}, solid}, yticklabel pos={left}, y axis line style={color={rgb,1:red,0.0;green,0.0;blue,0.0}, draw opacity={1.0}, line width={1}, solid}, colorbar={false}]
    \addplot[color={rgb,1:red,0.8889;green,0.4356;blue,0.2781}, name path={2}, draw opacity={1.0}, line width={1}, solid, mark={*}, mark size={2.25 pt}, mark repeat={1}, mark options={color={rgb,1:red,0.0;green,0.0;blue,0.0}, draw opacity={1.0}, fill={rgb,1:red,0.8889;green,0.4356;blue,0.2781}, fill opacity={1.0}, line width={0.75}, rotate={0}, solid}]
        table[row sep={\\}]
        {
            \\
            1.0  0.6831960108191952  \\
            2.0  1.5711769803428243e-9  \\
            3.0  1.4803988226864869e-9  \\
            4.0  2.1400451874906584e-9  \\
        }
        ;
\end{axis}
\end{tikzpicture}\hspace{3ex}

\begin{tikzpicture}[/tikz/background rectangle/.style={fill={rgb,1:red,1.0;green,1.0;blue,1.0}, fill opacity={1.0}, draw opacity={1.0}}, show background rectangle]
\begin{axis}[point meta max={nan}, point meta min={nan}, legend cell align={left}, legend columns={1}, legend style={color={rgb,1:red,0.0;green,0.0;blue,0.0}, draw opacity={1.0}, line width={1}, solid, fill={rgb,1:red,1.0;green,1.0;blue,1.0}, fill opacity={1.0}, text opacity={1.0}, font={{\fontsize{5 pt}{6 pt}\selectfont}}, text={rgb,1:red,0.3216;green,0.3176;blue,0.3059}, cells={anchor={center}}, at={(0.5, -0.62)}, anchor={north}}, axis background/.style={fill={rgb,1:red,1.0;green,1.0;blue,1.0}, opacity={1.0}}, anchor={north west}, xshift={0.4107mm}, yshift={-0.4043mm}, width={68.0000mm}, height={35.0000mm}, scaled x ticks={false}, xlabel={number of samples $M$}, x tick style={color={rgb,1:red,0.0;green,0.0;blue,0.0}, opacity={1.0}}, x tick label style={color={rgb,1:red,0.0;green,0.0;blue,0.0}, opacity={1.0}, rotate={0}}, xlabel style={at={(ticklabel cs:0.5)}, anchor=near ticklabel, at={{(ticklabel cs:0.5)}}, anchor={near ticklabel}, font={{\fontsize{8 pt}{10.4 pt}\selectfont}}, color={rgb,1:red,0.3216;green,0.3176;blue,0.3059}, draw opacity={1.0}, rotate={0.0}}, xmode={log}, log basis x={10}, xmajorgrids={true}, xmin={15.49033804788709}, xmax={129112.74071728889}, xticklabels={{$10^{2}$,$10^{4}$}}, xtick={{100.0,10000.0}}, xtick align={inside}, xticklabel style={font={{\fontsize{7 pt}{9.1 pt}\selectfont}}, color={rgb,1:red,0.5373;green,0.5294;blue,0.5059}, draw opacity={1.0}, rotate={0.0}}, x grid style={color={rgb,1:red,0.8824;green,0.8784;blue,0.851}, draw opacity={0.1}, line width={1}, solid}, xticklabel pos={left}, x axis line style={color={rgb,1:red,0.0;green,0.0;blue,0.0}, draw opacity={1.0}, line width={1}, solid}, scaled y ticks={false}, ylabel={$L_2\ \mathrm{error}$}, y tick style={color={rgb,1:red,0.0;green,0.0;blue,0.0}, opacity={1.0}}, y tick label style={color={rgb,1:red,0.0;green,0.0;blue,0.0}, opacity={1.0}, rotate={0}}, ylabel style={at={(ticklabel cs:0.5)}, anchor=near ticklabel, at={{(ticklabel cs:0.5)}}, anchor={near ticklabel}, font={{\fontsize{8 pt}{10.4 pt}\selectfont}}, color={rgb,1:red,0.3216;green,0.3176;blue,0.3059}, draw opacity={1.0}, rotate={0.0}}, ymode={log}, log basis y={10}, ymajorgrids={true}, ymin={5.205006361761529e-10}, ymax={1.6015246846874094}, yticklabels={{$10^{-5}$,$10^{0}$}}, ytick={{1.0e-5,1.0}}, ytick align={inside}, yticklabel style={font={{\fontsize{7 pt}{9.1 pt}\selectfont}}, color={rgb,1:red,0.5373;green,0.5294;blue,0.5059}, draw opacity={1.0}, rotate={0.0}}, y grid style={color={rgb,1:red,0.8824;green,0.8784;blue,0.851}, draw opacity={0.1}, line width={1}, solid}, yticklabel pos={left}, y axis line style={color={rgb,1:red,0.0;green,0.0;blue,0.0}, draw opacity={1.0}, line width={1}, solid}, colorbar={false}]
    \addplot[color={rgb,1:red,0.8889;green,0.4356;blue,0.2781}, name path={2}, draw opacity={1.0}, line width={1}, solid, mark={*}, mark size={2.25 pt}, mark repeat={1}, mark options={color={rgb,1:red,0.0;green,0.0;blue,0.0}, draw opacity={1.0}, fill={rgb,1:red,0.8889;green,0.4356;blue,0.2781}, fill opacity={1.0}, line width={0.75}, rotate={0}, solid}]
        table[row sep={\\}]
        {
            \\
            20.0  0.8221215322212242  \\
            40.0  0.8139138143431505  \\
            100.0  0.6804479704820462  \\
            200.0  1.1798039431898881e-7  \\
            400.0  4.184691787525915e-8  \\
            1000.0  1.197826010646155e-8  \\
            2000.0  5.8440273166427115e-9  \\
            4000.0  8.177187035795926e-9  \\
            10000.0  1.0685303424344645e-8  \\
            20000.0  2.2730360998888735e-9  \\
            40000.0  5.09428934197538e-9  \\
            100000.0  1.6831442355123002e-9  \\
        }
        ;
\end{axis}
\end{tikzpicture}}
			\caption{Error~$e$ for the \emph{academic test function} depending on the superposition dimension~$q$ (left) and the number of samples~$M$ (right).
				\label{fig:error_toy}}
		\end{figure}
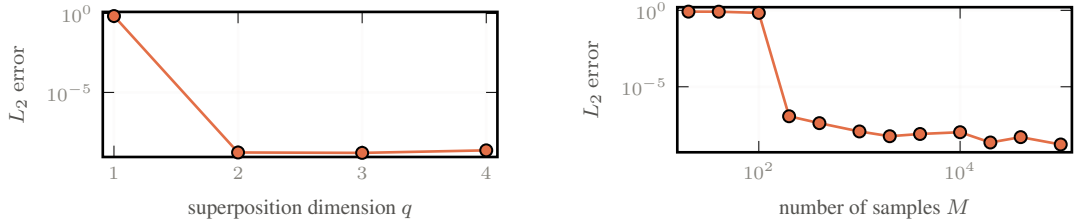

		\subsection{Diffusion equation with nonlinear advection\label{sec:advection}}
		As a second example, we consider the nonlinear advection-diffusion equation
		\begin{equation}
			\frac{\partial}{\partial t}x \;=\; D\,\frac{\partial^2}{\partial \omega^2}x \;+\; \big[x(1-x)\big]\,\frac{\partial}{\partial \omega}x,
			\qquad \omega \in(0,L), \; t>0.
			\label{eq:reaction-gradient}
		\end{equation}
		Here, $D>0$ denotes the diffusion coefficient and~$x(\omega,t)$ represents, e.\,g., the density of a biological population. The first term on the right-hand side accounts for diffusion. The second term combines the logistic nonlinearity~$R(x)=x(1-x)$ with the spatial gradient to form a state-dependent drift, which is strongest at intermediate concentrations and vanishes for~$x\in\{0,1\}$.

		Using a finite-difference approximation of the spatial derivatives with step~$\Delta \omega > 0$ and an explicit Euler method with time step~$\Delta t > 0$, we obtain the dynamics
		\begin{align*}
			X_i^+ &= X_i + \Delta t\left(D\,\frac{X_{i+1} - 2X_i + X_{i-1}}{\Delta \omega^2} + X_i (1-X_i)\,\frac{X_{i+1}-X_{i-1}}{2\Delta \omega} \right)\\[1.2ex]
			&= f_1(X_i) + f_2(X_{i+1}) + f_3(X_{i-1}) + f_4(X_i,X_{i+1}) + f_5(X_i,X_{i-1})
		\end{align*}
		at spatial point~$i$, for suitably chosen~$f_n$, $n=1,\ldots,5$.

		\medskip
		We now consider two different ways of adding a control to this system, which differ in how the control enters the dynamics and, consequently, in the structure of the resulting input map~$G$.

		\bigskip\noindent\emph{Gradient-coupled control.} First, we add a term~$\bigl(\frac{\partial x}{\partial \omega}\bigr)^2 u$, which acts most strongly where the spatial gradient is steep and results in		
		\begin{align}
			\label{eq:gradient_control}
			X_i^+ &= X_i + \Delta t\left(D\,\frac{X_{i+1} - 2X_i + X_{i-1}}{\Delta \omega^2} + X_i (1-X_i)\,\frac{X_{i+1}-X_{i-1}}{2\Delta \omega}   + \left(\frac{X_{i+1}-X_{i-1}}{2\Delta \omega}\right)^2U_i \right).
		\end{align}
		For this example, the threshold~$\varepsilon$ for the global sensitivity indices would have to be chosen considerably smaller to capture all relevant terms, see Figure~\ref{fig:advection_02}, where the indices are renormalized within each group. The weak couplings~$\{1\}$ and~$\{3\}$ of~$G_1$ have combined indices of~$2.2\cdot10^{-3}$ each, which is below~$\varepsilon=1/73\approx1.4\cdot10^{-2}$, such that they are truncated in the second least squares step. As a result, the error~$e = 8.87\cdot 10^{-2}$ is considerably larger than for the academic test function.
		\begin{figure}[htb]
			\centering
			\input{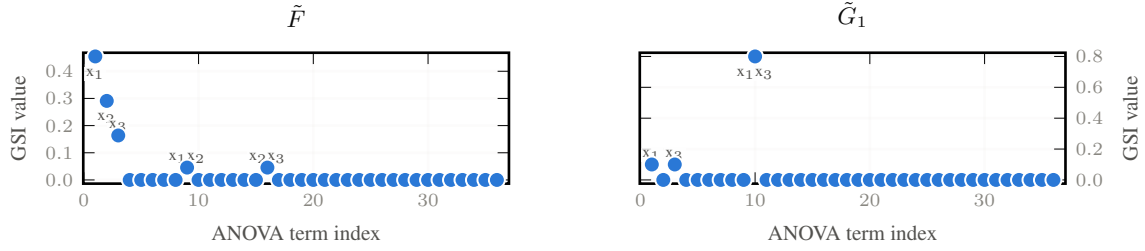}
			\caption{Approximated global sensitivity indices for the \emph{diffusion equation with gradient-coupled control}~\eqref{eq:gradient_control}, for~$i=2$ and~\mbox{$D=\Delta t=\Delta \omega=1$}, where the relevant terms are~$\{1\}, \{2\}, \{3\}, \{1,2\}, \{2,3\}$ of~$F$ and~$\{1\}, \{3\}, \{1,3\}$ of~$G_1$.
				\label{fig:advection_02}}
		\end{figure}

		\bigskip\noindent\emph{Controlled nonlinear advection.} Second, we multiply the advection term by~$u$, i.\,e., we consider
		\begin{align}
			\label{eq:nonlin_adv}
			X_i^+ &= X_i + \Delta t\left(D\,\frac{X_{i+1} - 2X_i + X_{i-1}}{\Delta \omega^2} + X_i (1-X_i)\,\frac{X_{i+1}-X_{i-1}}{2\Delta \omega} U_i \right).
		\end{align}
		Here, the advection term is multiplied by the control, such that the second-order interactions move from~$F$ to~$G_1$. Consequently, the relevant terms of~$F$ reduce to the three one-dimensional diffusion terms, while~$G_1$ contains the pair couplings~$\{1,2\}$ and~$\{2,3\}$ together with the weaker terms~$\{1\}$ and~$\{3\}$, see Figure~\ref{fig:advection_04}.
		As for~\eqref{eq:gradient_control}, these weak one-dimensional terms have combined indices of~$9.7\cdot10^{-3}$ each, fall below the threshold~$\varepsilon=1/73$, and are therefore truncated. As a result, we obtain an error~$e = 1.76\cdot 10^{-1}$.
		\begin{figure}[htb]
			\centering
			\input{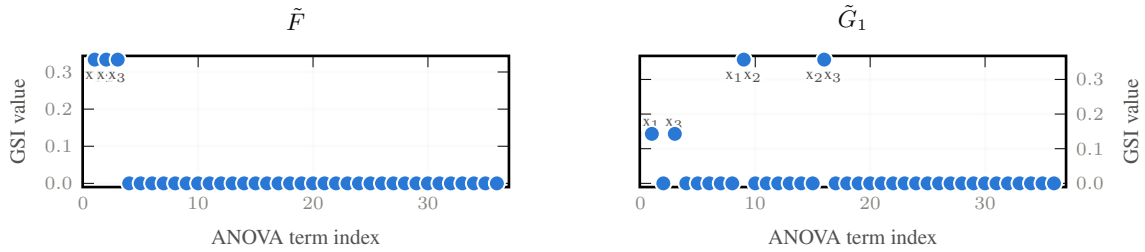}
			\caption{Approximated global sensitivity indices for the \emph{diffusion equation with controlled nonlinear advection~\eqref{eq:nonlin_adv}}, for~$i=2$ and~\mbox{$D=\Delta t=\Delta \omega=1$}, where the relevant terms are~$\{1\}, \{2\}, \{3\}$ of~$F$ and~$\{1\}, \{3\}, \{1,2\}, \{2,3\}$ of~$G_1$.
				\label{fig:advection_04}}
		\end{figure}

		\bigskip
		Finally, Figure~\ref{fig:error_advection} shows the influence of superposition dimension~$q$ and number of samples~$M$ for the two controlled systems~\eqref{eq:gradient_control} and~\eqref{eq:nonlin_adv}.
		Both require~$q\geq2$, since both dynamics contain pair interactions of neighbouring grid points. Hence, for~$q=1$, the relevant pair interactions cannot be represented and the error stagnates at~order~$10^{-1}$.
		At~$q=2$, however, the error remains large, since the weak couplings of~$G$ discussed above fall below the sparsity threshold and are truncated. For larger~$q$, the threshold~$\varepsilon=1/|V_{\mathrm{cand}}|$ decreases as the number of candidate terms increases, such that the weak couplings are eventually retained and the error drops to the solver accuracy, for~\eqref{eq:nonlin_adv} at~$q=3$ and for~\eqref{eq:gradient_control} at~$q=6$. The same effect renders~$e$ independent of~$M$. It settles at~$\approx 9.1\cdot10^{-2}$ and~$\approx1.9\cdot10^{-1}$ for~\eqref{eq:gradient_control} and~\eqref{eq:nonlin_adv}, respectively, and does not change up to~$M=10^5$, since increasing the number of samples cannot compensate for a term that is not contained in the ansatz space.

		\begin{figure}[htb]
			\centering
			\mbox{

\begin{tikzpicture}[/tikz/background rectangle/.style={fill={rgb,1:red,1.0;green,1.0;blue,1.0}, fill opacity={1.0}, draw opacity={1.0}}, show background rectangle]
\begin{axis}[point meta max={nan}, point meta min={nan}, legend cell align={left}, legend columns={1}, legend style={color={rgb,1:red,0.0;green,0.0;blue,0.0}, draw opacity={1.0}, line width={1}, solid, fill={rgb,1:red,1.0;green,1.0;blue,1.0}, fill opacity={1.0}, text opacity={1.0}, font={{\fontsize{5 pt}{6 pt}\selectfont}}, text={rgb,1:red,0.3216;green,0.3176;blue,0.3059}, cells={anchor={center}}, at={(0.5, -0.62)}, anchor={north}}, axis background/.style={fill={rgb,1:red,1.0;green,1.0;blue,1.0}, opacity={1.0}}, anchor={north west}, xshift={0.4109mm}, yshift={-0.4044mm}, width={68.0000mm}, height={35.0000mm}, scaled x ticks={false}, xlabel={superposition dimension $q$}, x tick style={color={rgb,1:red,0.0;green,0.0;blue,0.0}, opacity={1.0}}, x tick label style={color={rgb,1:red,0.0;green,0.0;blue,0.0}, opacity={1.0}, rotate={0}}, xlabel style={at={(ticklabel cs:0.5)}, anchor=near ticklabel, at={{(ticklabel cs:0.5)}}, anchor={near ticklabel}, font={{\fontsize{8 pt}{10.4 pt}\selectfont}}, color={rgb,1:red,0.3216;green,0.3176;blue,0.3059}, draw opacity={1.0}, rotate={0.0}}, xmajorgrids={true}, xmin={0.8499999999999996}, xmax={6.15}, xticklabels={{$1$,$2$,$3$,$4$,$5$,$6$}}, xtick={{1.0,2.0,3.0,4.0,5.0,6.0}}, xtick align={inside}, xticklabel style={font={{\fontsize{7 pt}{9.1 pt}\selectfont}}, color={rgb,1:red,0.5373;green,0.5294;blue,0.5059}, draw opacity={1.0}, rotate={0.0}}, x grid style={color={rgb,1:red,0.8824;green,0.8784;blue,0.851}, draw opacity={0.1}, line width={1}, solid}, xticklabel pos={left}, x axis line style={color={rgb,1:red,0.0;green,0.0;blue,0.0}, draw opacity={1.0}, line width={1}, solid}, scaled y ticks={false}, ylabel={$L_2\ \mathrm{error}$}, y tick style={color={rgb,1:red,0.0;green,0.0;blue,0.0}, opacity={1.0}}, y tick label style={color={rgb,1:red,0.0;green,0.0;blue,0.0}, opacity={1.0}, rotate={0}}, ylabel style={at={(ticklabel cs:0.5)}, anchor=near ticklabel, at={{(ticklabel cs:0.5)}}, anchor={near ticklabel}, font={{\fontsize{8 pt}{10.4 pt}\selectfont}}, color={rgb,1:red,0.3216;green,0.3176;blue,0.3059}, draw opacity={1.0}, rotate={0.0}}, ymode={log}, log basis y={10}, ymajorgrids={true}, ymin={1.6233220790973577e-9}, ymax={0.7234250096361489}, yticklabels={{$10^{-6}$,$10^{-3}$}}, ytick={{1.0e-6,0.001}}, ytick align={inside}, yticklabel style={font={{\fontsize{7 pt}{9.1 pt}\selectfont}}, color={rgb,1:red,0.5373;green,0.5294;blue,0.5059}, draw opacity={1.0}, rotate={0.0}}, y grid style={color={rgb,1:red,0.8824;green,0.8784;blue,0.851}, draw opacity={0.1}, line width={1}, solid}, yticklabel pos={left}, y axis line style={color={rgb,1:red,0.0;green,0.0;blue,0.0}, draw opacity={1.0}, line width={1}, solid}, colorbar={false}]
    \addplot[color={rgb,1:red,0.8889;green,0.4356;blue,0.2781}, name path={4}, draw opacity={1.0}, line width={1}, solid, mark={*}, mark size={2.25 pt}, mark repeat={1}, mark options={color={rgb,1:red,0.0;green,0.0;blue,0.0}, draw opacity={1.0}, fill={rgb,1:red,0.8889;green,0.4356;blue,0.2781}, fill opacity={1.0}, line width={0.75}, rotate={0}, solid}]
        table[row sep={\\}]
        {
            \\
            1.0  0.4117283920567246  \\
            2.0  0.08879947751741454  \\
            3.0  0.08840036148123909  \\
            4.0  0.08892258286658448  \\
            5.0  0.08918309391637885  \\
            6.0  4.2438603022844475e-8  \\
        }
        ;
\end{axis}
\end{tikzpicture}\hspace{3ex}

\begin{tikzpicture}[/tikz/background rectangle/.style={fill={rgb,1:red,1.0;green,1.0;blue,1.0}, fill opacity={1.0}, draw opacity={1.0}}, show background rectangle]
\begin{axis}[point meta max={nan}, point meta min={nan}, legend cell align={left}, legend columns={1}, legend style={color={rgb,1:red,0.0;green,0.0;blue,0.0}, draw opacity={1.0}, line width={1}, solid, fill={rgb,1:red,1.0;green,1.0;blue,1.0}, fill opacity={1.0}, text opacity={1.0}, font={{\fontsize{5 pt}{6 pt}\selectfont}}, text={rgb,1:red,0.3216;green,0.3176;blue,0.3059}, cells={anchor={center}}, at={(0.5, -0.62)}, anchor={north}}, axis background/.style={fill={rgb,1:red,1.0;green,1.0;blue,1.0}, opacity={1.0}}, anchor={north west}, xshift={0.4109mm}, yshift={-0.4044mm}, width={68.0000mm}, height={35.0000mm}, scaled x ticks={false}, xlabel={number of samples $M$}, x tick style={color={rgb,1:red,0.0;green,0.0;blue,0.0}, opacity={1.0}}, x tick label style={color={rgb,1:red,0.0;green,0.0;blue,0.0}, opacity={1.0}, rotate={0}}, xlabel style={at={(ticklabel cs:0.5)}, anchor=near ticklabel, at={{(ticklabel cs:0.5)}}, anchor={near ticklabel}, font={{\fontsize{8 pt}{10.4 pt}\selectfont}}, color={rgb,1:red,0.3216;green,0.3176;blue,0.3059}, draw opacity={1.0}, rotate={0.0}}, xmode={log}, log basis x={10}, xmajorgrids={true}, xmin={15.49033804788709}, xmax={129112.74071728889}, xticklabels={{$10^{2}$,$10^{4}$}}, xtick={{100.0,10000.0}}, xtick align={inside}, xticklabel style={font={{\fontsize{7 pt}{9.1 pt}\selectfont}}, color={rgb,1:red,0.5373;green,0.5294;blue,0.5059}, draw opacity={1.0}, rotate={0.0}}, x grid style={color={rgb,1:red,0.8824;green,0.8784;blue,0.851}, draw opacity={0.1}, line width={1}, solid}, xticklabel pos={left}, x axis line style={color={rgb,1:red,0.0;green,0.0;blue,0.0}, draw opacity={1.0}, line width={1}, solid}, scaled y ticks={false}, ylabel={$L_2\ \mathrm{error}$}, y tick style={color={rgb,1:red,0.0;green,0.0;blue,0.0}, opacity={1.0}}, y tick label style={color={rgb,1:red,0.0;green,0.0;blue,0.0}, opacity={1.0}, rotate={0}}, ylabel style={at={(ticklabel cs:0.5)}, anchor=near ticklabel, at={{(ticklabel cs:0.5)}}, anchor={near ticklabel}, font={{\fontsize{8 pt}{10.4 pt}\selectfont}}, color={rgb,1:red,0.3216;green,0.3176;blue,0.3059}, draw opacity={1.0}, rotate={0.0}}, ymode={log}, log basis y={10}, ymajorgrids={true}, ymin={2.279409420896454e-9}, ymax={2.5383739428325707}, yticklabels={{$10^{-5}$,$10^{0}$}}, ytick={{1.0e-5,1.0}}, ytick align={inside}, yticklabel style={font={{\fontsize{7 pt}{9.1 pt}\selectfont}}, color={rgb,1:red,0.5373;green,0.5294;blue,0.5059}, draw opacity={1.0}, rotate={0.0}}, y grid style={color={rgb,1:red,0.8824;green,0.8784;blue,0.851}, draw opacity={0.1}, line width={1}, solid}, yticklabel pos={left}, y axis line style={color={rgb,1:red,0.0;green,0.0;blue,0.0}, draw opacity={1.0}, line width={1}, solid}, colorbar={false}]
    \addplot[color={rgb,1:red,0.8889;green,0.4356;blue,0.2781}, name path={4}, draw opacity={1.0}, line width={1}, solid, mark={*}, mark size={2.25 pt}, mark repeat={1}, mark options={color={rgb,1:red,0.0;green,0.0;blue,0.0}, draw opacity={1.0}, fill={rgb,1:red,0.8889;green,0.4356;blue,0.2781}, fill opacity={1.0}, line width={0.75}, rotate={0}, solid}]
        table[row sep={\\}]
        {
            \\
            20.0  1.407719025931415  \\
            40.0  1.2557480776878795  \\
            100.0  1.1529935261053688  \\
            200.0  0.2908147156650633  \\
            400.0  0.11471443389592935  \\
            1000.0  0.09124080258785476  \\
            2000.0  0.09168238173879519  \\
            4000.0  0.09266775372737017  \\
            10000.0  0.09221266073202748  \\
            20000.0  0.09207411177607043  \\
            40000.0  0.09152475292680559  \\
            100000.0  0.09132061768919285  \\
        }
        ;
\end{axis}
\end{tikzpicture}}\\[2ex]
			\mbox{

\begin{tikzpicture}[/tikz/background rectangle/.style={fill={rgb,1:red,1.0;green,1.0;blue,1.0}, fill opacity={1.0}, draw opacity={1.0}}, show background rectangle]
\begin{axis}[point meta max={nan}, point meta min={nan}, legend cell align={left}, legend columns={1}, legend style={color={rgb,1:red,0.0;green,0.0;blue,0.0}, draw opacity={1.0}, line width={1}, solid, fill={rgb,1:red,1.0;green,1.0;blue,1.0}, fill opacity={1.0}, text opacity={1.0}, font={{\fontsize{5 pt}{6 pt}\selectfont}}, text={rgb,1:red,0.3216;green,0.3176;blue,0.3059}, cells={anchor={center}}, at={(0.5, -0.62)}, anchor={north}}, axis background/.style={fill={rgb,1:red,1.0;green,1.0;blue,1.0}, opacity={1.0}}, anchor={north west}, xshift={0.4109mm}, yshift={-0.4044mm}, width={68.0000mm}, height={35.0000mm}, scaled x ticks={false}, xlabel={superposition dimension $q$}, x tick style={color={rgb,1:red,0.0;green,0.0;blue,0.0}, opacity={1.0}}, x tick label style={color={rgb,1:red,0.0;green,0.0;blue,0.0}, opacity={1.0}, rotate={0}}, xlabel style={at={(ticklabel cs:0.5)}, anchor=near ticklabel, at={{(ticklabel cs:0.5)}}, anchor={near ticklabel}, font={{\fontsize{8 pt}{10.4 pt}\selectfont}}, color={rgb,1:red,0.3216;green,0.3176;blue,0.3059}, draw opacity={1.0}, rotate={0.0}}, xmajorgrids={true}, xmin={0.8499999999999996}, xmax={6.15}, xticklabels={{$1$,$2$,$3$,$4$,$5$,$6$}}, xtick={{1.0,2.0,3.0,4.0,5.0,6.0}}, xtick align={inside}, xticklabel style={font={{\fontsize{7 pt}{9.1 pt}\selectfont}}, color={rgb,1:red,0.5373;green,0.5294;blue,0.5059}, draw opacity={1.0}, rotate={0.0}}, x grid style={color={rgb,1:red,0.8824;green,0.8784;blue,0.851}, draw opacity={0.1}, line width={1}, solid}, xticklabel pos={left}, x axis line style={color={rgb,1:red,0.0;green,0.0;blue,0.0}, draw opacity={1.0}, line width={1}, solid}, scaled y ticks={false}, ylabel={$L_2\ \mathrm{error}$}, y tick style={color={rgb,1:red,0.0;green,0.0;blue,0.0}, opacity={1.0}}, y tick label style={color={rgb,1:red,0.0;green,0.0;blue,0.0}, opacity={1.0}, rotate={0}}, ylabel style={at={(ticklabel cs:0.5)}, anchor=near ticklabel, at={{(ticklabel cs:0.5)}}, anchor={near ticklabel}, font={{\fontsize{8 pt}{10.4 pt}\selectfont}}, color={rgb,1:red,0.3216;green,0.3176;blue,0.3059}, draw opacity={1.0}, rotate={0.0}}, ymode={log}, log basis y={10}, ymajorgrids={true}, ymin={2.4056518153848166e-9}, ymax={0.43723789890047216}, yticklabels={{$10^{-6}$,$10^{-3}$}}, ytick={{1.0e-6,0.001}}, ytick align={inside}, yticklabel style={font={{\fontsize{7 pt}{9.1 pt}\selectfont}}, color={rgb,1:red,0.5373;green,0.5294;blue,0.5059}, draw opacity={1.0}, rotate={0.0}}, y grid style={color={rgb,1:red,0.8824;green,0.8784;blue,0.851}, draw opacity={0.1}, line width={1}, solid}, yticklabel pos={left}, y axis line style={color={rgb,1:red,0.0;green,0.0;blue,0.0}, draw opacity={1.0}, line width={1}, solid}, colorbar={false}]
    \addplot[color={rgb,1:red,0.8889;green,0.4356;blue,0.2781}, name path={8}, draw opacity={1.0}, line width={1}, solid, mark={*}, mark size={2.25 pt}, mark repeat={1}, mark options={color={rgb,1:red,0.0;green,0.0;blue,0.0}, draw opacity={1.0}, fill={rgb,1:red,0.8889;green,0.4356;blue,0.2781}, fill opacity={1.0}, line width={0.75}, rotate={0}, solid}]
        table[row sep={\\}]
        {
            \\
            1.0  0.2542898708935593  \\
            2.0  0.1774410911795262  \\
            3.0  1.4765561693513836e-8  \\
            4.0  1.4970030524072083e-8  \\
            5.0  2.0799300214384842e-8  \\
            6.0  2.574027644929882e-8  \\
        }
        ;
\end{axis}
\end{tikzpicture}\hspace{3ex}

\begin{tikzpicture}[/tikz/background rectangle/.style={fill={rgb,1:red,1.0;green,1.0;blue,1.0}, fill opacity={1.0}, draw opacity={1.0}}, show background rectangle]
\begin{axis}[point meta max={nan}, point meta min={nan}, legend cell align={left}, legend columns={1}, legend style={color={rgb,1:red,0.0;green,0.0;blue,0.0}, draw opacity={1.0}, line width={1}, solid, fill={rgb,1:red,1.0;green,1.0;blue,1.0}, fill opacity={1.0}, text opacity={1.0}, font={{\fontsize{5 pt}{6 pt}\selectfont}}, text={rgb,1:red,0.3216;green,0.3176;blue,0.3059}, cells={anchor={center}}, at={(0.5, -0.62)}, anchor={north}}, axis background/.style={fill={rgb,1:red,1.0;green,1.0;blue,1.0}, opacity={1.0}}, anchor={north west}, xshift={0.4109mm}, yshift={-0.4044mm}, width={68.0000mm}, height={35.0000mm}, scaled x ticks={false}, xlabel={number of samples $M$}, x tick style={color={rgb,1:red,0.0;green,0.0;blue,0.0}, opacity={1.0}}, x tick label style={color={rgb,1:red,0.0;green,0.0;blue,0.0}, opacity={1.0}, rotate={0}}, xlabel style={at={(ticklabel cs:0.5)}, anchor=near ticklabel, at={{(ticklabel cs:0.5)}}, anchor={near ticklabel}, font={{\fontsize{8 pt}{10.4 pt}\selectfont}}, color={rgb,1:red,0.3216;green,0.3176;blue,0.3059}, draw opacity={1.0}, rotate={0.0}}, xmode={log}, log basis x={10}, xmajorgrids={true}, xmin={15.49033804788709}, xmax={129112.74071728889}, xticklabels={{$10^{2}$,$10^{4}$}}, xtick={{100.0,10000.0}}, xtick align={inside}, xticklabel style={font={{\fontsize{7 pt}{9.1 pt}\selectfont}}, color={rgb,1:red,0.5373;green,0.5294;blue,0.5059}, draw opacity={1.0}, rotate={0.0}}, x grid style={color={rgb,1:red,0.8824;green,0.8784;blue,0.851}, draw opacity={0.1}, line width={1}, solid}, xticklabel pos={left}, x axis line style={color={rgb,1:red,0.0;green,0.0;blue,0.0}, draw opacity={1.0}, line width={1}, solid}, scaled y ticks={false}, ylabel={$L_2\ \mathrm{error}$}, y tick style={color={rgb,1:red,0.0;green,0.0;blue,0.0}, opacity={1.0}}, y tick label style={color={rgb,1:red,0.0;green,0.0;blue,0.0}, opacity={1.0}, rotate={0}}, ylabel style={at={(ticklabel cs:0.5)}, anchor=near ticklabel, at={{(ticklabel cs:0.5)}}, anchor={near ticklabel}, font={{\fontsize{8 pt}{10.4 pt}\selectfont}}, color={rgb,1:red,0.3216;green,0.3176;blue,0.3059}, draw opacity={1.0}, rotate={0.0}}, ymode={log}, log basis y={10}, ymajorgrids={true}, ymin={9.60109896046584e-10}, ymax={2.4147098127363966}, yticklabels={{$10^{-5}$,$10^{0}$}}, ytick={{1.0e-5,1.0}}, ytick align={inside}, yticklabel style={font={{\fontsize{7 pt}{9.1 pt}\selectfont}}, color={rgb,1:red,0.5373;green,0.5294;blue,0.5059}, draw opacity={1.0}, rotate={0.0}}, y grid style={color={rgb,1:red,0.8824;green,0.8784;blue,0.851}, draw opacity={0.1}, line width={1}, solid}, yticklabel pos={left}, y axis line style={color={rgb,1:red,0.0;green,0.0;blue,0.0}, draw opacity={1.0}, line width={1}, solid}, colorbar={false}]
    \addplot[color={rgb,1:red,0.8889;green,0.4356;blue,0.2781}, name path={8}, draw opacity={1.0}, line width={1}, solid, mark={*}, mark size={2.25 pt}, mark repeat={1}, mark options={color={rgb,1:red,0.0;green,0.0;blue,0.0}, draw opacity={1.0}, fill={rgb,1:red,0.8889;green,0.4356;blue,0.2781}, fill opacity={1.0}, line width={0.75}, rotate={0}, solid}]
        table[row sep={\\}]
        {
            \\
            20.0  1.2611726331417143  \\
            40.0  1.201976216357805  \\
            100.0  1.030429646586456  \\
            200.0  0.4244473378251633  \\
            400.0  0.132549852658581  \\
            1000.0  0.19593938153289775  \\
            2000.0  0.1872740745037041  \\
            4000.0  0.18881740109599932  \\
            10000.0  0.18641983062349285  \\
            20000.0  0.18611026831937796  \\
            40000.0  0.1863341071153947  \\
            100000.0  0.18593033447271373  \\
        }
        ;
\end{axis}
\end{tikzpicture}}
			\caption{Error~$e$ for the \emph{diffusion equation with the two control variants~\eqref{eq:gradient_control} and~\eqref{eq:nonlin_adv}} (top to bottom), for~$i=2$, depending on the superposition dimension~$q$ (left) and the number of samples~$M$ (right).
				\label{fig:error_advection}}
		\end{figure}
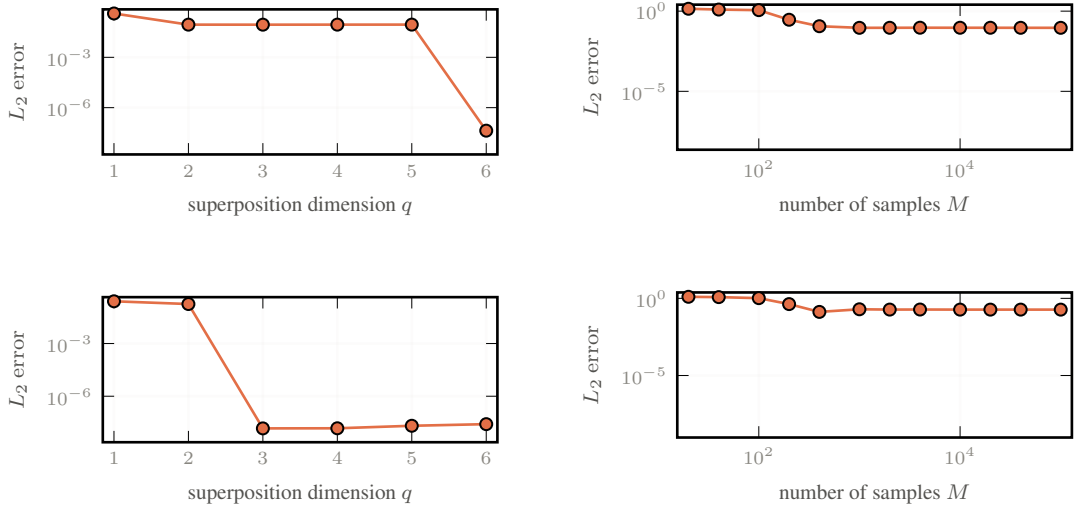

		\subsection{Shallow water equations\label{sec:shallow}}
		The 1D shallow water equations, see~\cite[Chapter 13]{LeVe02}, form a nonlinear system of conservation laws that describes the conservation of mass and momentum. They are given by
		\begin{align}\label{eq:SWE}
			\begin{split}
				\frac{\partial h}{\partial t} + \frac{\partial (hv)}{\partial \omega} &= 0 \qquad \text{(mass conservation)} \\
				\frac{\partial (hv)}{\partial t} + \frac{\partial}{\partial \omega}\left( hv^2 + \frac{1}{2} g h^2 \right) &= 0 \qquad \text{(momentum conservation)}.   
			\end{split}
		\end{align}
		Here, $g$ is the gravitational acceleration, while $h(\omega,t)$ and~$v(\omega,t)$  denote the fluid depth and the horizontal velocity, respectively, for~$(\omega,t)\in \Omega \times [0,\infty)$.
		Discretization of the spatial domain~$\Omega$, for instance using a Lax--Friedrichs scheme\footnote{More advanced finite-volume solvers, e.\,g., Godunov's method or Riemann solvers~\cite{LeVe02}, lead to a similar structure.} with an equidistant spatial grid consisting of the nodes~$\omega_1,\ldots,\omega_n$ with grid size~$\Delta \omega$, yields the explicit time stepping scheme
		\begin{equation}\label{eq:low_dim_PDE}
			\mathbf{Z}_i^+ = \frac{1}{2} \left( \mathbf{Z}_{i+1} + \mathbf{Z}_{i-1} \right) - \tfrac{\Delta t}{2\Delta \omega} \left( \mathbf{F}_{i+1} - \mathbf{F}_{i-1} \right),
		\end{equation}
		with temporal step size~$\Delta t$. The quantities~$\mathbf{Z}_i$ and~$\mathbf{F}_i$ approximate the conserved quantity and the nonlinearity at space point~$\omega_i$, $i\in \{1,\ldots,n\}$, respectively, i.\,e., 
		\begin{equation*}
			\mathbf{Z}_i = \begin{pmatrix}
				h(\omega_i) \\
				h(\omega_i) v(\omega_i)
			\end{pmatrix} \qquad\text{and}\qquad \mathbf{F}_i = \begin{pmatrix}
				h(\omega_i) v(\omega_i) \\
				h(\omega_i) v(\omega_i)^2 + \frac{1}{2} g h(\omega_i)^2
			\end{pmatrix}.
		\end{equation*}

		Rewriting this scheme in terms of the variables~$x = (h,q)$ with the conserved quantity~$q = hv$, and denoting its numerical approximation by~$X$, we obtain the discrete-time system
		\begin{align*}
			X_i^+ = \frac12( X_{i+1} + X_{i-1}) - \tfrac{\Delta t}{2\Delta \omega} \left(\mathbf{F}(X_{i+1}) - \mathbf{F}(X_{i-1})\right) = f_2(X_{i+1}) + f_1(X_{i-1})
		\end{align*}
		 for the spatial grid point~$i$, where
		\begin{align*}
			\mathbf{F}(x) = \mathbf{F}(h,q) = \begin{pmatrix}
				q\\
				q^2/h + \frac12 g h^2
			\end{pmatrix},\quad f_1(x) = \frac12 x + \frac{\Delta t}{2\Delta \omega}\mathbf{F}(x), \quad
			f_2(x) = \frac12 x - \frac{\Delta t}{2\Delta \omega}\mathbf{F}(x).
		\end{align*}
		At first sight, the two summands look like fully decoupled terms. However, $X$ is composed of the height~$h$ and the discharge~$q=hv$, and is thus two-dimensional. Applying our scalar-valued approximation approach, we obtain nonlinearly coupled terms, since, e.\,g., the second line reads
		\begin{align}\label{eq:Q}
			Q_i^+ = \frac12(Q_{i+1} + Q_{i-1}) - \frac{\Delta t}{2\Delta \omega}\left(Q_{i+1}^2/H_{i+1} + \frac12 g H_{i+1}^2 - Q_{i-1}^2/H_{i-1} - \frac12 g H_{i-1}^2\right).
		\end{align}
		Clearly, its right-hand side may be written in terms of~$z = (Q_{i+1},Q_{i-1},H_{i+1},H_{i-1})$ and thus involves (multiplicative) couplings of $(Q_{i+1},H_{i+1})$ and $(Q_{i-1},H_{i-1})$.

		\medskip\noindent\emph{Adding control.} An acceleration~$u$ (e.\,g., from a pump or turbine) enters the second equation through a term~$hu$. The factor~$h$ makes the momentum source stronger where the water is deeper. It also ensures dimensional consistency, since~$[h]=\mathrm{m}$ and~$[u]=\mathrm{m}/\mathrm{s}^2$ give~$[hu]=\mathrm{m}^2/\mathrm{s}^2$, which matches~$[\partial_t q]$ for the discharge per unit width~$q=hv$.
		In the discretization, this amounts to adding the term
		\[
			\Delta t\,H_{i} U_{i}
		\]
		to the right-hand side of~\eqref{eq:Q}.

		Concerning the sampling, note that the flux term~$Q^2/H$ is singular at~$H=0$. We therefore draw all eight components on~$[-1,1]$ as before. For the height components, we then use the shifted values
		\begin{align}
			\label{eq:shallow_shift}
			H_i \coloneqq X_i+2 \in [1,3], \qquad\text{where } i \text{ is a height index},
		\end{align}
		so that the water height stays bounded away from zero. The discharge components~$Q_i$ are used without shifting on~$[-1,1]$.

		In contrast to the previous examples, the flux term~$Q^2/H$ is not a polynomial, so the ANOVA approximation cannot reproduce the dynamics exactly.
		The error~\eqref{eq:mc_error} is~$e = 2.83\cdot 10^{-2}$. This is also reflected in~Figure~\ref{fig:error_shallow}, where the two panels show plateaus at approximately $2.8\cdot10^{-2}$ and~$3.0\cdot10^{-2}$, respectively. 
		\begin{figure}[htb]
			\centering
			\input{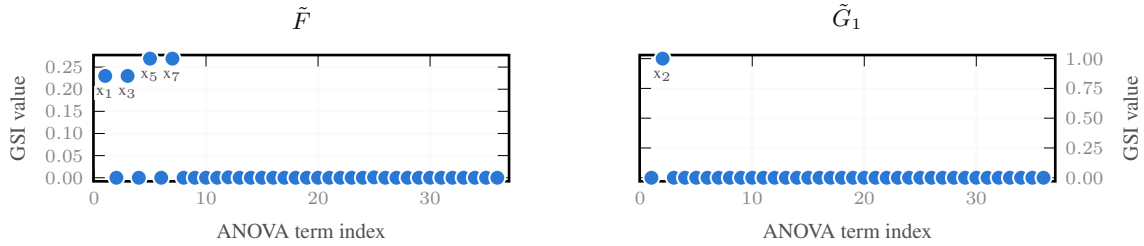}
			\caption{Approximated global sensitivity indices for the \emph{shallow water equations~\eqref{eq:Q}}, for~\mbox{$i=2$}, $X=(H_1,H_2,H_3,H_4,Q_1,Q_2,Q_3,Q_4)$, $\Delta t=g=1$, and~$\Delta \omega=2$, with the heights sampled from~$[1,3]$ (affinely transformed to~$[-1,1]$), where the dominant terms are~$\{1\}, \{3\}, \{5\}, \{7\}$ of~$F$ and~$\{2\}$ of~$G_1$.
				\label{fig:shallow}}
		\end{figure}

		The influence of~$q$ and~$M$ is displayed in Figure~\ref{fig:error_shallow}.
		In contrast to the previous examples, a non-vanishing error of about~$2.8\cdot10^{-2}$ remains for \emph{all} parameter choices. In the following, we identify its source among the three error contributions of Remark~\ref{Remark:lsqr_error}.
		It is not a superposition truncation error, since the flux term depends on two variables only, such that the truncation by~$q$ is exact for~$q\geq2$.
		It is not a projection error either. Expanding the row in the Chebyshev basis, the projection error at~$q=2$ decreases from~$2.8\cdot10^{-2}$ for~$N_2=2$ to~$2.8\cdot10^{-6}$ for~$\bN=(20,8)$, since~$1/H$ is analytic on~$[1,3]$. However, the measured error stays at~$2.8\cdot10^{-2}$ for all admissible bandwidths.
		Instead, the plateau stems from the truncation by the sensitivity threshold. The flux-induced pair couplings~$\{H_{i\pm1},Q_{i\pm1}\}$ carry sensitivity indices of~$8.5\cdot10^{-4}$ each, fall below~$\varepsilon=1/73$ and are hence discarded, cf.\ Figure~\ref{fig:shallow}.
		Removing these two terms from the exact expansion leaves an error of~$2.84\cdot10^{-2}$, which is independent of~$\bN$ by construction and matches the measured value.
		Hence, the error floor of this example is governed by the threshold~$\varepsilon$ and not by the bandwidths. This is a second instance of the effect already observed for~\eqref{eq:gradient_control} and~\eqref{eq:nonlin_adv}.

		\begin{figure}[htb]
			\centering
			\mbox{

\begin{tikzpicture}[/tikz/background rectangle/.style={fill={rgb,1:red,1.0;green,1.0;blue,1.0}, fill opacity={1.0}, draw opacity={1.0}}, show background rectangle]
\begin{axis}[point meta max={nan}, point meta min={nan}, legend cell align={left}, legend columns={1}, legend style={color={rgb,1:red,0.0;green,0.0;blue,0.0}, draw opacity={1.0}, line width={1}, solid, fill={rgb,1:red,1.0;green,1.0;blue,1.0}, fill opacity={1.0}, text opacity={1.0}, font={{\fontsize{5 pt}{6 pt}\selectfont}}, text={rgb,1:red,0.3216;green,0.3176;blue,0.3059}, cells={anchor={center}}, at={(0.5, -0.62)}, anchor={north}}, axis background/.style={fill={rgb,1:red,1.0;green,1.0;blue,1.0}, opacity={1.0}}, anchor={north west}, xshift={0.4109mm}, yshift={-0.4045mm}, width={68.0000mm}, height={35.0000mm}, scaled x ticks={false}, xlabel={superposition dimension $q$}, x tick style={color={rgb,1:red,0.0;green,0.0;blue,0.0}, opacity={1.0}}, x tick label style={color={rgb,1:red,0.0;green,0.0;blue,0.0}, opacity={1.0}, rotate={0}}, xlabel style={at={(ticklabel cs:0.5)}, anchor=near ticklabel, at={{(ticklabel cs:0.5)}}, anchor={near ticklabel}, font={{\fontsize{8 pt}{10.4 pt}\selectfont}}, color={rgb,1:red,0.3216;green,0.3176;blue,0.3059}, draw opacity={1.0}, rotate={0.0}}, xmajorgrids={true}, xmin={0.8499999999999996}, xmax={6.15}, xticklabels={{$1$,$2$,$3$,$4$,$5$,$6$}}, xtick={{1.0,2.0,3.0,4.0,5.0,6.0}}, xtick align={inside}, xticklabel style={font={{\fontsize{7 pt}{9.1 pt}\selectfont}}, color={rgb,1:red,0.5373;green,0.5294;blue,0.5059}, draw opacity={1.0}, rotate={0.0}}, x grid style={color={rgb,1:red,0.8824;green,0.8784;blue,0.851}, draw opacity={0.1}, line width={1}, solid}, xticklabel pos={left}, x axis line style={color={rgb,1:red,0.0;green,0.0;blue,0.0}, draw opacity={1.0}, line width={1}, solid}, scaled y ticks={false}, ylabel={$L_2\ \mathrm{error}$}, y tick style={color={rgb,1:red,0.0;green,0.0;blue,0.0}, opacity={1.0}}, y tick label style={color={rgb,1:red,0.0;green,0.0;blue,0.0}, opacity={1.0}, rotate={0}}, ylabel style={at={(ticklabel cs:0.5)}, anchor=near ticklabel, at={{(ticklabel cs:0.5)}}, anchor={near ticklabel}, font={{\fontsize{8 pt}{10.4 pt}\selectfont}}, color={rgb,1:red,0.3216;green,0.3176;blue,0.3059}, draw opacity={1.0}, rotate={0.0}}, ymode={log}, log basis y={10}, ymajorgrids={true}, ymin={0.02549008133967687}, ymax={0.7477609715320228}, yticklabels={{$10^{-1.5}$,$10^{-1.0}$,$10^{-0.5}$}}, ytick={{0.03162277660168379,0.1,0.31622776601683794}}, ytick align={inside}, yticklabel style={font={{\fontsize{7 pt}{9.1 pt}\selectfont}}, color={rgb,1:red,0.5373;green,0.5294;blue,0.5059}, draw opacity={1.0}, rotate={0.0}}, y grid style={color={rgb,1:red,0.8824;green,0.8784;blue,0.851}, draw opacity={0.1}, line width={1}, solid}, yticklabel pos={left}, y axis line style={color={rgb,1:red,0.0;green,0.0;blue,0.0}, draw opacity={1.0}, line width={1}, solid}, colorbar={false}]
    \addplot[color={rgb,1:red,0.8889;green,0.4356;blue,0.2781}, name path={26}, draw opacity={1.0}, line width={1}, solid, mark={*}, mark size={2.25 pt}, mark repeat={1}, mark options={color={rgb,1:red,0.0;green,0.0;blue,0.0}, draw opacity={1.0}, fill={rgb,1:red,0.8889;green,0.4356;blue,0.2781}, fill opacity={1.0}, line width={0.75}, rotate={0}, solid}]
        table[row sep={\\}]
        {
            \\
            1.0  0.6795678802542797  \\
            2.0  0.02822375380401462  \\
            3.0  0.028222501131408357  \\
            4.0  0.02820911181825524  \\
            5.0  0.028230822488713246  \\
            6.0  0.028047953031351394  \\
        }
        ;
\end{axis}
\end{tikzpicture}\hspace{3ex}

\begin{tikzpicture}[/tikz/background rectangle/.style={fill={rgb,1:red,1.0;green,1.0;blue,1.0}, fill opacity={1.0}, draw opacity={1.0}}, show background rectangle]
\begin{axis}[point meta max={nan}, point meta min={nan}, legend cell align={left}, legend columns={1}, legend style={color={rgb,1:red,0.0;green,0.0;blue,0.0}, draw opacity={1.0}, line width={1}, solid, fill={rgb,1:red,1.0;green,1.0;blue,1.0}, fill opacity={1.0}, text opacity={1.0}, font={{\fontsize{5 pt}{6 pt}\selectfont}}, text={rgb,1:red,0.3216;green,0.3176;blue,0.3059}, cells={anchor={center}}, at={(0.5, -0.62)}, anchor={north}}, axis background/.style={fill={rgb,1:red,1.0;green,1.0;blue,1.0}, opacity={1.0}}, anchor={north west}, xshift={0.4107mm}, yshift={-0.4043mm}, width={68.0000mm}, height={35.0000mm}, scaled x ticks={false}, xlabel={number of samples $M$}, x tick style={color={rgb,1:red,0.0;green,0.0;blue,0.0}, opacity={1.0}}, x tick label style={color={rgb,1:red,0.0;green,0.0;blue,0.0}, opacity={1.0}, rotate={0}}, xlabel style={at={(ticklabel cs:0.5)}, anchor=near ticklabel, at={{(ticklabel cs:0.5)}}, anchor={near ticklabel}, font={{\fontsize{8 pt}{10.4 pt}\selectfont}}, color={rgb,1:red,0.3216;green,0.3176;blue,0.3059}, draw opacity={1.0}, rotate={0.0}}, xmode={log}, log basis x={10}, xmajorgrids={true}, xmin={15.49033804788709}, xmax={129112.74071728889}, xticklabels={{$10^{2}$,$10^{4}$}}, xtick={{100.0,10000.0}}, xtick align={inside}, xticklabel style={font={{\fontsize{7 pt}{9.1 pt}\selectfont}}, color={rgb,1:red,0.5373;green,0.5294;blue,0.5059}, draw opacity={1.0}, rotate={0.0}}, x grid style={color={rgb,1:red,0.8824;green,0.8784;blue,0.851}, draw opacity={0.1}, line width={1}, solid}, xticklabel pos={left}, x axis line style={color={rgb,1:red,0.0;green,0.0;blue,0.0}, draw opacity={1.0}, line width={1}, solid}, scaled y ticks={false}, ylabel={$L_2\ \mathrm{error}$}, y tick style={color={rgb,1:red,0.0;green,0.0;blue,0.0}, opacity={1.0}}, y tick label style={color={rgb,1:red,0.0;green,0.0;blue,0.0}, opacity={1.0}, rotate={0}}, ylabel style={at={(ticklabel cs:0.5)}, anchor=near ticklabel, at={{(ticklabel cs:0.5)}}, anchor={near ticklabel}, font={{\fontsize{8 pt}{10.4 pt}\selectfont}}, color={rgb,1:red,0.3216;green,0.3176;blue,0.3059}, draw opacity={1.0}, rotate={0.0}}, ymode={log}, log basis y={10}, ymajorgrids={true}, ymin={0.025964770835867484}, ymax={1.806040427422512}, yticklabels={{$10^{-1}$,$10^{0}$}}, ytick={{0.1,1.0}}, ytick align={inside}, yticklabel style={font={{\fontsize{7 pt}{9.1 pt}\selectfont}}, color={rgb,1:red,0.5373;green,0.5294;blue,0.5059}, draw opacity={1.0}, rotate={0.0}}, y grid style={color={rgb,1:red,0.8824;green,0.8784;blue,0.851}, draw opacity={0.1}, line width={1}, solid}, yticklabel pos={left}, y axis line style={color={rgb,1:red,0.0;green,0.0;blue,0.0}, draw opacity={1.0}, line width={1}, solid}, colorbar={false}]
    \addplot[color={rgb,1:red,0.8889;green,0.4356;blue,0.2781}, name path={2}, draw opacity={1.0}, line width={1}, solid, mark={*}, mark size={2.25 pt}, mark repeat={1}, mark options={color={rgb,1:red,0.0;green,0.0;blue,0.0}, draw opacity={1.0}, fill={rgb,1:red,0.8889;green,0.4356;blue,0.2781}, fill opacity={1.0}, line width={0.75}, rotate={0}, solid}]
        table[row sep={\\}]
        {
            \\
            20.0  1.5609024957260307  \\
            40.0  1.601716632182643  \\
            100.0  1.441936321856274  \\
            200.0  0.09585196276637752  \\
            400.0  0.030639144771592252  \\
            1000.0  0.029623433723668536  \\
            2000.0  0.029497598709908538  \\
            4000.0  0.029555479559991476  \\
            10000.0  0.029549215263019176  \\
            20000.0  0.029656035086099666  \\
            40000.0  0.02960425858310866  \\
            100000.0  0.02955198510109972  \\
        }
        ;
\end{axis}
\end{tikzpicture}}
			\caption{Error~$e$ for the \emph{shallow water equations~\eqref{eq:Q}}, for~$i=2$,  depending on the superposition dimension~$q$ (left) and the number of samples~$M$ (right).
				\label{fig:error_shallow}}
		\end{figure}
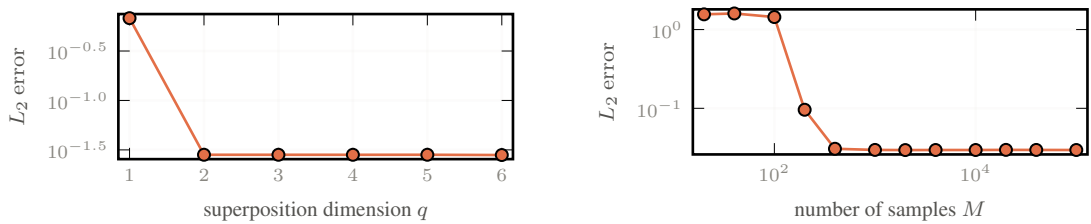

		\subsection{Viscous Burgers equation}
		As a final example, we return to the controlled viscous Burgers equation of Example~\ref{ex:burgers}. Unlike the previous examples, it features a multivariate control with~$m=2$, namely the state-affine actuation~$\gamma(\omega)\,x\,U_1$ and the purely additive actuation~$b(\omega)\,U_2$.
		Recall from Example~\ref{ex:burgers} that the discretization yields the dynamics
		\begin{align}
			\label{eq:burgers_control}
			X_i^+ = X_i + \Delta t\left(D\,\frac{X_{i+1} - 2X_i + X_{i-1}}{\Delta \omega^2} - X_i\,\frac{X_{i+1}-X_{i-1}}{2\Delta \omega} + \gamma_i\, X_i\, U_1 + b_i\, U_2 \right)
		\end{align}
		for spatial point~$i$,
		i.\,e., the index sets to be identified are $V_{i,0}=\bigl\{\{i-1\},\{i\},\{i+1\},\{i-1,i\},\{i,i+1\}\bigr\}$, $V_{i,1}=\bigl\{\{i\}\bigr\}$, and~$V_{i,2}=\{\varnothing\}$.
		As before, we consider the row~$i=2$ with~$d=8$ and choose~$D=\Delta t=\Delta\omega=1$ as well as~$\gamma_i=b_i=1$.

		The results are shown in Figure~\ref{fig:burgers}.
		The relevant ANOVA terms of~$F$ and~$G_1$ are correctly identified.
		For~$G_2$, the situation is degenerate. Since~$G_2$ consists of the constant term only, all terms coupling~$u_2$ to the state have vanishing global sensitivity indices, which are of order~$10^{-18}$ in the computation, such that they are correctly discarded. As these indices sum to zero up to rounding errors, they cannot be renormalized within their group, and Figure~\ref{fig:burgers} shows them without renormalization. The constant actuation~$b_i\,U_2$ is captured by the ANOVA term in~$u_2$ alone, which does not belong to the displayed group.
		As expected, the reconstruction is accurate with~$e = 8.15\cdot 10^{-9}$.
		\begin{figure}[htb]
			\centering
			\resizebox{\textwidth}{!}{\input{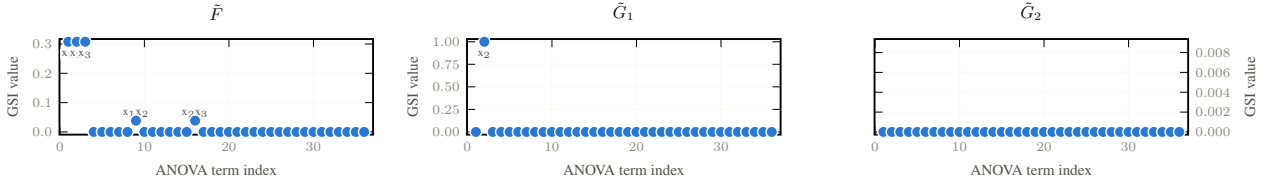}}
			\caption{Approximated global sensitivity indices for the \emph{viscous Burgers equation~\eqref{eq:burgers_control}}, for~$i=2$ and~$D=\Delta t=\Delta \omega=\gamma_i=b_i=1$, where the relevant terms are~$\{1\}, \{2\}, \{3\}, \{1,2\}, \{2,3\}$ of~$F$~(left), $\{2\}$ of~$G_1$ (middle) and~$\varnothing$ of~$G_2$ (right).
				\label{fig:burgers}}
		\end{figure}

		Finally, Figure~\ref{fig:error_burgers} shows the same parameter study for this example with two controls.
		The behaviour matches the one of the academic test function. The truncation error vanishes for~$q\geq2$ and the estimation error drops below~$10^{-6}$ once~$M\gtrsim 300$, since the right-hand side of~\eqref{eq:burgers_control} is polynomial and hence contained in the ansatz space.
		The additional control component does not change the qualitative behaviour but only increases the number of coefficients to be determined.

		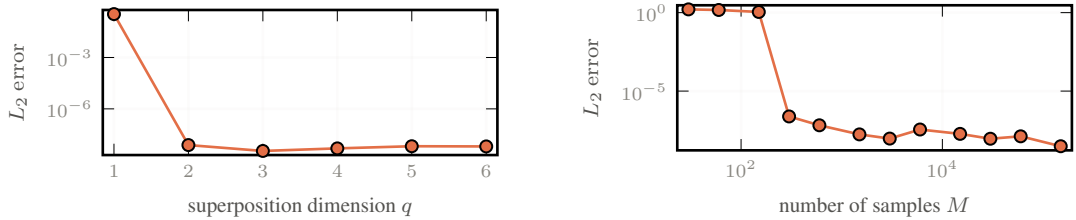
\begin{figure}[htb]
			\centering
			\mbox{

\begin{tikzpicture}[/tikz/background rectangle/.style={fill={rgb,1:red,1.0;green,1.0;blue,1.0}, fill opacity={1.0}, draw opacity={1.0}}, show background rectangle]
\begin{axis}[point meta max={nan}, point meta min={nan}, legend cell align={left}, legend columns={1}, legend style={color={rgb,1:red,0.0;green,0.0;blue,0.0}, draw opacity={1.0}, line width={1}, solid, fill={rgb,1:red,1.0;green,1.0;blue,1.0}, fill opacity={1.0}, text opacity={1.0}, font={{\fontsize{5 pt}{6 pt}\selectfont}}, text={rgb,1:red,0.3216;green,0.3176;blue,0.3059}, cells={anchor={center}}, at={(0.5, -0.62)}, anchor={north}}, axis background/.style={fill={rgb,1:red,1.0;green,1.0;blue,1.0}, opacity={1.0}}, anchor={north west}, xshift={0.4109mm}, yshift={-0.4044mm}, width={68.0000mm}, height={35.0000mm}, scaled x ticks={false}, xlabel={superposition dimension $q$}, x tick style={color={rgb,1:red,0.0;green,0.0;blue,0.0}, opacity={1.0}}, x tick label style={color={rgb,1:red,0.0;green,0.0;blue,0.0}, opacity={1.0}, rotate={0}}, xlabel style={at={(ticklabel cs:0.5)}, anchor=near ticklabel, at={{(ticklabel cs:0.5)}}, anchor={near ticklabel}, font={{\fontsize{8 pt}{10.4 pt}\selectfont}}, color={rgb,1:red,0.3216;green,0.3176;blue,0.3059}, draw opacity={1.0}, rotate={0.0}}, xmajorgrids={true}, xmin={0.8499999999999996}, xmax={6.15}, xticklabels={{$1$,$2$,$3$,$4$,$5$,$6$}}, xtick={{1.0,2.0,3.0,4.0,5.0,6.0}}, xtick align={inside}, xticklabel style={font={{\fontsize{7 pt}{9.1 pt}\selectfont}}, color={rgb,1:red,0.5373;green,0.5294;blue,0.5059}, draw opacity={1.0}, rotate={0.0}}, x grid style={color={rgb,1:red,0.8824;green,0.8784;blue,0.851}, draw opacity={0.1}, line width={1}, solid}, xticklabel pos={left}, x axis line style={color={rgb,1:red,0.0;green,0.0;blue,0.0}, draw opacity={1.0}, line width={1}, solid}, scaled y ticks={false}, ylabel={$L_2\ \mathrm{error}$}, y tick style={color={rgb,1:red,0.0;green,0.0;blue,0.0}, opacity={1.0}}, y tick label style={color={rgb,1:red,0.0;green,0.0;blue,0.0}, opacity={1.0}, rotate={0}}, ylabel style={at={(ticklabel cs:0.5)}, anchor=near ticklabel, at={{(ticklabel cs:0.5)}}, anchor={near ticklabel}, font={{\fontsize{8 pt}{10.4 pt}\selectfont}}, color={rgb,1:red,0.3216;green,0.3176;blue,0.3059}, draw opacity={1.0}, rotate={0.0}}, ymode={log}, log basis y={10}, ymajorgrids={true}, ymin={2.0043162385211816e-9}, ymax={0.5963903704799286}, yticklabels={{$10^{-6}$,$10^{-3}$}}, ytick={{1.0e-6,0.001}}, ytick align={inside}, yticklabel style={font={{\fontsize{7 pt}{9.1 pt}\selectfont}}, color={rgb,1:red,0.5373;green,0.5294;blue,0.5059}, draw opacity={1.0}, rotate={0.0}}, y grid style={color={rgb,1:red,0.8824;green,0.8784;blue,0.851}, draw opacity={0.1}, line width={1}, solid}, yticklabel pos={left}, y axis line style={color={rgb,1:red,0.0;green,0.0;blue,0.0}, draw opacity={1.0}, line width={1}, solid}, colorbar={false}]
    \addplot[color={rgb,1:red,0.8889;green,0.4356;blue,0.2781}, name path={24}, draw opacity={1.0}, line width={1}, solid, mark={*}, mark size={2.25 pt}, mark repeat={1}, mark options={color={rgb,1:red,0.0;green,0.0;blue,0.0}, draw opacity={1.0}, fill={rgb,1:red,0.8889;green,0.4356;blue,0.2781}, fill opacity={1.0}, line width={0.75}, rotate={0}, solid}]
        table[row sep={\\}]
        {
            \\
            1.0  0.3433307979131953  \\
            2.0  7.685294356030766e-9  \\
            3.0  3.4816419363368844e-9  \\
            4.0  4.910437173929684e-9  \\
            5.0  6.565382951744956e-9  \\
            6.0  6.392824709653343e-9  \\
        }
        ;
\end{axis}
\end{tikzpicture}\hspace{3ex}

\begin{tikzpicture}[/tikz/background rectangle/.style={fill={rgb,1:red,1.0;green,1.0;blue,1.0}, fill opacity={1.0}, draw opacity={1.0}}, show background rectangle]
\begin{axis}[point meta max={nan}, point meta min={nan}, legend cell align={left}, legend columns={1}, legend style={color={rgb,1:red,0.0;green,0.0;blue,0.0}, draw opacity={1.0}, line width={1}, solid, fill={rgb,1:red,1.0;green,1.0;blue,1.0}, fill opacity={1.0}, text opacity={1.0}, font={{\fontsize{5 pt}{6 pt}\selectfont}}, text={rgb,1:red,0.3216;green,0.3176;blue,0.3059}, cells={anchor={center}}, at={(0.5, -0.62)}, anchor={north}}, axis background/.style={fill={rgb,1:red,1.0;green,1.0;blue,1.0}, opacity={1.0}}, anchor={north west}, xshift={0.4107mm}, yshift={-0.4043mm}, width={68.0000mm}, height={35.0000mm}, scaled x ticks={false}, xlabel={number of samples $M$}, x tick style={color={rgb,1:red,0.0;green,0.0;blue,0.0}, opacity={1.0}}, x tick label style={color={rgb,1:red,0.0;green,0.0;blue,0.0}, opacity={1.0}, rotate={0}}, xlabel style={at={(ticklabel cs:0.5)}, anchor=near ticklabel, at={{(ticklabel cs:0.5)}}, anchor={near ticklabel}, font={{\fontsize{8 pt}{10.4 pt}\selectfont}}, color={rgb,1:red,0.3216;green,0.3176;blue,0.3059}, draw opacity={1.0}, rotate={0.0}}, xmode={log}, log basis x={10}, xmajorgrids={true}, xmin={23.235507071830632}, xmax={193669.11107593332}, xticklabels={{$10^{2}$,$10^{4}$}}, xtick={{100.0,10000.0}}, xtick align={inside}, xticklabel style={font={{\fontsize{7 pt}{9.1 pt}\selectfont}}, color={rgb,1:red,0.5373;green,0.5294;blue,0.5059}, draw opacity={1.0}, rotate={0.0}}, x grid style={color={rgb,1:red,0.8824;green,0.8784;blue,0.851}, draw opacity={0.1}, line width={1}, solid}, xticklabel pos={left}, x axis line style={color={rgb,1:red,0.0;green,0.0;blue,0.0}, draw opacity={1.0}, line width={1}, solid}, scaled y ticks={false}, ylabel={$L_2\ \mathrm{error}$}, y tick style={color={rgb,1:red,0.0;green,0.0;blue,0.0}, opacity={1.0}}, y tick label style={color={rgb,1:red,0.0;green,0.0;blue,0.0}, opacity={1.0}, rotate={0}}, ylabel style={at={(ticklabel cs:0.5)}, anchor=near ticklabel, at={{(ticklabel cs:0.5)}}, anchor={near ticklabel}, font={{\fontsize{8 pt}{10.4 pt}\selectfont}}, color={rgb,1:red,0.3216;green,0.3176;blue,0.3059}, draw opacity={1.0}, rotate={0.0}}, ymode={log}, log basis y={10}, ymajorgrids={true}, ymin={1.6732960397437243e-9}, ymax={2.9474240099020617}, yticklabels={{$10^{-5}$,$10^{0}$}}, ytick={{1.0e-5,1.0}}, ytick align={inside}, yticklabel style={font={{\fontsize{7 pt}{9.1 pt}\selectfont}}, color={rgb,1:red,0.5373;green,0.5294;blue,0.5059}, draw opacity={1.0}, rotate={0.0}}, y grid style={color={rgb,1:red,0.8824;green,0.8784;blue,0.851}, draw opacity={0.1}, line width={1}, solid}, yticklabel pos={left}, y axis line style={color={rgb,1:red,0.0;green,0.0;blue,0.0}, draw opacity={1.0}, line width={1}, solid}, colorbar={false}]
    \addplot[color={rgb,1:red,0.8889;green,0.4356;blue,0.2781}, name path={2}, draw opacity={1.0}, line width={1}, solid, mark={*}, mark size={2.25 pt}, mark repeat={1}, mark options={color={rgb,1:red,0.0;green,0.0;blue,0.0}, draw opacity={1.0}, fill={rgb,1:red,0.8889;green,0.4356;blue,0.2781}, fill opacity={1.0}, line width={0.75}, rotate={0}, solid}]
        table[row sep={\\}]
        {
            \\
            30.0  1.6134928887651476  \\
            60.0  1.4760068049964015  \\
            150.0  1.10456176157234  \\
            300.0  2.4192924972497293e-7  \\
            600.0  6.605380462583731e-8  \\
            1500.0  1.7382748532574523e-8  \\
            3000.0  9.401627599598611e-9  \\
            6000.0  3.5584477796657244e-8  \\
            15000.0  1.8747029549372303e-8  \\
            30000.0  9.187408726696916e-9  \\
            60000.0  1.3104768788485222e-8  \\
            150000.0  3.0566685217864454e-9  \\
        }
        ;
\end{axis}
\end{tikzpicture}}
			\caption{Error~$e$ for the \emph{viscous Burgers equation~\eqref{eq:burgers_control}} with~$i=2$ and~$m=2$ depending on the superposition dimension~$q$ (left) and the number of samples~$M$ (right).
				\label{fig:error_burgers}}
		\end{figure}
		\FloatBarrier

		\section{Conclusions\label{sec:conclusions}}

		We have extended the ANOVA approximation to control-affine systems~$\bx^+=F(\bx)+G(\bx)\bu$ whose components depend on only a few state variables, such that no effort is spent on interactions between control components, which are excluded a priori by the control-affine structure.
		To this end, we have encoded the control-affine structure into the approximation ansatz itself, which allows to work with given state-control samples.
		In Proposition~\ref{Prop:limit_projection}, we have characterized the limit for~$M\to\infty$ as an orthogonal projection onto the span of the dictionary consisting of the state basis functions and their products with the control components. For finitely many samples, we have established a Marcinkiewicz--Zygmund inequality for the resulting ansatz space in Proposition~\ref{Prop:MZ}. Based on it, we have derived a probabilistic $L_2$-error bound with explicit constants in Theorem~\ref{Thm:err_coupled_MZ} and Corollary~\ref{Cor:err_est_fullanova_m}, which also controls the errors of the drift and of the input map separately.
		Finally, we have illustrated the approach by means of an academic test function and discretizations of a nonlinear advection-diffusion equation, the shallow water equations, and the viscous Burgers equation. In all examples, the dominant couplings are identified, and for polynomial dynamics the error reaches the accuracy of the least squares solver, provided that the relevant terms are not discarded by the sparsity threshold.


		\section*{Acknowledgments}
		The authors would like to express their sincere gratitude to their colleague
		Laura Weidensager for several fruitful discussions.

		\section*{Declaration on the use of AI tools}
		During the preparation of this work, the authors used Claude (Anthropic, model Claude Fable 5.1 and Opus~5.5) for the language editing, the discussion of proof ideas, proof reading and the documentation of code. The authors have reviewed, verified and revised all AI-assisted content and take full responsibility for the content of this publication.

        \section*{Data Availability Statement}
           The code associated with this article is available at
           \begin{center}
               \url{https://gitlab.hrz.tu-chemnitz.de/manuels-at-tu-chemnitz.de/anova-for-control},
           \end{center}see~\cite{KiPoScWo26code}.

		\bibliographystyle{abbrv-doi}
		\bibliography{literature}

		\appendix

		\section{Implementation via preprocessed ANOVA index sets\label{sec:complete_func}}
		In this part, we show that the coupled approach may be realized entirely within the existing ANOVA approximation software~\cite{ANOVAapprox,pyANOVAapprox}, that is, without implementing the coupled least squares problem~\eqref{eq:opt_fullanova} from scratch.
		To this end, we choose the ANOVA index sets~\eqref{eq:indexset_Nv} and~\eqref{eq:indexset_V} for the joint function~$h(\bz)=H_j(\bx,\bu)$ a priori such that they enforce the structure~\eqref{eq:approx_H_rowwise} and~\eqref{eq:approx_H_rowwise_m}, respectively. At the same time, this avoids the computation of ANOVA terms that vanish due to the control-affine structure.

		Accordingly, set
		\begin{align*}
			h(\bz) \coloneqq H_j(\bx,\bu), \quad \bz \coloneqq (\bx,\bu) \in \R^{d+m}.
		\end{align*}
		The naive way of applying the machinery of Section~\ref{sec:ANOVA} to the function~$h$ is to consider its full ANOVA decomposition~\eqref{eq:anova}, i.\,e.,
		\begin{align*}
			h(\bz) = \sum_{\bv\subseteq [d+m]} h_\bv(\bz_\bv).
		\end{align*}
		However, this decomposition consists of~$2^{d+m}$ ANOVA terms instead of the~$2^d+2^dm$ terms that would be necessary for the theoretical scenario of approximating~$F$ and (the columns of)~$G$ separately.
		This is inefficient because all terms coupling two or more components of~$\bu$ vanish a priori by the control-affine structure of~\eqref{eq:system_H}. Hence, computing these couplings between the~$u_\ell$ is unnecessary. Moreover, noisy sensitivity indices could incorrectly identify such terms as relevant.
        Further, despite the terms containing the control variable have to be \emph{linear} in~$u_\ell$, a generic ANOVA term may depend on~$u_\ell$ in an arbitrary nonlinear fashion.

		In the following, we address these two issues by means of a suitable a priori choice of the index sets~\eqref{eq:indexset_Nv} and~\eqref{eq:indexset_V}.

		We first address the number of ANOVA terms.
		To this end, recall that in the standard ANOVA approximation, the first step to reduce the number of ANOVA terms is to introduce a superposition dimension~$q\in\N$ and to consider only the terms of lower order, i.\,e., all terms with~$|\bv|\leq q$, cf.~\eqref{eq:anova_superposition}.
		This default is chosen since usually no a priori information about the structure of the given function is available.
		Here, however, we already know that the system~\eqref{eq:system_H} depends linearly on the individual components of~$\bu$, i.\,e., that there are no couplings between the~$u_\ell$, $\ell=1,\dots,m$.
		Therefore, we can encode this structural knowledge at the very beginning by providing a suitable index set~$V$ of relevant ANOVA terms, cf.~\eqref{eq:indexset_V}, and computing an approximation of the form
		\begin{align*}
			h(\bz) = \sum_{\bk\in\I(V)} c_\bk(h) \varphi_\bk(\bz)
		\end{align*}
		with~$\I(V)$ in~\eqref{eq:indexset_N_truncated}.
		Thereby, we avoid computing terms that vanish a priori.
		More precisely, with this choice of the index set, we obtain the~$2^d$ terms purely in~$\bx$ and, for each~$u_\ell$, its~$2^d$ couplings with~$\bx$.
		In other words, we work with~$2^d+2^dm$ ANOVA terms, i.\,e., the same number as in the theoretical scenario of approximating~$F$ and~$G$ separately.

		In practice, we again introduce a superposition dimension~$q$ to further reduce the number of terms.
		If the terms with and without a control component are to admit the same couplings purely in~$\bx$, we have to start with~$q+1$, since couplings with~$u_\ell$ are one order higher. We then remove couplings purely in~$\bx$ with~$|\bv|=q+1$ to obtain an ANOVA decomposition of the form
		\begin{align*}
			h(\bz) \approx \sum_{\substack{\bv\subseteq[d]\\ |\bv|\leq q}} h_\bv(\bx_\bv) + \sum_{\ell=1}^{m}\ \sum_{\substack{\bv\subseteq[d]\\ |\bv|\leq q}} h_{\bv\cup\{d+\ell\}}(\bx_\bv,u_\ell),
		\end{align*}
		i.\,e., terms purely in~$\bx$ up to order~$q$ as well as terms of order at most~$q+1$ containing exactly one control component.
		The terms involving two or more control components vanish by the control-affine structure. Consequently, the above is an equality whenever all index sets in~\eqref{eq:ANOVA_struct_F,G} satisfy~$|\bv|\leq q$. The converse does not hold in general, since the~$f^j_\bv$ in~\eqref{eq:ANOVA_struct_F,G} form a structural rather than an ANOVA decomposition. Therefore, an index set with~$|\bv|>q$ may still yield vanishing ANOVA terms of order larger than~$q$.

		Note, however, that we have not yet achieved the desired structure~\eqref{eq:approx_H_rowwise}, since linearity in the~$u_\ell$ is still missing. For general basis functions and index sets~\eqref{eq:indexset_Nv}, the terms~$h_{\bv\cup\{d+\ell\}}(\bx_\bv,u_\ell)$ may depend on~$u_\ell$ in an arbitrary nonlinear fashion. For instance, already the univariate term~$h_{\{d+\ell\}}(u_\ell)$ is in general not proportional to~$u_\ell$.

		A closer look at the definition of the index sets~\eqref{eq:indexset_Nv} reveals that the ANOVA approximation allows to specify the number of coefficients and basis functions for each individual ANOVA term.
		More precisely, by choosing~$N^\bv_j=2$ for a certain set of indices~$\bv$, the index set~$\I_{\bN^\bv}$ contains only the value~$1$ for the corresponding component. Hence, we use only one coefficient~$c_1$ from~$\{c_\bk\}$ and one basis function~$\varphi_1$ from~$\{\varphi_\bk\}$.

		In addition, note that for the Chebyshev basis~\eqref{eq:basis_cheb} we have~$\{\varphi^{\mathrm{cheb}}_{k}\}_{k\in\N_0} = \{1,\sqrt{2}x,\dots\}$.
		Thus, combining the above restriction of the index sets with the structure of the Chebyshev basis, we reconstruct the component~$u_\ell$ exactly by setting~$N^\bv_j=2$ for all index sets~$\bv$ containing one of the~$u_\ell$. The normalization constant of~$\varphi_1$ is then absorbed into the corresponding coefficients.
		Note that this restriction only concerns the control variables. In the $\bx$-directions, we are free to choose any orthonormal basis as introduced in Section~\ref{sec:ANOVA}, and we may even employ several different bases depending on the structure of~$F$ and~$G$.

		In summary, the strategy described above realizes the same approximation~\eqref{eq:approx_H_rowwise} as in Section~\ref{sec:linearity}, provided that the same basis systems~$\varphi^1=\varphi^2$ with possibly different index sets~$\I(V_1)$ and~$\I(V_2)$ are used.
		The approach of Section~\ref{sec:linearity} is thus slightly more general, since it admits different basis systems for~$F$ and~$G$, whereas the present strategy has the advantage that it may be carried out entirely within the existing ANOVA software~\cite{ANOVAapprox, pyANOVAapprox}.

		\begin{algorithm}{ANOVA approximation for control via preprocessed index sets}
			\label{alg:preprocessed}
			\vspace{-1.5em}

			\begin{tabbing}
				\textbf{Input:} \quad \= $d,m\in\N$ \hspace{10.5em} \= dimension and control dimension \\
				\> $q\in\N$, $q\ll d$ \> superposition dimension \\
				\> $\{(\bx^i,\bu^i)\}_{i=1}^{M}\subset\Omega\times\U$, $M\in\N$ \> i.i.d.\ sample points drawn according to the density~$\omega\cdot\widetilde\omega$ \\
				\> $\left(H(\bx^i,\bu^i)\right)_{i=1}^{M}$ \> function evaluations at the sample points \\
				\> $\varepsilon\in(0,1)$ \> threshold parameter \\
				\> $r>0$ \> confidence parameter, cf.\ Theorem~\ref{Thm:ANOVA_cheb}
			\end{tabbing}
			\begin{enumerate}
				\item Set~$V=\{\bv\subseteq[d]\colon|\bv|\leq q\}\cup\{\bv\cup\{d+\ell\}\colon \bv\subseteq[d],\,|\bv|\leq q,\,\ell=1,\dots,m\}$.
				\item Choose the bandwidths~$\bN^\bv$ in~\eqref{eq:indexset_Nv} with~$N^\bv_j=2$ for the component corresponding to~$u_\ell$ in every~$\bv\ni d+\ell$, using the Chebyshev basis in the control variables such that~$\varphi_1(u_\ell)\propto u_\ell$, cf.\ Remark~\ref{Remark:control_measure}.
				\item For~$j=1,\dots,d$: perform steps 2--5 of Algorithm~\ref{alg:anova} for~$h(\bz)=H_j(\bx,\bu)$ with the index set~$\I(V)$ from steps 1--2, where the truncation preserves the structure of~$V$ (at most one control component per term, bandwidth~$2$ in the control directions).
			\end{enumerate}
			\textbf{Output:} coefficients, where the terms with~$\bv\subseteq[d]$ form~$\tilde F_j$ and the terms with~$d+\ell\in\bv$ form~$\tilde G_{j\ell}$

		\end{algorithm}

		We briefly comment on the second step, which is more restrictive than it may seem at first sight.
		\begin{remark}[The control measure fixed by the preprocessing]
			\label{Remark:control_measure}
			Using the Chebyshev basis in the control variables fixes the control measure, since the ANOVA approximation of Section~\ref{sec:ANOVA} requires the basis to be orthonormal with respect to the sampling density. Hence, $\U_\ell=[-1,1]$ is equipped with the Chebyshev density~$\omega_{d+\ell}=\omega^{\mathrm{cheb}}$, for which~$\varphi_1(u_\ell)=\sqrt2\,u_\ell$.
			Step~2 of Algorithm~\ref{alg:preprocessed} is therefore a \emph{specialization} of the setting of Section~\ref{sec:linearity}, where~$\widetilde\omega$ was arbitrary. It yields~$\eta_\ell=0$ and~$\nu_\ell=\frac12$, that is, the zero-mean situation in which the limit systems of Remarks~\ref{Remark:schur} and~\ref{Remark:anova_separate} become block diagonal.
			Note that a different control measure remains admissible. However, the control basis then has to be orthonormalized accordingly, since the ANOVA partition in the~\mbox{$\bu$-direction} requires an orthonormal system in~$L_2(\U_\ell,\omega_{d+\ell})$ whose first non-constant element is proportional to~$u_\ell$. In particular, the control measure has to have zero mean.
		\end{remark}

		We close this appendix by illustrating the preprocessing strategy of Algorithm~\ref{alg:preprocessed} with two examples of increasing complexity.

		\begin{example}
			As a first example, we consider again the Euler approximation of a controlled Duffing oscillator~\eqref{eq:Duffing} from Example~\ref{ex:Duffing}.
			Here, we have~$d=2$ and~$m=1$, i.\,e., the coupled variable is~$\bz=(x_1,x_2,u)^\top\in\R^3$ with~$z_1=x_1$, $z_2=x_2$, and~$z_3=u$.
			Since the control is univariate, no unnecessary couplings between control components can occur. We use this simple example to illustrate the proposed choice of the index sets and of the basis.
			Additionally, as the dimension~$d+m=3$ is already small, we do not introduce a superposition dimension~$q<3$, but keep all possible ANOVA terms, i.\,e.,
			$V=\{\varnothing,\{1\},\{2\},\{3\},\{1,2\},\{1,3\},\{2,3\},\{1,2,3\}\}$.
			Following the preprocessing strategy described above, the bandwidths~\eqref{eq:indexset_Nv} are then chosen as~$N^\bv_j=2$ for the component corresponding to the control variable~$z_3=u$ and arbitrary even numbers otherwise, i.\,e.,
			with arbitrary~$a_n \in 2\N$, $n=1,\dots,8$, we set
			\begin{align*}
				\bN^{\{1\}} = a_1, \bN^{\{2\}} = a_2, \bN^{\{3\}} = 2,
				\bN^{\{1,2\}} = \begin{pmatrix} a_3 \\ a_4 \end{pmatrix},
				\bN^{\{1,3\}} = \begin{pmatrix} a_5 \\ 2 \end{pmatrix},
				\bN^{\{2,3\}} = \begin{pmatrix} a_6 \\ 2 \end{pmatrix},
				\bN^{\{1,2,3\}} = \begin{pmatrix} a_7 \\ a_8 \\ 2 \end{pmatrix}.
			\end{align*}
			The corresponding index sets~\eqref{eq:indexset_Nv} then read
			\begin{align*}
				\I_{\bN^{\{1\}}} &= \left\{ \bk=(k_1,0,0)^\top \colon k_1\in\left\{1,\dots,a_1-1\right\} \right\}, \\
				\I_{\bN^{\{2\}}} &= \left\{ \bk=(0,k_2,0)^\top \colon k_2\in\left\{1,\dots,a_2-1\right\} \right\}, \\
				\I_{\bN^{\{3\}}} &= \left\{ \bk=(0,0,k_3)^\top \colon k_3\in\left\{1\right\} \right\} = \left\{ (0,0,1)^\top \right\}, \\
				\I_{\bN^{\{1,2\}}} &= \left\{ \bk=(k_1,k_2,0)^\top \colon k_1\in\left\{1,\dots,a_3-1\right\}, k_2\in\left\{1,\dots,a_4-1\right\} \right\}, \\
				\I_{\bN^{\{1,3\}}} &= \left\{ \bk=(k_1,0,1)^\top \colon k_1\in\left\{1,\dots,a_5-1\right\} \right\}, \\
				\I_{\bN^{\{2,3\}}} &= \left\{ \bk=(0,k_2,1)^\top \colon k_2\in\left\{1,\dots,a_6-1\right\} \right\}, \\
				\I_{\bN^{\{1,2,3\}}} &= \left\{ \bk=(k_1,k_2,1)^\top \colon k_1\in\left\{1,\dots,a_7-1\right\}, k_2\in\left\{1,\dots,a_8-1\right\} \right\}.
			\end{align*}
			Denoting the resulting approximation of the component~$h(\bz)=H_j(\bx,u)$ by~$K$ and using the Chebyshev basis in the control variable, such that~$\varphi_1(u)\propto u$ (where the normalization constant is absorbed into the coefficients), the ansatz takes the form
			\begin{align*}
				K(\bz) &= c_{\boldsymbol 0}(K) + \sum_{\bk\in\I_{\bN^{\{1\}}}} c_\bk(K) \,\varphi_{k_1}(x_1) + \sum_{\bk\in\I_{\bN^{\{2\}}}} c_\bk(K) \,\varphi_{k_2}(x_2) + \sum_{\bk\in\I_{\bN^{\{1,2\}}}} c_\bk(K) \,\varphi_{k_1}(x_1) \,\varphi_{k_2}(x_2) \\
				&+ c_{(0,0,1)^\top}(K) \cdot u + \sum_{\bk\in\I_{\bN^{\{1,3\}}}} c_\bk(K) \,\varphi_{k_1}(x_1) \cdot u + \sum_{\bk\in\I_{\bN^{\{2,3\}}}} c_\bk(K) \,\varphi_{k_2}(x_2) \cdot u \\
				&+ \sum_{\bk\in\I_{\bN^{\{1,2,3\}}}} c_\bk(K) \,\varphi_{k_1}(x_1) \,\varphi_{k_2}(x_2) \cdot u.
			\end{align*}
			Defining the sets~$V_1\coloneqq\{\varnothing,\{1\},\{2\},\{1,2\}\}$ and~$V_2\coloneqq\{\{3\},\{1,3\},\{2,3\},\{1,2,3\}\}$, this can be written as
			\begin{align*}
				K(\bz) &= \sum_{\bk\in\I(V_1)} c_\bk(K) \,\varphi_\bk(\bx) + \sum_{\bk\in\I(V_2)} c_\bk(K) \,\varphi_\bk(\bx) \cdot u,
			\end{align*}
			which is the desired structure~\eqref{eq:approx_H_rowwise}.
		\end{example}

		The following example is more elaborate and shows that the preprocessing indeed reduces the number of ANOVA terms.
		\begin{example}
			As a second example, we consider a setting in which the naive application of the ANOVA decomposition would indeed introduce unnecessary couplings between the components of~$\bu$.
			While the previous example primarily illustrated how the choice of the index sets~$\bN^\bv$ in~\eqref{eq:indexset_Nv} affects the number of basis functions, this example illustrates how to choose the set~$V$ to avoid redundant ANOVA terms.
			To this end, we choose~$d=3$ and~$m=2$, i.\,e., we have~$\bz=(x_1,x_2,x_3,u_1,u_2)^\top$.

			Since the power set of~$[d+m]$ contains~$2^5=32$ elements, we admit only ANOVA terms of order at most~$2$ in~$\bx$.
			Thus, in the terminology of Algorithm~\ref{alg:preprocessed} we choose~$q=2$, which means that we include terms of total order at most~$q+1=3$ in~$\bz$. We then manually remove all couplings with more than two components of~$\bx$ as well as couplings with more than one component of~$\bu$, i.\,e., we utilize
			\begin{align*}
				V = \{&\varnothing,\{1\},\{2\},\{3\},\{4\},\{5\},\{1,2\},\{1,3\},\{2,3\},\{1,4\},\{2,4\},\{3,4\}, \\
				&\{1,5\},\{2,5\},\{3,5\}, \{1,2,4\},\{1,3,4\},\{2,3,4\},\{1,2,5\},\{1,3,5\},\{2,3,5\}\}.
			\end{align*}
			Thereby, we reduce the number of ANOVA terms from~$32$ to~$21$. Moreover, we choose the index sets~$\bN^\bv$ in~\eqref{eq:indexset_Nv} accordingly, that is, we restrict the number of basis functions for all sets containing a component of~$\bu$ by setting
			\begin{align*}
				2 &=
				\bN^{\{4\}} = \bN^{\{5\}} \\
				&= 
				N^{\{1,4\}}_2 = N^{\{2,4\}}_2 = N^{\{3,4\}}_2 = N^{\{1,5\}}_2 = N^{\{2,5\}}_2 = N^{\{3,5\}}_2 \\
				&= 
				N^{\{1,2,4\}}_3 = N^{\{1,3,4\}}_3 = N^{\{2,3,4\}}_3 = N^{\{1,2,5\}}_3 = N^{\{1,3,5\}}_3 = N^{\{2,3,5\}}_3.
			\end{align*}
			In analogy to the previous example, the resulting ansatz then reads
			\begin{align*}
				K(\bz)
				&=
				\left( c_{\boldsymbol 0}(K) + \sum_{\bk\in\I_{\bN^{\{1\}}}} c_\bk(K) \,\varphi_{k_1}(x_1) + \sum_{\bk\in\I_{\bN^{\{2\}}}} c_\bk(K) \,\varphi_{k_2}(x_2) + \sum_{\bk\in\I_{\bN^{\{3\}}}} c_\bk(K) \,\varphi_{k_3}(x_3) \right. \\
				&\qquad+ \sum_{\bk\in\I_{\bN^{\{1,2\}}}} c_\bk(K) \,\varphi_{k_1}(x_1) \,\varphi_{k_2}(x_2) + \sum_{\bk\in\I_{\bN^{\{1,3\}}}} c_\bk(K) \,\varphi_{k_1}(x_1) \,\varphi_{k_3}(x_3) \\
				&\qquad+ \left. \sum_{\bk\in\I_{\bN^{\{2,3\}}}} c_\bk(K) \,\varphi_{k_2}(x_2) \,\varphi_{k_3}(x_3) \right) \\
				&+ \left( c_{(0,0,0,1,0)^\top}(K) \cdot u_1 + \sum_{\bk\in\I_{\bN^{\{1,4\}}}} c_\bk(K) \,\varphi_{k_1}(x_1) \cdot u_1 + \sum_{\bk\in\I_{\bN^{\{2,4\}}}} c_\bk(K) \,\varphi_{k_2}(x_2) \cdot u_1 \right. \\
				&\qquad+ \sum_{\bk\in\I_{\bN^{\{3,4\}}}} c_\bk(K) \,\varphi_{k_3}(x_3) \cdot u_1 + \sum_{\bk\in\I_{\bN^{\{1,2,4\}}}} c_\bk(K) \,\varphi_{k_1}(x_1) \,\varphi_{k_2}(x_2) \cdot u_1 \\
				&\qquad+ \left. \sum_{\bk\in\I_{\bN^{\{1,3,4\}}}} c_\bk(K) \,\varphi_{k_1}(x_1) \,\varphi_{k_3}(x_3) \cdot u_1 + \sum_{\bk\in\I_{\bN^{\{2,3,4\}}}} c_\bk(K) \,\varphi_{k_2}(x_2) \,\varphi_{k_3}(x_3) \cdot u_1 \right) \\
				&+ \left( c_{(0,0,0,0,1)^\top}(K) \cdot u_2 + \sum_{\bk\in\I_{\bN^{\{1,5\}}}} c_\bk(K) \,\varphi_{k_1}(x_1) \cdot u_2 + \sum_{\bk\in\I_{\bN^{\{2,5\}}}} c_\bk(K) \,\varphi_{k_2}(x_2) \cdot u_2 \right. \\
				&\qquad+ \sum_{\bk\in\I_{\bN^{\{3,5\}}}} c_\bk(K) \,\varphi_{k_3}(x_3) \cdot u_2 + \sum_{\bk\in\I_{\bN^{\{1,2,5\}}}} c_\bk(K) \,\varphi_{k_1}(x_1) \,\varphi_{k_2}(x_2) \cdot u_2 \\
				&\qquad+ \left. \sum_{\bk\in\I_{\bN^{\{1,3,5\}}}} c_\bk(K) \,\varphi_{k_1}(x_1) \,\varphi_{k_3}(x_3) \cdot u_2 + \sum_{\bk\in\I_{\bN^{\{2,3,5\}}}} c_\bk(K) \,\varphi_{k_2}(x_2) \,\varphi_{k_3}(x_3) \cdot u_2 \right).
			\end{align*}
			Again, by defining the sets
			\begin{align*}
				V_1 &\coloneqq \{\varnothing,\{1\},\{2\},\{3\},\{1,2\},\{1,3\},\{2,3\}\}, \\
				V_2 &\coloneqq \{\{4\},\{1,4\},\{2,4\},\{3,4\},\{1,2,4\},\{1,3,4\},\{2,3,4\}\}, \\
				V_3 &\coloneqq \{\{5\},\{1,5\},\{2,5\},\{3,5\},\{1,2,5\},\{1,3,5\},\{2,3,5\}\},
			\end{align*}
			this can be written as
			\begin{align*}
				K(\bz) &= \sum_{\bk\in\I(V_1)} c_\bk(K) \,\varphi_\bk(\bx) + \sum_{\bk\in\I(V_2)} c_\bk(K) \,\varphi_\bk(\bx) \cdot u_1 + \sum_{\bk\in\I(V_3)} c_\bk(K) \,\varphi_\bk(\bx) \cdot u_2,
			\end{align*}
			i.\,e., we obtain the multivariate structure~\eqref{eq:approx_H_rowwise_m} without ever computing couplings between~$u_1$ and~$u_2$.
		\end{example}

	\end{document}